\documentclass{siamart}

\usepackage{amssymb}
\usepackage{cite}
\usepackage{graphicx}

\newcommand{\ctilde}{\raise.17ex\hbox{$\scriptstyle\sim$}}
\newcommand{\upcite}[1]{\textsuperscript{\textsuperscript{\cite{#1}}}}
\newcommand{\R}{\mathbb{R}}
\newcommand{\C}{\mathbb{C}}
\newcommand{\I}{\mathrm{i}\,}

\usepackage{ulem}

\newsiamremark{remark}{Remark}
\newsiamremark{example}{Example}

\newcommand{\rovlp}{[\hspace{-0.08em}\rangle}
\newcommand{\rovlpint}{[\hspace{-0.06em}\hspace{-0.14em}\raisebox{0.055em}{\hbox{\tiny$\bullet$}}\hspace{-0.25em}\rangle}
\newcommand{\lovlp}{\langle\hspace{-0.08em}]}
\newcommand{\lovlpint}{\langle\hspace{-0.14em}\hspace{-0.14em}\raisebox{0.055em}{\hbox{\tiny$\bullet$}}\hspace{-0.18em}]}
\newcommand{\all}{\langle\hspace{-0.07em}\rule[-0.21em]{0.2em}{0.04em}\hspace{-0.2em}\rule[0.5em]{0.2em}{0.04em}\hspace{-0.06em}\rangle}
\newcommand{\allint}{\langle\hspace{-0.07em}\rule[-0.21em]{0.2em}{0.04em}\hspace{-0.2em}\rule[0.5em]{0.2em}{0.04em}\hspace{-0.06em}\hspace{-0.22em}\raisebox{0.055em}{\hbox{\tiny$\bullet$}}\hspace{-0.16em}\rangle}

\newcommand{\thetaint}{[\hspace{-0.188em}\raisebox{0.06em}{\hbox{\tiny$\bullet$}}\hspace{-0.21em}]}

\newcommand{\Allint}{\langle\hspace{-0.07em}\rule[-0.21em]{0.2em}{0.04em}\hspace{-0.2em}\rule[0.5em]{0.2em}{0.04em}\hspace{-0.06em}\hspace{-0.15em}\raisebox{0.055em}{\hbox{\tiny$\bullet$}}\hspace{-0.1em}\rangle}

\newcommand{\Lovlpint}{\langle\hspace{-0.14em}\hspace{-0.04em}\raisebox{0.055em}{\hbox{\tiny$\bullet$}}\hspace{-0.14em}]}

\newcommand{\Rovlpint}{[\hspace{-0.14em}\raisebox{0.055em}{\hbox{\tiny$\bullet$}}\hspace{-0.04em}\hspace{-0.14em}\rangle}

\newcommand{\Thetaint}{[\hspace{-0.16em}\raisebox{0.06em}{\hbox{\tiny$\bullet$}}\hspace{-0.14em}]}

\title{A Class of Iterative Solvers for the Helmholtz Equation: Factorizations,
  Sweeping Preconditioners, Source Transfer, Single Layer Potentials, Polarized
  Traces, and Optimized Schwarz Methods}

\author{Martin J. Gander\thanks{Section of Mathematics, University of Geneva,
    1211 Geneva 4, Switzerland (\email{martin.gander@unige.ch}). He thanks the
    Swiss National Science Foundation SNF (200020\_168999/1).} \and Hui
  Zhang\thanks{Corresponding author. Key Laboratory of Oceanographic Big Data
    Mining \& Application of Zhejiang Province, Zhejiang Ocean University,
    Zhoushan 316022, China (\email{huiz@zjou.edu.cn}). He thanks Zhejiang
    Provincial Natural Science Foundation of China (LY17A010014), Natural
    Science Foundation of China (11671074, 11771313, 11771398) and research
    start funding of ZJOU.}}

\date{\today}

\begin{document}
\maketitle

\begin{abstract}
  Solving time-harmonic wave propagation problems by iterative methods
  is a difficult task, and over the last two decades, an important
  research effort has gone into developing preconditioners for the
  simplest representative of such wave propagation problems, the
  Helmholtz equation. A specific class of these new preconditioners
  are considered here. They were developed by researchers with various
  backgrounds using formulations and notations that are very
  different, and all are among the most promising preconditioners for
  the Helmholtz equation.

  The goal of the present manuscript is to show that this class of
  preconditioners are based on a common mathematical principle, and
  they can all be formulated in the context of domain decomposition
  methods called optimized Schwarz methods. This common formulation
  allows us to explain in detail how and why all these methods
  work. The domain decomposition formulation also allows us to avoid
  technicalities in the implementation description we give of these
  recent methods.

  The equivalence of these methods with optimized Schwarz methods
  translates at the discrete level into equivalence with approximate
  block LU decomposition preconditioners, and we give in each case the
  algebraic version, including a detailed description of the
  approximations used. While we chose to use the Helmholtz equation
  for which these methods were developed, our notation is completely
  general and the algorithms we give are written for an arbitrary
  second order elliptic operator. The algebraic versions are even more
  general, assuming only a connectivity pattern in the discretization
  matrix.

  All these new methods studied here are based on sequential decompositions
  of the problem in space into a sequence of subproblems, and they have in their
  optimal form the property to lead to nilpotent iterations, like an exact block
  LU factorization. Using our domain decomposition formulation, we finally
  present an algorithm for two dimensional decompositions, i.e. decompositions
  that contain cross points, which is still nilpotent in its optimal form. Its
  approximation is currently an active area of research, and it would have been
  difficult to discover such an algorithm without the domain decomposition
  framework.
  
\end{abstract}


%
%
\section{Introduction}

Solving the Helmholtz equation numerically for moderate to high
wavenumbers is a difficult task, and very different from solving
Laplace-like problems. This is for three main reasons: first,
Helmholtz problems are often posed on unbounded domains, which have to
be artificially truncated to perform computations on finite computers,
or by using Green's functions, see e.g. \cite{Nedelec, CGLS}. Second,
one needs to have a much higher mesh resolution than what would be
required to represent the solution because of the so called pollution
effect \cite{Ihlen, BabuSaut, Zhu}. And finally, one has then to solve
the resulting very large scale system of linear equations, for which
classical iterative methods are not suitable \cite{EG, Erlangga}.  Our
focus here is on a class of recently developed novel iterative solvers
for the Helmholtz equation based on sequential decompositions in
space. Many research groups around the world have focused on
developing such solvers, and for practitioners, and even specialists,
it is not easy to keep up with these developments, and to grasp
important novelties in these methods. The main reason for this is that
these new methods sometimes are formulated at the continuous level,
sometimes at the discrete level, sometimes using integral formulations
and sometimes volume discretizations, and the groups developing these
methods come from different backgrounds and use different motivations,
intuitions and notations when formulating their methods. Furthermore,
most of these new methods require absorbing boundary conditions or
perfectly matched layers for their formulation, which are ingredients
that are not commonly encountered for classical iterative methods for
Laplace-like problems.

The purpose of the present manuscript is to first describe in simple
terms the main underlying fundamental algorithms for the new class of
methods based on sequential decompositions in space. At the continuous
level, the underlying algorithms are optimal and optimized Schwarz
methods, and at the discrete level, the underlying algorithms are
exact and approximate block LU factorizations. This first, relatively
short part is giving the main insight needed to understand the new
Helmholtz solvers in a simplified and non-technical setting. In the
second, main part, we then rigorously show how this new class of
Helmholtz solvers are tightly related. To do so, we introduce a
compact notation that allows us to formulate all these new techniques,
and we give for each one first the original formulation given by the
authors, and then equivalent formulations at the continuous and
discrete level in the form of the fundamental underlying algorithms,
for which we can prove equivalence results. We hope that our
manuscript will help people working in this challenging area of
numerical analysis to rapidly understand this new class of
algorithms and their potential.

\section{Underlying Fundamental Algorithms}\label{SecUnderlying}

We start by explaining two fundamental algorithms which are very much related,
one at the discrete level and one at the continuous level.  These algorithms are
the key ingredient in all the recent iterative methods proposed for the
Helmholtz equation. Even though these algorithms can be formulated for other
partial differential equations and all our equivalence results still hold, we
use here first the Helmholtz equation in its simplest form to explain them,
namely
\begin{equation}\label{HelmholtzSimple}
  (\Delta+k^2) u = f\qquad \mbox{in $\Omega:=(0,1)\times(0,1)$},
\end{equation}
with suitable boundary conditions to make the problem well posed\footnote{We use
  this simplest form of the Helmholtz equation only here at the beginning, and
  treat in the main part the more complete formulation given in \Cref{eq:pde}.}.
Discretizing \Cref{HelmholtzSimple} using a standard five point finite
difference discretization for the Laplacian on an equidistant grid leads to the
linear system of equations
\begin{equation}\label{eq:trilP}
  \left[
    \begin{array}{ccccc}
      D_{1} & U_{1} &  &  &  \\ 
      L_{1} & D_{2} & U_{2} & & \\
      & \ddots & \ddots & \ddots & \\
      &  & L_{J-2} & D_{J-1} & U_{J-1}\\
      & & & L_{J-1} & D_J
    \end{array}
  \right]
  \left[
    \begin{array}{c}
      \mathfrak{u}_1 \\ \mathfrak{u}_{2} \\ \vdots \\ \mathfrak{u}_{J-1} \\
    \mathfrak{u}_{J}
    \end{array}
  \right]
  =
  \left[
    \begin{array}{c}
      \mathfrak{f}_1 \\ \mathfrak{f}_{2} \\ \vdots \\ \mathfrak{f}_{J-1} \\
      \mathfrak{f}_J
    \end{array}
  \right],
\end{equation}
where $D_j=\mathrm{tridiag}\,(\frac{1}{h^2},-\frac{4}{h^2}+k^2,\frac{1}{h^2})$
\footnote{We assume here homogeneous Dirichlet boundary conditions and well-posedness for simplicity at the beginning, see \Cref{sec:solvers} for more information.},
$L_j=U_j=\mathrm{diag}\,(\frac{1}{h^2})$. The block LU factorization of
the coefficient matrix in \Cref{eq:trilP} is given by
\begin{equation}\label{Factors}
  A=
  \arraycolsep0.35em
  \left[
    \begin{array}{ccccc}
      T_{1} &  &  &  &  \\
      L_{1} & T_{2} &  & & \\
      & \ddots & \ddots &  & \\
      &  & L_{J-2} & T_{J-1} & \\
      & & & L_{J-1} & T_J
    \end{array}
  \right]
  \left[
    \begin{array}{ccccc}
      I_{1} & T_1^{-1}U_{1} &  &  &  \\
      & I_{2} & T_2^{-1}U_{2} & & \\
      &  & \ddots & \ddots & \\
      &  & & I_{J-1} & T_{J-1}^{-1}U_{J-1}\\
      & & & & I_J
    \end{array}
  \right],
\end{equation}
where $T_j$'s are the Schur complements\footnote{We also assume here
for simplicity at the beginning that the $T_j$'s are
    invertible.} that satisfy the recurrence relation
\begin{equation}\label{eq:TjP}
  T_1 = D_1, \quad T_j = D_j - L_{j-1}T_{j-1}^{-1}U_{j-1} \mbox{ for }j\geq 2,
\end{equation}
as one can see by simply multiplying the two factors in
\Cref{Factors} and comparing with the original matrix in
\Cref{eq:trilP}.  Using this factorization, we can solve
\Cref{eq:trilP} by first solving by forward substitution the block lower
triangular system
\begin{equation}\label{ForwardSubstitution}
  \left[
    \begin{array}{ccccc}
      T_{1} &  &  &  &  \\
      L_{1} & T_{2} &  & & \\
      & \ddots & \ddots &  & \\
      &  & L_{J-2} & T_{J-1} & \\
      & & & L_{J-1} & T_J
    \end{array}
  \right]
  \left[
    \begin{array}{c}
      \mathfrak{v}_1 \\ \mathfrak{v}_{2} \\ \vdots \\ \mathfrak{v}_{J-1} \\ \mathfrak{v}_{J}
    \end{array}
  \right]
  =
  \left[
    \begin{array}{c}
      \mathfrak{f}_1 \\ \mathfrak{f}_{2} \\ \vdots \\ \mathfrak{f}_{J-1} \\
      \mathfrak{f}_J
    \end{array}
  \right],
\end{equation}
and then solving by backward substitution the block upper triangular system
\begin{equation}\label{BackwardSubstitution}
  \left[
    \begin{array}{ccccc}
      I_{1} & T_1^{-1}U_{1} &  &  &  \\
      & I_{2} & T_2^{-1}U_{2} & & \\
      &  & \ddots & \ddots & \\
      &  & & I_{J-1} & T_{J-1}^{-1}U_{J-1}\\
      & & & & I_J
    \end{array}
  \right]
    \left[
    \begin{array}{c}
      \mathfrak{u}_1 \\ \mathfrak{u}_{2} \\ \vdots \\ \mathfrak{u}_{J-1} \\
    \mathfrak{u}_{J}
    \end{array}
  \right]
  =
  \left[
    \begin{array}{c}
      \mathfrak{v}_1 \\ \mathfrak{v}_{2} \\ \vdots \\ \mathfrak{v}_{J-1} \\
      \mathfrak{v}_J
    \end{array}
  \right].
\end{equation}
This shows that one forward sweep (forward substitution) and one
backward sweep (backward substitution) are enough to solve the linear
system, and this is the fundamental underlying idea of the new
`sweeping algorithms' for the Helmholtz equation mentioned in the
title. This becomes a preconditioner, if the block LU factorization is
approximated by using approximate Schur complement matrices instead of
the exact ones. If we use the exact ones, then the iteration would
converge in one step and thus the iteration matrix is nilpotent of
degree (or index) one. One can however already see another one of
  the new algorithms here by taking a closer look at the forward
substitution in \Cref{ForwardSubstitution}: solving the first
equation, and substituting into the second one, and the result
obtained into the third one, and so on, we get
\begin{equation}\label{SourceTransferIdea}
  \begin{array}{rccccccl}
    \mathfrak{v}_1&=&T_1^{-1}\mathfrak{f}_1,\hfill&\\
    \mathfrak{v}_2&=&T_2^{-1}(\mathfrak{f}_2-L_{1}\mathfrak{v}_1)
             &=&T_2^{-1}(\mathfrak{f}_2-L_{1}T_1^{-1}\mathfrak{f}_1)
             &=:&T_2^{-1}\tilde{\mathfrak{f}}_2,\\
    \mathfrak{v}_3&=&T_3^{-1}(\mathfrak{f}_3-L_{2}\mathfrak{v}_2)
             &=&T_3^{-1}(\mathfrak{f}_3-L_{2}T_2^{-1}\tilde{\mathfrak{f}}_2)
             &=:&T_3^{-1}\tilde{\mathfrak{f}}_3,\\
             &\vdots& &\vdots& &\vdots& 
  \end{array}
\end{equation}
where we introduced new source terms
$\tilde{\mathfrak{f}}_2:=\mathfrak{f}_2-L_{1}T_1^{-1}\mathfrak{f}_1$,
$\tilde{\mathfrak{f}}_3:=\mathfrak{f}_3-L_{2}T_2^{-1}\tilde{\mathfrak{f}}_2$,\ldots
to make the solve for $\mathfrak{v}_2$, $\mathfrak{v}_3$, \ldots look
like the first solve for $\mathfrak{v}_1$.  These new source terms
contain a transferred source term from the previous line,
$$
  \tilde{\mathfrak{f}}_j:=\mathfrak{f}_j-L_{j-1}T_{j-1}^{-1}\tilde{\mathfrak{f}}_{j-1}, 
$$
which is the feature that led to the so called `source transfer'
methods mentioned in the title. Note that $\mathfrak{v}_J=\mathfrak{u}_J$,
so after the forward substitution, the last set of unknowns is already
the exact solution, a property that will be used later by some algorithms.

In the form we presented the block LU decomposition, the diagonal
blocks only contained one grid line of unknowns, but one could also
collect several grid lines into one block. This suggests to
look at the problem at the continuous level, where we decompose the
domain into subdomains, as illustrated in \Cref{DomDecFig}.
\begin{figure}
  \centering
  \includegraphics[width=\textwidth,trim=0 10 0 0,clip]{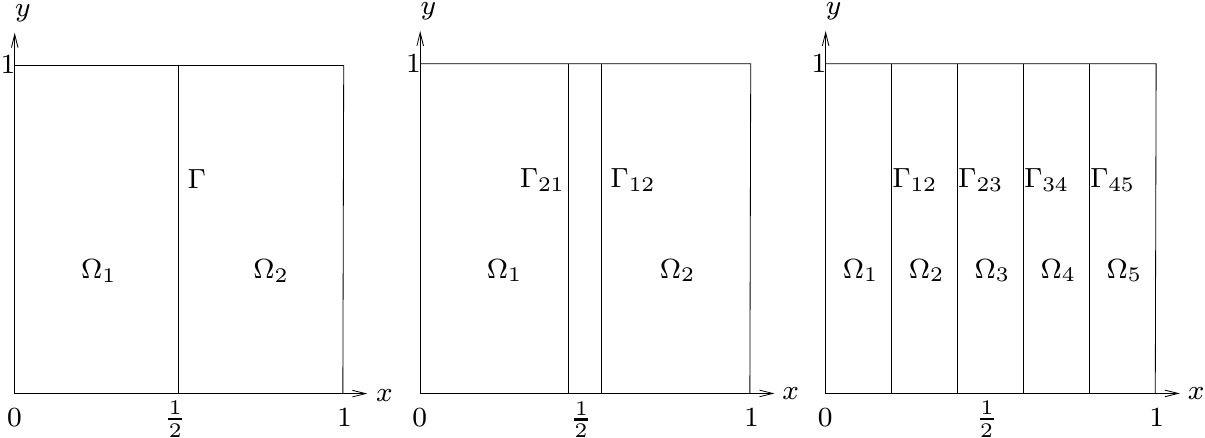}
  \caption{Domain decompositions of the unit square: without or with overlap,
    and many subdomains}
  \label{DomDecFig}
\end{figure}
For the first decomposition on the left, which would correspond to only two
large diagonal blocks at the discrete level, an optimal parallel Schwarz
algorithm is given by the iterative method (see \cite{NRS94, EZ98}) for
  arbitrary initial guess $u_1^0$, $u_2^0$,
\begin{equation}\label{OptSchwarzNoOvlp}
  \begin{array}{rcll}
  (\Delta+k^2)u_1^n & = & f \qquad &\mbox{in $\Omega_{1}$},\\
  \partial_{\mathbf{n}_1} u_1^n+\mbox{DtN}_1(u_1^n) & = & 
    \partial_{\mathbf{n}_1} u_2^{n-1}+\mbox{DtN}_1(u_2^{n-1}) & \mbox{on $\Gamma$},\\
  (\Delta+k^2)u_2^n & = & f \qquad &\mbox{in $\Omega_2$},\\
  \partial_{\mathbf{n}_2} u_2^n+\mbox{DtN}_2(u_2^n) & = & 
    \partial_{\mathbf{n}_2} u_1^{n-1}+\mbox{DtN}_2(u_1^{n-1}) & \mbox{on $\Gamma$},\\
  \end{array}
\end{equation}
where $\partial_{\mathbf{n}_j}$ denotes the outward normal derivative
for subdomain $\Omega_j$, and $\mbox{DtN}_j$ is the {\it
  Dirichlet-to-Neumann} operator taking Dirichlet interface data on
the interface $\Gamma$ and returning the outward normal derivative of
the corresponding solution of the Helmholtz equation on the exterior
of the subdomain $\Omega_j$. This algorithm converges in two
iterations, and thus the iteration operator is nilpotent of degree
two, as one can easily understand as follows: denoting by
$e_j^n:=u-u_j^n$ the error at iteration $n$, this error satisfies by
linearity the same equation as $u_j^n$, but with zero right hand side,
$f=0$. Since after the first iteration, the error $e_j^1$ satisfies
the equation in subdomain $\Omega_j$, its normal derivative at the
interface will exactly be canceled by the Dirichlet-to-Neumann
operator result when evaluating the right hand side on the interface
$\Gamma$ for the second iteration\footnote{The right hand side on the
  interface is in fact an exact or transparent boundary condition for
  the neighboring subdomain.}. The error $e_j^2$ then satisfies the
equation in $\Omega_j$ with homogeneous data and thus by uniqueness is
zero, the algorithm has converged.

The optimal parallel Schwarz algorithm in \Cref{OptSchwarzNoOvlp} can
also be run with overlap, as indicated in \Cref{DomDecFig} in
the middle, i.e.
\begin{equation}
  \begin{array}{rcll}
  (\Delta+k^2)u_1^n & = & f \qquad &\mbox{in $\Omega_{1}$},\\
  \partial_{\mathbf{n}_1} u_1^n+\mbox{DtN}_1(u_1^n) & = & 
    \partial_{\mathbf{n}_1} u_2^{n-1}+\mbox{DtN}_1(u_2^{n-1}) & \mbox{on $\Gamma_{12}$},\\
  (\Delta+k^2)u_2^n & = & f \qquad &\mbox{in $\Omega_2$},\\
  \partial_{\mathbf{n}_2} u_2^n+\mbox{DtN}_2(u_2^n) & = & 
    \partial_{\mathbf{n}_2} u_1^{n-1}+\mbox{DtN}_2(u_1^{n-1}) & \mbox{on $\Gamma_{21}$}.\\
  \end{array}
\end{equation}
The overlap has no influence on the two step convergence property of
the optimal parallel Schwarz method\footnote{This will be different if
  one uses approximations of the DtN operators, as we will see.}. With
$J$ subdomains, as indicated in \Cref{DomDecFig} on the right,
the corresponding optimal parallel Schwarz algorithm
\begin{equation}
  \begin{array}{rcll}
  (\Delta+k^2)u_j^n & = & f \qquad &\mbox{in $\Omega_{j}$},\\
  \partial_{\mathbf{n}_j} u_j^n+\mbox{DtN}_{j}(u_j^n) & = & 
    \partial_{\mathbf{n}_j} u_{j+1}^{n-1}+\mbox{DtN}_{j}(u_{j+1}^{n-1}) & \mbox{on $\Gamma_{j,j+1}$},\\
  \partial_{\mathbf{n}_j} u_j^n+\mbox{DtN}_{j}(u_j^n) & = & 
    \partial_{\mathbf{n}_j} u_{j-1}^{n-1}+\mbox{DtN}_{j}(u_{j-1}^{n-1}) & \mbox{on $\Gamma_{j,j-1}$},\\
  \end{array}
\end{equation}
converges in $J$ iterations \cite{NRS94}, and thus the iteration operator is
nilpotent of degree $J$. At the discrete level, this result was also presented
by F.-X. Roux at the IMACS conference in 2001. If we however organize the solves
in sweeps, starting on the leftmost subdomain and going to the rightmost
subdomain and back, i.e. we sweep once forward and once backward, the algorithm
converges in one such double sweep, independently of the number of subdomains
$J$, and thus the sweeping iteration operator becomes nilpotent of degree
one. This algorithm is in fact the continuous analog of the block LU
factorization then, with just a small modification that the DtN transmission
condition on the right is replaced by the Dirichlet condition, as we will see
later in \Cref{sec:lu}. Optimized Schwarz methods use approximations of the
Dirichlet-to-Neumann operator and thus the transparent boundary condition, in
the same spirit as approximate block LU factorizations use approximations to the
Schur complement matrices. Well known such approximations are {\it absorbing
  boundary conditions} (ABCs, c.f. \cite{EM1,EM2,HagRev1,HanWu}) and {\it
  perfectly matched layers} (PMLs, c.f. \cite{Berenger,Chew,HPR,DGK16}).

\section{The Helmholtz equation}
\label{sec:helm}

To be able to explain the details of recent iterative solvers for the
Helmholtz equation, we need to consider a slightly more general
Helmholtz equation than the simple model problem in
\Cref{HelmholtzSimple}, namely
\begin{equation}\label{eq:pde}
  \mathcal{L}\, u:=-\nabla^T(\alpha\nabla u)
    -\frac{\omega^2}{\kappa}\,u = f \mbox{\ in\ } \Omega,\quad
  \mathcal{B}\,u = g \mbox{\ on\ } \partial\Omega,
\end{equation}
where $\omega\in\C,$ $\Omega\subset\R^d$, $d=2,3$, is a bounded Lipschitz
domain, the coefficient matrix $\alpha$, the scalar field $\kappa$ and the
source $f$ are all given complex-valued quantities varying on $\Omega$, and the
unkown function $u$ on $\Omega$ is to be sought. In addition to the truly
physical part, the domain $\Omega$ may contain also artificial layers, for
example representing PMLs. The boundary condition $\mathcal{B}\,u=g$ is
piece-wise defined on
$\partial\Omega=\overline{\Gamma}_D\cup\overline{\Gamma}_R\cup\overline{\Gamma}_V$
as follows\footnote{For simplicity, we assume either $\Gamma_V=\emptyset$ or
    $\Gamma_R=\emptyset$ and that $\Gamma_V$ has no edges and
    corners.}:
\begin{align}
  &u=g_D,\quad \mbox{on }\Gamma_D,\label{eq:BD}\\
  &\mathbf{n}^T(\alpha\nabla u)+p_0u=g_R,
  \quad\mbox{on }\Gamma_R,\label{eq:BR}\\
  &\mathbf{n}^T(\alpha\nabla u) + q_0u + p_1\mathbf{n}^T(\alpha\nabla_Su)
  -\nabla_S^T\left(q_1\Pi_S(\alpha\mathbf{n})u+p_2\alpha\nabla_Su\right)=g_V,
  \quad\mbox{on }\Gamma_V,\label{eq:BV}
\end{align}%
where $\mathbf{n}$ is the unit outer normal vector, $\nabla_S$ is the surface
gradient, $p_0,$ $p_1, q_1, p_2$ are complex-valued functions, and $\Pi_{S}$ is
the orthogonal projection onto the tangential plane of the surface.  Besides
acoustic waves, the Helmholtz \Cref{eq:pde} is also used to describe
electromagnetics, where it can be derived from Maxwell's equations, see
\cite{Nedelec}.

\begin{example}
  In a typical geophysical application, $\Omega$ is a rectangle in
  $\R^2$ or a box in $\R^3,$ obtained from truncation of the free
  space problem. The original time domain equation in free space is
  given by
  \begin{displaymath}
    \frac{1}{c^2}\frac{\partial^2\tilde{u}}{\partial t^2} -
    \rho\nabla_{\mathbf{x}}^T\left(\frac{1}{\rho}\nabla_{\mathbf{x}}
      \tilde{u}\right) = \tilde{f},
  \end{displaymath}
  where $c$ is the wave speed and $\rho$ is the mass density, both dependent
  only on space, $\tilde{u}$ is the pressure, and $\tilde{f}$ is the space-time
  source term with compact support in $\Omega^{phy}\subset\Omega$ at all time.
  $\Omega^{phy}$ is the domain in which the solution is of interest.  We make
  the ansatz that $\tilde{f}(\mathbf{x},t)$ is a superposition of the
  time-harmonic sources
  $\rho(\mathbf{x})f(\mathbf{x},\omega)\mathrm{e}^{-\I\omega{t}}$.
  Then, for each $\omega$, the corresponding mode
  $u(\mathbf{x},\omega) \mathrm{e}^{-\I\omega{t}}$ satisfies
  \begin{displaymath}
  -\frac{\omega^2}{\rho c^2}u -
  \nabla_{\mathbf{x}}^T\left(\frac{1}{\rho}\nabla_{\mathbf{x}} u\right) = f.
  \end{displaymath}
  The time domain solution $\tilde{u}(\mathbf{x},t)$ is the sum of the
  time-harmonic modes $u(x,\omega)$ over all possible values of
  $\omega$.  Note that $\omega$ is the time {\it frequency},
  $k:=\frac{\omega}{c}$ is called {\it wavenumber}, and the {\it
    wavelength} is $\lambda=\frac{2\pi}{k}.$ A certain boundary
  condition along infinity is imposed to exclude energy incoming from
  infinity and to allow energy outgoing to infinity, viewed from
  $\Omega^{phy.}.$ An example is the Sommerfeld radiation condition
  (c.f. \cite[p. 189]{Sommer}) in a homogeneous medium,
  \begin{equation}\label{eq:sommer}
    \lim_{|\mathbf{x}|\rightarrow\infty}|\mathbf{x}|^{\frac{d-1}{2}}
    \left(\partial_{|\mathbf{x}|}{u} -\I{k}{u}\right)=0.
  \end{equation}
  Since we are interested in
  only the near-field solution (i.e. in $\Omega^{phy}$), the free
  space model is truncated to $\Omega^{phy}$ by imposing on
  $\partial\Omega^{phy}$ artificial boundary conditions or artificial
  layers, which leads to \Cref{eq:pde}.
\end{example}

\begin{example}
  Some models have physical boundaries from special materials, which leads
  directly to a boundary condition, e.g. Dirichlet for sound soft
  and Neumann for sound hard matter in acoustics.  As a simple model, one can consider a parallel
  pipe open in one dimension and closed with walls in the other dimensions on
  which Dirichlet, Neumann or artificial boundary conditions (layers) are
  imposed.  We further truncate the open dimension to obtain \Cref{eq:pde}.
  The truncated dimension is typically still much larger than the other
  dimensions such as for optical waveguides, see \cite{Lu06}.
\end{example}

\begin{example}
  An important class of models are the so-called scattering problems which are
  posed on an unbounded domain exterior to obstacles delimited by physical
  boundaries.  A given incident wave then hits the obstacles and gets scattered.
  The sum of the incident wave and the scattered wave gives the total wave field
  which satisfies homogeneous Dirichlet, Neumann or impedance boundary
  conditions as shown in \Cref{eq:BD,eq:BR} on the physical boundaries.  The
  scattered wave field satisfies homogeneous Helmholtz equation and some
  condition along infinity. Usually, this is the Sommerfeld radiation condition
  given in \Cref{eq:sommer} based on the assumption that the medium is
  homogeneous outside a bounded region.  The unbounded domain is truncated to a
  bounded region near the obstacles which results in \Cref{eq:pde}.  Once the
  Dirichlet and Neumann traces of the solution are known on some surface, the
  solution in the far-field, i.e. far away from the obstacles, can be recovered
  by using a representation formula, see \cite{Nedelec}.
\end{example}

\begin{remark}
  The algorithms to be discussed in this paper are applicable not only to the
  model \Cref{eq:pde} but also to more complicated cases as long as the partial
  differential equation and the boundary conditions are defined locally in space
  such that they make sense in subsets of $\Omega$ and $\partial\Omega$.  For
  instance, we can supplement \Cref{eq:BV} with edge and corner conditions, see
  \cite{Bamberg}, or use high-order absorbing boundary conditions localized
  with auxiliary unknowns, see \cite{Collino,HH}, which can also be viewed as
  semi-discretized PMLs, see \cite{DK99, GT00}, and solve other partial
  differential equations (see e.g. \cite{NN97}). This will become clearer in the
  following sections.
\end{remark}


We will occasionally need the weak formulation of \Cref{eq:pde} in appropriate
function spaces; see e.g. \cref{lem:TC}.  Multiplying both sides of
\Cref{eq:pde} with the complex conjugate of an arbitrary function $v$ and
integrating by parts in $\Omega$, we find formally
\begin{displaymath}
\int_{\Omega}(\alpha\nabla{u})^T\nabla{\bar{v}}-\frac{\omega^2}{\kappa}u\bar{v}
-\int_{\partial\Omega}\mathbf{n}^T(\alpha\nabla{u})\bar{v} =
\int_{\Omega}f\bar{v}.
\end{displaymath}
Substituting the boundary conditions from \Cref{eq:BD,eq:BR,eq:BV} into the
above equation leads us to the following weak formulation of \Cref{eq:pde}: find
$u-\mathcal{E}g_D\in V$, such that
\begin{equation}\label{eq:weak}
  \medmuskip=0.6mu
  \thinmuskip=0.6mu
  \thickmuskip=0.6mu
  \nulldelimiterspace=-0pt
  \scriptspace=0pt 
  \begin{array}{rl}
    &a(u,v) + b(u,v) = c(v),\mbox{ \ }\forall\,v\in V,\\
    &a(u,v):=\int_{\Omega}(\alpha\nabla{u})^T\nabla{\bar{v}}
    -\frac{\omega^2}{\kappa}u\bar{v},\\
    &b(u,v):= \int_{\Gamma_R\cup\Gamma_V} p_0u\bar{v} +
    \int_{\Gamma_V}p_1\mathbf{n}^T(\alpha\nabla_Su)\bar{v}+
    \int_{\Gamma_V}\left(q_1\Pi_S(\alpha\mathbf{n})u+
      p_2\alpha\nabla_Su\right)^T\nabla_S\bar{v},\\
    &c(v):= {}_{V'}\langle f,v\rangle_V +
    {}_{H^{-\frac{1}{2}}}\langle g_R,v|_{\Gamma_R}\rangle_{H^{\frac{1}{2}}} +
    {}_{H^{-\frac{1}{2}}}\langle g_V,v|_{\Gamma_V}\rangle_{H^{\frac{1}{2}}},
  \end{array}
\end{equation}
where $\mathcal{E}g_D\in V$ is an extension of $g_D,$ and
\begin{displaymath}
V=\{v\in H^1(\Omega):~v|_{\Gamma_D}=0,~v|_{\Gamma_V}\in H^1(\Gamma_V)\}, \
\|v\|_V=\sqrt{\|v\|_{H^1(\Omega)}^2+\|v\|_{H^1(\Gamma_V)}^2}.
\end{displaymath}
The well-posedness of \Cref{eq:weak} can be ensured by the following
simultaneous assumptions:
\begin{enumerate}
\item[(i)] $\alpha,\frac{\omega^2}{\kappa}$ are uniformly bounded in
  $\Omega,$ and so are $p_0,p_1,q_1,p_2$ in their domains;
\item[(ii)] there exist constants $C_1,C_2,C_3>0$ independent of $u$
  such that
  \begin{displaymath}
    \medmuskip=-1mu
    \thinmuskip=-1mu
    \thickmuskip=-1mu
    \nulldelimiterspace=-0pt
    \scriptspace=0pt 
    \left|\int_{\Omega}(\alpha\nabla{u})^T\nabla\bar{u}
    +\int_{\Gamma_V}p_2(\alpha\nabla_Su)^T\nabla_S\bar{u}
    +C_2\int_{\Omega}|u|^2+C_3\int_{\Gamma_V\cup\Gamma_R}|u|^2\right|
    \geq C_1\|u\|_V^2;
  \end{displaymath}
\item[(iii)] $a(u,u)+b(u,u)=0$ and $g_D\equiv0$ together imply $u\equiv0;$
\item[(iv)] $f\in V',$ $g_D\in H^{\frac{1}{2}}_{00}(\Gamma_D),$
  $g_R\in H^{-\frac{1}{2}}(\Gamma_R),$ $g_V\in
  H^{-\frac{1}{2}}(\Gamma_V).$
\end{enumerate}
We refer the reader to \cite{TWbook} for the definitions of the
function spaces and the well-posedness which is based on the
Lax-Milgram lemma and the Fredholm alternative.

\section{Direct and iterative solvers}
\label{sec:solvers}

After discretization of \Cref{eq:pde}, we obtain the linear system
\begin{equation}\label{eq:las}
  A\mathbf{u}=\mathbf{f},
\end{equation}
where $A$ is an $N$-by-$N$ matrix, $\mathbf{u}$ is the solution to be
sought and $\mathbf{f}$ is given.

Gaussian elimination is probably the oldest\footnote{Gaussian elimination can
  already be found in Chinese scripts \cite{LiuHui}} and the most fundamental
solver for linear algebraic systems like \Cref{eq:las}.  Its modern form
consists in first computing the LU factorization $A=RPLUQ$ with $R$ a row
scaling diagonal matrix, $L$ ($U$) being lower (upper) triangular and $P$ ($Q$)
being row (column) permutations, and then solving
$L\mathbf{v}=P(R^{-1}\mathbf{f})$ and $U(Q\mathbf{u})=\mathbf{v}$. For a dense
matrix $A$, e.g. from boundary element or spectral discretizations, the
factorization requires $\mathcal{O}(N^3)$ flops and $\mathcal{O}(N^2)$ storage,
and the triangular solves cost $\mathcal{O}(N^2)$ flops. For a sparse matrix
$A$, e.g. from a low-order finite element discretization, one can benefit from
the non-zero structure of the entries of $A$ by ordering the eliminations such
that as few non-zeros as possible are generated in $L$ and $U$, because the
number of non-zeros determines the storage and time complexities of the
triangular solves, see e.g. \cite{Davis, Duff}.  Doing so on 2-D regular meshes,
we need $\mathcal{O}(N^{\frac{3}{2}})$ flops for a sparse factorization, and the
$L, U$ factors one obtains have $\mathcal{O}(N\log{N})$ non-zeros, see
\cite{George, Hoffman}; on 3-D regular meshes, we need $\mathcal{O}(N^2)$ flops
and get $\mathcal{O}(N^{4/3})$ non-zeros in $L, U$, see \cite[p.143]{Davis}.
Significant progress has been made on reducing the actual constants hidden in
the asymptotic complexities, which is nowadays coded in widely-used software
packages like UMFPACK\upcite{DavisUMF}, PARDISO\upcite{Schenk},
SuperLU\upcite{LiSuperLU} and MUMPS\upcite{AMUMPS}.  The classical,
factorization based direct solvers aim at the exact solution and introduce no
error in the algorithms other than round-off errors due to finite precision
arithmetic.  They have been proved in practice to be robust for various
problems, and they are robust with respect to problem parameters. Moreover, they
are very advantageous for multiple {\it right hand sides} (r.h.s.), because the
factorization can be reused for different r.h.s. just performing triangular
solves, which is much faster than the factorization stage.

The drawbacks of direct solvers are the superlinear complexities
they have in time and storage requirements, and also the important
communication overhead in a parallel environment both in the
factorization and the triangular solution stages.

A recent trend of direct solvers is introducing low rank truncation of some
off-diagonal dense blocks arising in the process of factorization.  This is
accomplished by $\mathcal{H}$-matrix techniques \cite{Hack1, Hack2, Banja,
  Bebe1, Bebe2, Bebe3}, and related approaches \cite{MaRo, MaRo2, CHSS, Wang,
  Wang1, Wang2, Xia, HoGr, Gillman, WMUMPS2}.  The numerical low rank property
depends on the Green's function of the underlying {\it partial differential
  equation} (PDE).  In particular, for the Helmholtz equation, the numerical
rank can be shown to grow in specific geometric conditions in 2-D only
logarithmically with the wavenumber (see \cite{MaRo2, EY1}).  In general,
however, as indicated in \cite[p. 157]{Bebe2}, the growth seems to be linear in
the wavenumber.  In \cite{Hackbook}, the author says that there are two types of
  off-diagonal blocks, one type is small and can be treated easily by
  $\mathcal{H}$-matrix arithemetic, and the other type is large and should better be
  treated by a multipole expansion; see \cite{Banja} for more details.
  Recently, \cite{EZgreen} gave lower and upper bounds for the separability of
  Green's function.  Nearly linear complexities for 2-D Helmholtz problems have
been presented in e.g. \cite{MaRo, MaRo2, Banja, Wang, Xia, HoGr, Gillman}, with
fixed or increasing wavenumbers.  For 3-D Helmholtz problems, although
considerable improvements over classical direct solvers have been made using
$\mathcal{H}$-matrix techniques, the numerical experiments in \cite{Wang1,
  WMUMPS2} show that the time complexity of factorization tends to
$\mathcal{O}(N^{5/3})$\ctilde$\mathcal{O}(N^2)$ at high wavenumber on
proportionally refined meshes.

In contrast to direct solvers which deliver very accurate solutions in a finite
number of operations, iterative solvers start from an initial guess and improve
the accuracy successively by iteration. Iterative solvers have become a core
area of research in numerical analysis\footnote{Trefethen
  \cite{NickTrefethen}: ``The name of the new game is {\it iteration with
    preconditioning}. Increasingly often it is not optimal to try to solve a
  problem exactly in one pass; instead, solve it approximately, then iterate''.}.
The central issue of designing an iterative solver is finding an approximation
of $A^{-1}$ which is called {\it preconditioner}.  In this sense, the direct
solvers with low-rank truncation mentioned in the previous paragraph can be used
as preconditioners; see e.g. \cite{Banja, EY1, Xia}.  The simplest way to use a
preconditioner $M^{-1}\approx A^{-1}$ is iterative refinement, also called
Richardson iteration\footnote{The method Richardson proposed is much
  more sophisticated, including a relaxation parameter that changes with each
  iteration and is chosen to lead to an optimized polynomial
  \cite{richardson1911approximate}.} or deferred correction:
$\mathbf{u}\leftarrow \mathbf{u}+M^{-1}(\mathbf{f}-A\mathbf{u}).$ More advanced
are Krylov subspace methods; see \cite{Greenbaum, SaadBook} for general
introductions to iterative methods for linear systems.  It is particularly
difficult to design a fast iterative solver for the Helmholtz equation; see
e.g. the review papers \cite{Erlangga, AD, EG}.  The main challenge is to
accomplish $\mathcal{O}(N)$ time complexity for increasing frequency $\omega$ of
\Cref{eq:pde} on appropriately refined meshes; an easier goal is the linear
complexity under mesh refinement for fixed frequency, because this does not add
more propagating waves to the solution.  To tackle the special difficulties of
the Helmholtz equation, many techniques have been developed and integrated into
three major frameworks: incomplete factorizations, (algebraic) multigrid and
domain decomposition.  We will now briefly review some of these techniques.

The shifted-Laplace preconditioner $M_{\epsilon}^{-1}$, proposed in
\cite{Kim, EVO04}, introduces an imaginary shift $\epsilon$ to the
frequency $\omega$ or $\omega^2$ in the continuous problem given in
\Cref{eq:pde}, and $M_{\epsilon}$ is obtained from a discretization of the
shifted operator. The analyses in \cite{EG, GEV, GGS, CV13,
  GanderCocquet} exhibit altogether a gap between the requirements for
the shifted operator being close to the original operator and yet
being cheap to solve.
In practice, $\mathcal{O}(\omega^2)$ imaginary shifts to $\omega^2$
are often used.  In this case, it is easy to find an iterative solver
$\widetilde{M}_{\epsilon}^{-1}$ of $\mathcal{O}(N)$ complexity for
$M_{\epsilon}$ for any $\omega$; but $M_{\epsilon}$ deviates from the
original matrix $A$ more and more as $\omega$ increases so that the
iteration numbers for the original system with the preconditioner
$\widetilde{M}_{\epsilon}^{-1}$ also grow with $\omega$ and can be
$\mathcal{O}(\omega)$ for the truncated free space problem or even
$\mathcal{O}(\omega^2)$ in the case of a waveguide \cite{GanderCocquet}.  In
the former case, we observed from numerical experiments in
\cite{EOV06, BGS, CCFWY11, CPV13, GSV} for 2-D problems with
$N=\mathcal{O}(\omega^2)$ the overall time complexity becomes
$\mathcal{O}(N^{3/2})$, and in \cite{CPV13, RKEVOPM, CGPV} for 3-D
problems with $N=\mathcal{O}(\omega^3)$ it seems to be
$\mathcal{O}(N^{4/3})$.  Even though not optimal, for 3-D models and
$\mathcal{O}(1)$ r.h.s. these iterative solvers can be faster than a
complete factorization.

To accelerate convergence of iterations, an important idea is identifying the
slowly convergent components of the errors, and the corresponding residuals,
and projecting them out from the iterates by solving the original problem
restricted to the corresponding subspace.  This is called coarse correction or
deflation. When these slowly convergent coarse components are however based on a
grid discretization, then for the Helmholtz equation the coarse problem for
projection needs to be {\it fine} enough, typically of dimension
$\mathcal{O}(\omega^2)$ in 2-D and $\mathcal{O}(\omega^3)$ in 3-D, to keep the
convergence independent of $\omega$; see e.g. \cite{Cai92, LiTu09} for
convergence theory with such rich coarse problems.  This excessive requirement
manifests inadequacy of the basic iterations presumed in $\mathcal{O}(N)$ time
for oscillatory waves.  Of course, the more time we allocate to the basic
underlying iteration, the smaller the coarse problem will be that we need to
compensate for inadequacies of the basic iteration. Another approach is to try
to develop an efficient solver for the coarse problem. For the shifted-Laplace
preconditioner, multilevel Krylov with multigrid deflation was studied in
\cite{EN08, SLV13} and it was seen that the first coarse level (with mesh size
twice as coarse as the finest level) needs to be solved more and more accurately
to keep iteration numbers from growing as $\omega$ increases.  Another direction
is to seek more efficient coarse problems.  Those based on wave-ray or plane
waves that originated from \cite{Taasan, BL97} in the context of multigrid have
become popular in algebraic multigrid methods (see e.g. \cite{VMB, OS10}) and in
domain decomposition methods (see e.g. \cite{BFMMR, FATL, GZ1, HuLi}).  Some recent
developments include the bootstrap trick \cite{Liv11} to discover the slowly
convergent subspace, the local eigen-spaces \cite{CDKN} for heterogeneous media,
and improvement in stability of the coarse problems \cite{CLX14,
  Stolk2}. Complex-symmetric least squares formulations
\cite{OS10,Liv11,Gordon13}, numerical-asymptotic hybridization \cite{Popovic}
and block Krylov methods for multiple r.h.s. \cite{Lago} have also given
further insight in the search for a scalable Helmholtz solver.

When applying domain decomposition methods (see e.g. \cite{LeTallec,
  FarhatRoux, SBGbook, QVbook, TWbook, DNJbook}) to the Helmholtz
equation, an immediate obstacle is using the usual Dirichlet or
Neumann boundary conditions on subdomain interfaces.  In
particular, taking a diagonal block of $A$, one can not ensure that it
is nonsingular.  For example, in the unit square $(0,1)^2$ with
homogeneous Dirichlet boundary conditions, the negative Laplace
operator $-\Delta=-\partial_{xx}-\partial_{yy}$ has eigenvalues
$(n^2+m^2)\pi^2$ ($n,m=1,2,..$), so the Helmholtz operator
$-\Delta-k^2$ is singular if $k^2$ is equal to one of these eigenvalues.
This will not happen if a subdomain is sufficiently small, because the
minimal eigenvalue will then be larger than the given $k^2\in\R$: for
example, in the square $(0,H)^2$, the minimal eigenvalue is
$2\pi^2/H^2$, which can be made bigger than any given real number $k^2$
when $H$ is sufficiently small.  This strategy was adopted in
\cite{Cai92,FATL,LiTu09}, but the coarse problems become then very
large to maintain scalability with so many subdomains.  Another
natural treatment is regularization of subdomain problems with
absorbing boundary conditions (or layers), which helps verify
assumption (iii) in \Cref{sec:helm}.  For example, let $\Delta u
+k^2u=0$ in $\Omega$ ($k\in\mathbb{R}$) and
$\partial_{\mathbf{n}}u+pu=0$ on $\partial\Omega$, we have
$\int_{\Omega} |\nabla u|^2 -k^2|u|^2 +
\int_{\partial\Omega}p|u|^2=0$; if $\mathrm{Im}\,p\neq 0,$ we obtain
$\partial_{\mathbf{n}}u=u=0$ on $\partial\Omega$ which implies
$u\equiv0$ in $\Omega$ by the unique continuation property (see
\cite{Wolff}).  The well-posedness of high-order ABCs can be found in
e.g.\cite{TrefHalp}.
Regularization with a zeroth-order absorbing condition was first used by
B. Despr\'{e}s in his thesis \cite{Despres}, and later in e.g. \cite{BFMMR,
  Cai98, FMLRMB}.  Actually, one gets even more from this choice: compared to
the classical Schwarz method that uses Dirichlet transmission conditions, faster
convergence was observed.  This can be understood in the ideal case with {\it
  transparent} transmission conditions as we have seen in \Cref{SecUnderlying}
for the optimal Schwarz method \cite{NRS94, EZ98} motivated by the numerical
study in \cite{Hagstrom}.  Based on this principle, optimized Schwarz methods
(see e.g. \cite{GanderOSM} for an introduction, \cite{qin2006parallel,Loisel,gander2015analysis} for analyses and \cite{GX14,GX16} for geometry-dependent optimization)
leverage various approximations
i.e. absorbing transmission conditions or PMLs for fast convergence.  For
Helmholtz problems, the second-order Taylor expansion was used in
\cite{Douglas}, square-root based nonlocal conditions were studied in
\cite{Ghanemi, CGJ00}, best approximations of zero- and second-order were sought
in \cite{Chevalier, GMN02, GHM07, GZ2, GZolp, Stupfel}, Pad\'{e} approximants
with complex-shifted wavenumbers were used in \cite{Boubendir},
PMLs were first employed in \cite{Toselli, Schadle}, and
recently some rational interpolants were tested in \cite{KimZhang} for waveguide
problems.  For a numerical comparison of low-order and high-order
  transmission conditions for the overlapping Schwarz methods, we refer to
  \cite{GZolp}.

Parallel to the development of optimized Schwarz methods, absorbing transmission
conditions have also found use in the {\it analytic incomplete LU} (AILU)
preconditioner; see \cite{GanderAILU00, GanderAILU05}. The idea is based on
the identification of the DtN based transparent transmission condition with the
Schur complements arising in the block LU factorization we have seen in
\Cref{SecUnderlying}.  An important improvement to the AILU preconditioner has
been made by the independent development in \cite{EY2} using PML instead of the
second-order approximation used in \cite{GanderAILU05}.  This triggered more
studies on exploiting PML for the iterative solution of the Helmholtz equation
in a forward and backward sweeping fashion, see e.g. \cite{Chen13a, Chen13b,
  Stolk, Poulson, Vion, Zepeda, ZD, Kimsweep, EGdouble}.  A recursive
version that solves the 2-D sub-problems
  in a 3-D domain recursively by decomposing them into 1-D lines and sweeping
  can be found in \cite{LiuYingRecur}, see also 
  \cite{ZepedaNested} for a similar idea. A 
  recursive sweeping algorithm with low-order ABCs was already proposed earlier,  see
  \cite{achdou2003dimension}.  Another double sweep process that extends the
Dirichlet-Neumann alternating method \cite{QVbook} to many subdomains is
proposed in \cite{CLX16}.  In all these sweeping methods, 
there is little parallelism across
the block solves, since the blocks (or
subdomains) are passed through one by one, 
but in return an expensive coarse problem for connecting the
blocks is avoided.  Hence, the parallelism and the complexity within each block
become crucial.  In \cite{Poulson} for 3-D models, each block is kept quasi-2D
with fixed thickness, and a tailored parallel direct solver for the quasi-2D
problems is used.  The sequential
complexity was shown to be $\mathcal{O}(\delta^2 N^{4/3})$ for the
setup and $\mathcal{O}(\delta N\log N)$ for the solve, where
$\delta=\delta(k)$ is the thickness of the discrete PML on one side of
each block. Instead of PML, hierarchical matrix approximations can
also be used, see \cite{EY1, Bagci}.  More recently, in an effort to
parallelize the sweeping preconditioner, the authors of
\cite{LiuYingAdd} proposed to decompose the source term into
subdomains and then to simulate its influence on the other subdomains
by sweeping from that subdomain towards the first and the last
subdomain. The final approximation is then obtained by adding the
solutions corresponding to the different subdomain sources.  The
sweeping methods have also been combined with a two-level method in
\cite{StolkImproved}, and with the sparsifying preconditioner
\cite{ying2015sparsifying} for a volume integral reformulation of the
Helmholtz equation in \cite{ZepedaLippmann,LiuYingInt}.

The methods above based on approximation of transparent boundary
conditions are currently among the most promising iterative methods
for the Helmholtz equation and more general wave propagation
phenomena.  In the following sections, we will explain how these
methods were invented following various paths from very different
starting points, and give a formulation of each method in a common
notation that allows us to prove that each of these methods is in fact
a special optimized Schwarz method distinct only in transmission
conditions, overlaps, and/or implementation. A first such relation
between the source transfer method and an optimized Schwarz method was
discovered in the proceedings paper \cite{CGZ}, and futher relations
were pointed out in \cite{GZieee}.

\section{Notation}\label{sec:notation}

To make the analogy we have seen between the block LU factorization and the
optimal Schwarz algorithm mathematically rigorous, and then to show precisely
how all the new Helmholtz solvers are related to one another requires a common
notation that works for all formulations. This formulation must permit at the
same time the use of overlapping and non-overlapping blocks or subdomains,
Green's function formulations and volume discretizations, and very general
transmission conditions including absorbing boundary conditions and PML, and all
this both for continuous and discrete formulations. We introduce the reader to
this notation in this section, as we introduce the fundamental concepts common
to all algorithms step by step. The first steps learning the
notation will be hard, but it will be rewarding to be able to understand the
details of all these new Helmholtz solvers and their tight relation.

\subsection{Domain decomposition}
\label{sec:dd}

As we have seen, the algorithms are based on a decomposition, and we
introduce this decomposition for the original domain $\Omega$ on which
\Cref{eq:pde} is posed.  We decompose $\overline{\Omega}$ into
serially connected subdomains $\overline{\Omega}_j, j=1,\ldots,J$ such
that
\begin{equation}\label{eq:dd}
  \overline{\Omega}=\cup_{j=1}^J\overline{\Omega}_j,\quad
  \overline{\Omega}_j\cap\overline{\Omega}_l=\emptyset\ \mbox{if } |j-l|>1.
\end{equation}
To simplify the notation for the algorithms, we also introduce at each
end an empty subdomain, $\Omega_{0}=\Omega_{J+1}=\emptyset$.  We
denote the overlap between the subdomains by
$O_j:=\Omega_j\cap\Omega_{j+1}$, $j=1,..,J-1$, the interfaces by
$\Gamma_{j,j\pm1}:=\partial\Omega_j\cap\partial(\Omega_{j\pm1}-\Omega_j)$,
$j,j\pm1\in\{1,..,J\}$, and the non-overlapping region within each
subdomain by
$\Theta_j:=\overline{\Omega}_j-(\overline{\Omega}_{j-1}\cup\overline{\Omega}_{j+1})$,
$j=1,..,J$, as indicated in \Cref{fig:dd}. Note that this partition
can be considered either for the continuous domain or for the
discretized domain.
\begin{figure}
  \centering
  \includegraphics[scale=.5,trim=30 110 30 80,clip]{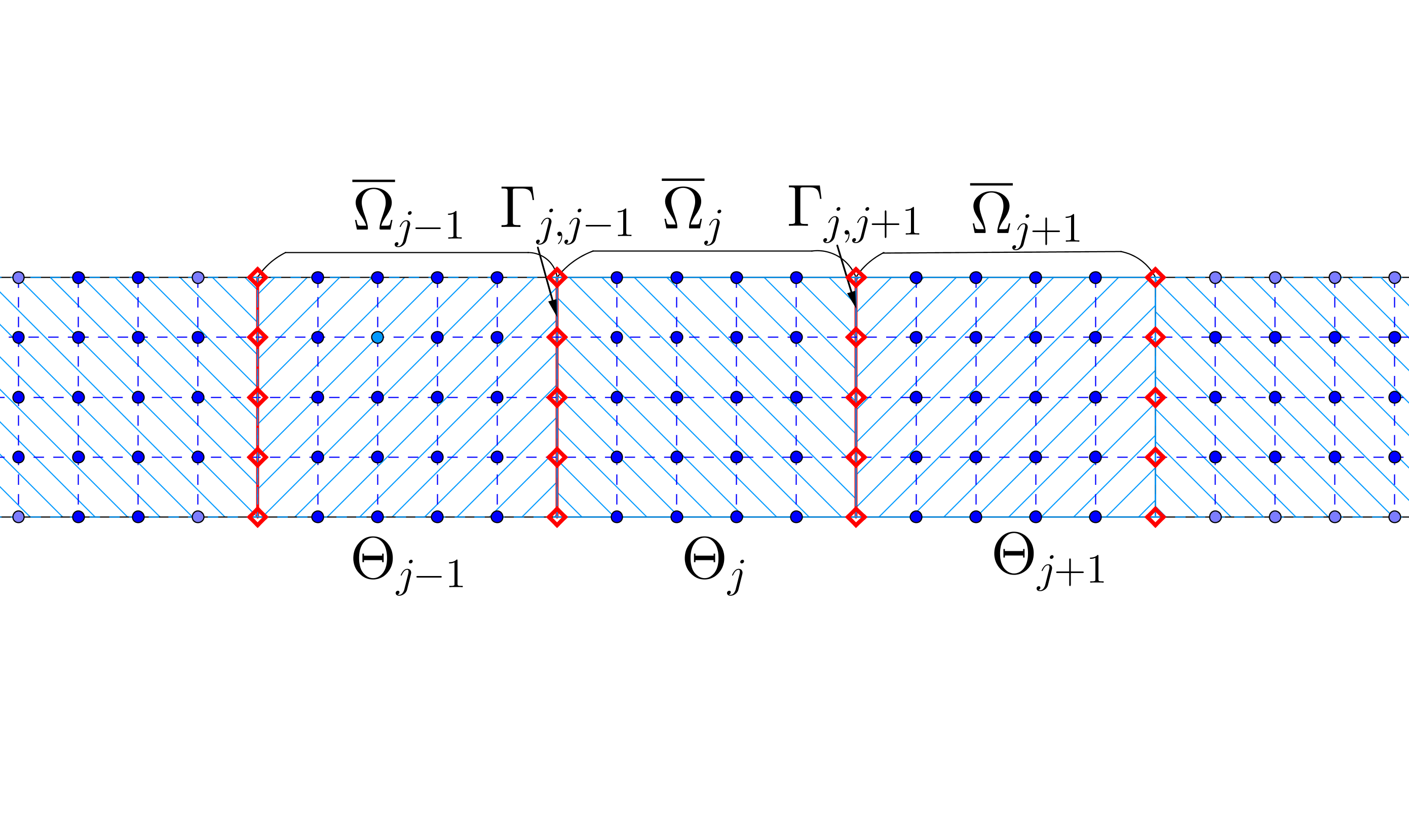}%
  \\
  \includegraphics[scale=.52,trim=30 90 30 80,clip]{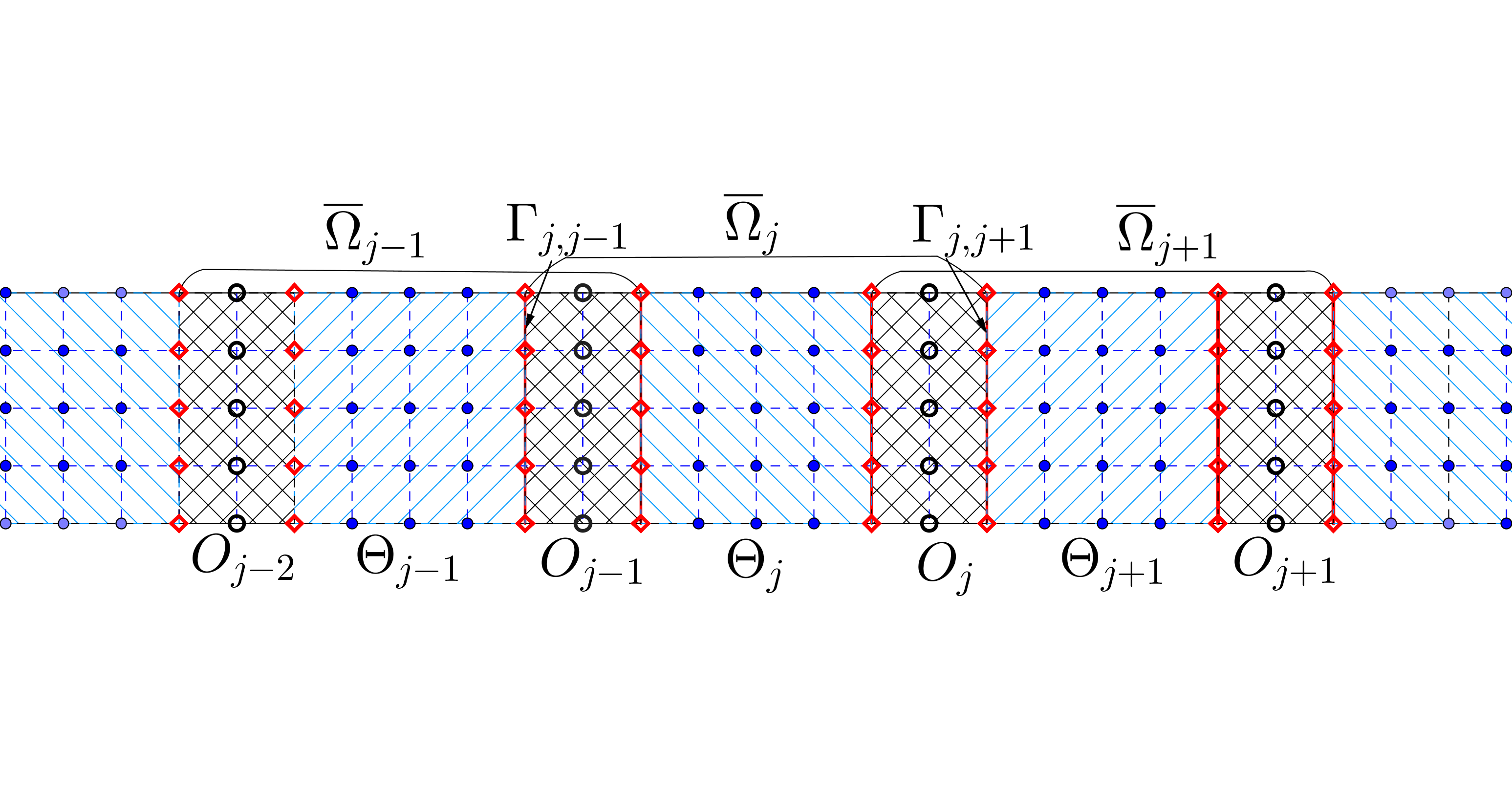}%
  \caption{ Non-overlapping and overlapping domain decomposition,
   ${\color{blue}\bullet}\in\Theta_*$, $\circ \in O_*$,
    ${\color{red}\Diamond}\in \Gamma_{*,\#}$.
   }
  \label{fig:dd}
\end{figure}
At the discrete level, the domain decomposition is a partition of the
d.o.f. $\mathbf{u}$ of \Cref{eq:las}. Corresponding to that partition,
we introduce the notion of index sets as shown in \Cref{tab:symb},
\begin{table}
  \centering
  \caption{Meaning of the subscripted vector $\mathbf{v}_*$}
  \label{tab:symb}
  \tabcolsep0.2em
  \begin{tabular}{|c|l|}
    \hline
    subscript $*$ & meaning of the vector $\mathbf{v}_*$\\
    \hline
    $\Large{j\all}$ or $j$ & $\overline{\Omega}_j$: all the d.o.f. on the $j$-th subdomain $\overline{\Omega}_j$ \\
    \hline
    $\Large{j\langle}$ & left interface of $\Omega_j$: restriction of $\mathbf{v}_j$ to the interface $\Gamma_{j,j-1}$ \\
    \hline
    $\Large{j\rangle}$ & right interface of $\Omega_j$: restriction of $\mathbf{v}_j$ to the interface $\Gamma_{j,j+1}$ \\
    \hline
    $\Large{j\lovlp}$ & left overlap of $\Omega_j$ : restriction of $\mathbf{v}_j$ to the overlap $\overline{\Omega}_{j}\cap\overline{\Omega}_{j-1}$ \\
    \hline
    $\Large{j\rovlp}$ & right overlap of $\Omega_j$: restriction of $\mathbf{v}_j$ to the overlap $\overline{\Omega}_{j}\cap\overline{\Omega}_{j+1}$\\
    \hline
    $\Large j]$ & end of left overlap in $\Omega_j$: restriction of $\mathbf{v}_j$ to $\Gamma_{j-1,j}$ \\
    \hline
    $\Large j[$ & end of right overlap in $\Omega_j$: restriction of $\mathbf{v}_j$ to $\Gamma_{j+1,j}$ \\
    \hline
    $\Large{j\Allint}$ or $j\bullet$ & interior unknowns of $\Omega_j$: after removing $\mathbf{v}_{j\langle}$ and $\mathbf{v}_{j\rangle}$ from $\mathbf{v}_{j\all}$ \\   
    \hline
    $\Large{j\Lovlpint}$ & interior of left overlap of $\Omega_j$: after removing $\mathbf{v}_{j\langle}$ and $\mathbf{v}_{j]}$ from $\mathbf{v}_{j\lovlp}$\\
    \hline
    $\Large{j\Rovlpint}$ & interior of right overlap of $\Omega_j$: after removing $\mathbf{v}_{j[}$ and $\mathbf{v}_{j\rangle}$ from $\mathbf{v}_{j\rovlp}$ \\
    \hline
    $\Large{j\Thetaint}$ & interior without overlaps  of $\Omega_j$: after removing $\mathbf{v}_{j\lovlp}$ and $\mathbf{v}_{j\rovlp}$ from $\mathbf{v}_{j\all}$\\
    \hline
    $\ctilde j$ & everything to the left of $\Omega_j$, i.e. all the d.o.f. in $\Large(\cup_{m=1}^{j-1}\overline{\Omega}_m)
                  -\overline{\Omega}_j$\\
    \hline
    $j\ctilde$ & everything to the right of $\Omega_j$, i.e. all the d.o.f. in $\Large(\cup_{m=j+1}^{J}\overline{\Omega}_m)
                 -\overline{\Omega}_j$\\
    \hline
  \end{tabular}
\end{table}  
which we chose to simplify remembering them: for example, in the non-overlapping
case, we have for subdomain $\Omega_j$ on the left boundary
$\mathbf{u}_{j\langle}=\mathbf{u}_{j]}=\mathbf{u}_{j\lovlp}$, and similarly on
the right boundary $\mathbf{u}_{j\rangle}=\mathbf{u}_{j[}=\mathbf{u}_{j\rovlp}$,
and the unknowns interior to subdomain $\Omega_j$ are
$\mathbf{u}_{j\allint}=\mathbf{u}_{j\thetaint}$, where a dot always means only
the interior unknowns, and we have equality because the overlap is empty. As
another example, if $\Theta_j\neq\emptyset$ in the overlapping case, we have
$\mathbf{u}_{j\allint}=\left[\mathbf{u}_{j\lovlpint}; \mathbf{u}_{j]};
  \mathbf{u}_{j\thetaint}; \mathbf{u}_{j[};
  \mathbf{u}_{j\rovlpint}\right]$\footnote{we use Matlab notation for
  concatenating column vectors vertically to avoid having to use the transpose symbol $^T$.}.
We denote by $I^{\#}_{*}$ the 0-1 restriction or extension matrix from the set
of d.o.f. indicated by $\#$ to the set indicated by $*$. For example,
$I_{j\langle}^{j-1}$ is the restriction from $\overline{\Omega}_{j-1}$ to
$\Gamma_{j,j-1}$.  For the identity matrix $I_*^*$ we simply write $I_*.$ A
common restriction operator we will often use is $R_*:=I^{1,..,J}_*$.  Note that
by $\mathbf{v}_l$ and $\mathbf{v}_j$ we do not mean in general the restriction
$\mathbf{v}_l=R_l\mathbf{v}$ and $\mathbf{v}_j=R_j\mathbf{v}$ of a common
$\mathbf{v}$, but that the components of $\mathbf{v}_j$ correspond to the
d.o.f. on $\overline{\Omega}_j$. Hence $I^j_{j\rovlp}\mathbf{v}_j$ can be
different from $I^{j+1}_{j+1\lovlp}\mathbf{v}_{j+1}$.  Only for the exact
solution $\mathbf{u}$ and the r.h.s. $\mathbf{f}$ of \Cref{eq:las}, we use
$\mathbf{u}_j$ and $\mathbf{f}_j$ for $R_j\mathbf{u}$ and $R_j\mathbf{f}$. For
submatrices of $A$, we also use the subscripts $j[$ interchangeably
with $j+1\langle$ because they refer to the same index set in
the global case, and similarly for $j]$ and $j-1\rangle$.  

If $\Theta_j\neq\emptyset$ for $j=1,..,J$, we assume that
\Cref{eq:las} has block tridiagonal structure\footnote{This holds
  naturally for classical finite difference and finite element
  discretizations which approximate derivatives by only looking at
  neighboring nodes.} when the vector of unknowns is partitioned as
$\mathbf{u}=[\mathbf{u}_{1\thetaint}; \mathbf{u}_{1\rovlp}; \ldots;
  \mathbf{u}_{J-1\rovlp}; \mathbf{u}_{J\thetaint}]$, i.e.
\begin{equation}\label{eq:tri}
  \left[
    \begin{array}{ccccc}
      A_{1\thetaint} & A_{1\thetaint1\rovlp} &  &  &  \\
      A_{1\rovlp1\thetaint} & A_{1\rovlp} & A_{1\rovlp2\thetaint} &  & \\
                     & \ddots & \ddots & \ddots & \\
                     & & A_{J-1\rovlp J-1\thetaint} & A_{J-1\rovlp} & A_{J-1\rovlp J\thetaint}\\
                     & & & A_{J\thetaint J-1\rovlp} & A_{J\thetaint}
    \end{array}
  \right]
  \left[
    \begin{array}{c}
      \mathbf{u}_{1\thetaint}\\ \mathbf{u}_{1\rovlp} \\
      \vdots \\ \mathbf{u}_{J-1\rovlp} \\ \mathbf{u}_{J\thetaint}
    \end{array}
  \right]
  = \mathbf{f}.
\end{equation}
For overlapping decompositions (i.e. $O_j\neq \emptyset$), we can also
partition the overlap, $\mathbf{u}_{j\rovlp}=\left[\mathbf{u}_{j[};
    \mathbf{u}_{j\rovlpint};\mathbf{u}_{j\rangle}\right]$, and similar
  to the block tridiagonal assumption, we assume that there are no
  direct interactions between d.o.f. across the interfaces, e.g.
\begin{displaymath}
  A_{j\rovlpint j\thetaint}=0,\quad A_{j\rangle j\thetaint}=0,\quad A_{j[ j+1\thetaint}=0,\quad
  A_{j\rovlpint j+1\thetaint}=0.
\end{displaymath}
In the non-overlapping case we also use $j\rangle$ for $j\rovlp$, since
the index sets $\rangle$ and $\rovlp$ coincide in that case, and the same
holds for $\langle$ and $\lovlp$.

\begin{remark}\label{rem:ddst}
  When $\Theta_1, \Theta_J\neq \emptyset,$ $\Theta_j=\emptyset$ for $j=2,..,J-1$
  and $O_j\neq\emptyset$ for $j=1,..,J-1$, we assume \Cref{eq:las} is block
  tridiagonal under the partition
  \begin{displaymath}
    \mathbf{u}=[\mathbf{u}_{1\thetaint}; \mathbf{u}_{1[};
    \mathbf{u}_{1\rovlpint}; \mathbf{u}_{1\rangle};
    \mathbf{u}_{2\rovlpint}; \mathbf{u}_{2\rangle};\ldots; \mathbf{u}_{J-1\rangle};
    \mathbf{u}_{J\thetaint}].
  \end{displaymath}
  Then, all our subsequent discussions on \Cref{eq:tri} can be adapted to this case
  without essential difference.  This case corresponds to what is called {\it generous overlap} in domain
  decomposition, i.e. almost every subdomain consists of overlaps
  with neighbors: $\Omega_j=O_{j-1}\cup\Gamma_{j-1,j}\cup O_j$ and
  $\Gamma_{j+1,j}=\Gamma_{j-1,j}$ for $j=2,..,J-1.$ For convenience,
  we will also denote by $O_0:=\Theta_1$ and $O_J:=\Theta_J$ for later use
  in \Cref{sec:stddm}.
\end{remark}

\begin{remark}
  The algorithms we will discuss below also permit $A_{j\rovlp
    l\rovlp}\neq 0$ for $|j-l|=1$, which can happen for example for
  spectral element discretizations with only one spectral element over the
  entire subdomain width. We exclude however this situation for
  simplicity in what follows.
\end{remark}

\begin{remark}
  The block tridiagonal form we assumed in \Cref{eq:tri} is natural if the
  d.o.f. in the overlap including the interfaces,
  $O_j\cup\Gamma_{j+1,j}\cup\Gamma_{j,j+1}$, are the same and shared by
  $\Omega_j$ and $\Omega_{j+1}$, i.e. the problem stems from a globally
  assembled problem. In domain decomposition, non-matching grids are however
  also common, and we may have two sets of d.o.f. in
  $O_j\cup\Gamma_{j+1,j}\cup\Gamma_{j,j+1}$, one set $\mathbf{u}_{j\rovlp}$ for
  $\Omega_j$ and another set $\mathbf{u}_{j+1\lovlp}$ for $\Omega_{j+1}$. In
  this case, when $O_j\neq\emptyset$, we may assume \Cref{eq:las} has the form
  of the {\it augmented system}
  \begin{displaymath}
    \left[\begin{array}{ccccccc}
        A_{1\thetaint} & A_{1\thetaint1\rovlp} &  &  &  & \\
        A_{1\rovlp1\thetaint} & A_{1\rovlp} & A_{1\rovlp2\lovlp} & \fbox{$A_{1\rovlp2\thetaint}$} & & \\
        \fbox{$A_{2\lovlp1\thetaint}$} & A_{2\lovlp1\rovlp} & A_{2\lovlp} & A_{2\lovlp2\thetaint} & &\\
        & & \ddots & \ddots & \ddots & \fbox{$A_{J-1\rovlp J\thetaint}$}\\
        & & \fbox{$A_{J\lovlp J-1\thetaint}$} & A_{J\lovlp J-1\rovlp} & A_{J\lovlp} & A_{J\lovlp J\thetaint}\\
        & & & & A_{J\thetaint J\lovlp} & A_{J\thetaint}
      \end{array}\right]
    \left[
      \begin{array}{c}
        \mathbf{u}_{1\thetaint}\vspace{.28em}\\ \mathbf{u}_{1\rovlp} \vspace{.28em}\\
        \mathbf{u}_{2\lovlp}\vspace{.28em}\\\vdots \vspace{.28em}\\
        \mathbf{u}_{J\lovlp} \vspace{.28em}\\ \mathbf{u}_{J\thetaint}
      \end{array}
    \right]
    = \mathbf{f},
  \end{displaymath}
  which would be block tridiagonal if we removed the boxed blocks.
  This form also arises naturally from non-conforming discretizations
  in the overlaps and on the interfaces, e.g. from certain
  discontinuous Galerkin methods (see e.g. \cite{Arnold}).  It is
  possible to generalize our discussions to this case, but we would
  not gain more insight and will thus not do so here to avoid
  further complications in the notation.
\end{remark}

\subsection{Transmission conditions}

We now present the three fundamental types of transmission conditions
used by the algorithms: Dirichlet, Neumann and generalized Robin
conditions. For the Dirichlet condition, based on the assumptions in
\Cref{sec:dd}, if we take the rows for $\mathbf{u}_{j\bullet}$ from
\Cref{eq:las}, we will find
$A_{j\bullet}\mathbf{u}_{j\bullet}=\mathbf{f}_{j\bullet} - A_{j\bullet
  j\langle}\mathbf{u}_{j\langle} - A_{j\bullet
  j\rangle}\mathbf{u}_{j\rangle}.  $ We rewrite this as a linear
system for $\mathbf{u}_j$ with the interface data
$\mathbf{u}_{j\langle}$ and $\mathbf{u}_{j\rangle}$ provided by the
neighboring subdomains,
\begin{equation}\label{eq:matdt}
\left[
  \begin{array}{ccc}
    I_{j\langle} & &\\
    A_{j\bullet j\langle} & A_{j\bullet} & A_{j\bullet j\rangle}\\
    & & I_{j\rangle}
  \end{array}
\right]
\left[
  \begin{array}{c}
    \mathbf{u}_{j\langle} \\ \mathbf{u}_{j\bullet} \\ \mathbf{u}_{j\rangle}
  \end{array}
\right] =
\left[
  \begin{array}{c}
    \mathbf{u}_{j-1[} \\
    \mathbf{f}_{j\bullet} \\ \mathbf{u}_{j+1]}
  \end{array}
\right].
\end{equation}
At the continuous level, this corresponds to \Cref{eq:pde} localized to
$\Omega_j$ with Dirichlet transmission conditions on the interfaces,
\begin{equation}\label{subdomainproblem}
\begin{array}{r@{\hspace{0.2em}}c@{\hspace{0.2em}}ll}
  \mathcal{L}\,u_j&=&f & \mbox{ in }\Omega_j,\\
  \mathcal{B}\,u_j&=&g & \mbox{ on }\partial\Omega_j\cap\partial\Omega,\\
  u_j&=&u_{j-1} & \mbox{ on }\Gamma_{j,j-1},\\
  u_j&=&u_{j+1} & \mbox{ on }\Gamma_{j,j+1},
\end{array}
\end{equation}
where $u_l:=u|_{\Omega_l}$ ($l=j,j\pm1$).  As mentioned before, the
subdomain problem in \Cref{subdomainproblem} might not be well-posed
if we replace $u_j$ by $v_j$ and then try to solve
\Cref{subdomainproblem} for $v_j$. Similarly, the matrix in
\Cref{eq:matdt} might not be invertible.

For the Neumann condition, we split the diagonal block of $A$
corresponding to the interface $\Gamma_{j,j-1}$ into two parts as it
would arise naturally from the splitting of the bilinear forms in
\Cref{eq:weak} in a conforming finite element method:
$A_{j\langle}=A_{j\langle}^{\rangle}+A_{j\langle}^{\langle}$, where
$A_{j\langle}^{\rangle}$ is the contribution from the left side of
$\Gamma_{j,j-1}$ and $A_{j\langle}^{\langle}$ is the contribution from
the right side of $\Gamma_{j,j-1}$. The reader might wonder why we use
the superscript $\rangle$ here (\verb!\rangle! !) for the
contribution from the left, and the superscript $\langle$ for the
contribution from the right.  The reason is that the contribution from
the right will be used on the left of the corresponding subdomain that
uses it, and vice versa, and based on the assumptions in
\Cref{sec:dd}, the rows from \Cref{eq:las} for $\mathbf{u}_j$ can then
be written similar to the Dirichlet case in the form
\begin{equation}\label{eq:matnt}
\left[
  \begin{array}{ccc}
    A_{j\langle}^{\langle} & A_{j\langle j\bullet} &\\
    A_{j\bullet j\langle} & A_{j\bullet} & A_{j\bullet j\rangle}\\
    & A_{j\rangle j\bullet} & A_{j\rangle}^{\rangle}
  \end{array}
\right]
\left[
  \begin{array}{c}
    \mathbf{u}_{j\langle}\vphantom{_{j\langle}^{\langle}} \\ \mathbf{u}_{j\bullet} \\ \mathbf{u}_{j\rangle}\vphantom{_{j\langle}^{\langle}}
  \end{array}
\right] =
\left[
  \begin{array}{c}
    \mathbf{f}_{j\langle}-A_{j\langle j-1\thetaint}\mathbf{u}_{j-1\thetaint}
    -A_{j\langle}^{\rangle}\mathbf{u}_{j-1[} \\
    \mathbf{f}_{j\bullet} \\
    \mathbf{f}_{j\rangle}-A_{j\rangle j+1\thetaint}\mathbf{u}_{j+1\thetaint}
    -A_{j\rangle}^{\langle}\mathbf{u}_{j+1]}
  \end{array}
\right],
\end{equation}
and now the superscript looks very natural and easy to remember.
\Cref{eq:matnt} corresponds to \Cref{eq:pde} localized to $\Omega_j$ with
Neumann transmission conditions on the interfaces,
\begin{equation}\label{neumannsubproblem}
\begin{array}{r@{\hspace{0.2em}}c@{\hspace{0.2em}}ll}
  \mathcal{L}\,u_j&=&f & \mbox{ in }\Omega_j,\\
  \mathcal{B}\,u_j&=&g & \mbox{ on }\partial\Omega_j\cap\partial\Omega,\\
  \mathbf{n}_j^T(\alpha\nabla u_j)&=&\mathbf{n}_j^T(\alpha\nabla u_{j-1})
                       & \mbox{ on }\Gamma_{j,j-1},\\
  \mathbf{n}_j^T(\alpha\nabla u_j)&=&\mathbf{n}_j^T(\alpha\nabla u_{j+1})
                       & \mbox{ on }\Gamma_{j,j+1}.
\end{array}
\end{equation}
In particular, we note that the discretization of
$-\mathbf{n}_j^T(\alpha\nabla u_j)$ on $\Gamma_{j,j-1}$ gives
$\mathbf{f}_{j\langle}^{\langle}-A_{j\langle}^{\langle}\mathbf{u}_{j\langle}
-A_{j\langle j\bullet}\mathbf{u}_{j\bullet}$ and the discretization of
$\mathbf{n}_j^T(\alpha\nabla u_{j-1})$ on $\Gamma_{j,j-1}$ gives
$\mathbf{f}_{j\langle}^{\rangle}$
$-A_{j\langle}^{\rangle}\mathbf{u}_{j-1[}$ $-A_{j\langle
    j-1\thetaint}$ $\mathbf{u}_{j-1\thetaint},$ where
  $f_{j\langle}=f_{j\langle}^{\rangle}+f_{j\langle}^{\langle}$ is
  again the splitting of the contribution from the two sides of
  $\Gamma_{j,j-1}$. Note that as in the case of Dirichlet conditions,
  if we replace $u_j$ by $v_j$ in \Cref{neumannsubproblem}, 
  the resulting subdomain problem might not be well-posed, and 
the matrix in \Cref{eq:matnt} might not be invertible.

By generalized Robin transmission conditions, we mean the generalized linear
combination of Dirichlet and Neumann conditions, i.e.
\begin{equation}\label{eq:rt}
\mathcal{Q}\left(\mathbf{n}_j^T(\alpha\nabla u_j)\right)+\mathcal{P}u_j
=\mathcal{Q}\left(\mathbf{n}_j^T(\alpha\nabla u_l)\right)+\mathcal{P}u_l
\mbox{ on }\Gamma_{j,l},
\end{equation}
where $\mathcal{Q}$ and $\mathcal{P}$ are linear, possibly non-local operators
along $\Gamma_{j,l},$ $l=j\pm1.$ At the discrete level, this corresponds to a
generalized linear combination of the interface rows of \Cref{eq:matdt} and
\Cref{eq:matnt}, while the interior rows are unchanged, i.e.
\begin{equation}\label{eq:matrt}
  \begin{array}{l}
  \left[
    \begin{array}{ccc}
      Q_{j\langle}^{\langle}A_{j\langle}^{\langle}+P_{j\langle}^{\langle} &
      Q_{j\langle}^{\langle}A_{j\langle j\bullet} &\\
      A_{j\bullet j\langle} & A_{j\bullet} & A_{j\bullet j\rangle}\\
      & Q_{j\rangle}^{\rangle}A_{j\rangle j\bullet} &
      Q_{j\rangle}^{\rangle}A_{j\rangle}^{\rangle}+P_{j\rangle}^{\rangle}
    \end{array}
  \right]
  \left[
    \begin{array}{c}
      \mathbf{u}_{j\langle}\vphantom{_{j\langle}^{\langle}} \\ \mathbf{u}_{j\bullet} \\ \mathbf{u}_{j\rangle}\vphantom{_{j\langle}^{\langle}}
    \end{array}
  \right]\\ \qquad=
  \left[
    \begin{array}{c}
      Q_{j\langle}^{\langle}\left(\mathbf{f}_{j\langle}-
      A_{j\langle j-1\thetaint}\mathbf{u}_{j-1\thetaint}\right)
      +\left(P_{j\langle}^{\langle}-Q_{j\langle}^{\langle}A_{j\langle}^{\rangle}\right)
      \mathbf{u}_{j-1[}) \\
      \mathbf{f}_{j\bullet} \\
      Q_{j\rangle}^{\rangle}\left(\mathbf{f}_{j\rangle}-
      A_{j\rangle j+1\thetaint}\mathbf{u}_{j+1\thetaint}\right)
      +\left(P_{j\rangle}^{\rangle}-Q_{j\rangle}^{\rangle}A_{j\rangle}^{\langle}\right)
      \mathbf{u}_{j+1]}
    \end{array}
  \right].
  \end{array}
\end{equation}
If $Q_{j\langle}^{\langle}=I_{j\langle}$ and
$Q_{j\rangle}^{\rangle}=I_{j\rangle}$, we can also rewrite \Cref{eq:matrt}
without the explicit splitting of $A_{j\langle}$ and $A_{j\rangle}$,
\begin{equation}\label{eq:matrt2}
  \arraycolsep0.1em
  \left[
    \begin{array}{ccc}
      \tilde{S}_{j\langle}^{\langle} & A_{j\langle j\bullet} &
      \vphantom{\left(\tilde{S}_{j\langle}^{\langle}\right)}\\
      A_{j\bullet j\langle} & A_{j\bullet} & A_{j\bullet j\rangle}\\
      & A_{j\rangle j\bullet} & \tilde{S}_{j\rangle}^{\rangle}
      \vphantom{\left(\tilde{S}_{j\langle}^{\langle}\right)}
    \end{array}
  \right]
  \left[
    \begin{array}{c}
      \mathbf{u}_{j\langle}\vphantom{\left(\tilde{S}_{j\langle}^{\langle}\right)} \\
      \mathbf{u}_{j\bullet} \\ \mathbf{u}_{j\rangle}
      \vphantom{\left(\tilde{S}_{j\langle}^{\langle}\right)}
    \end{array}
  \right]=
  \left[
    \begin{array}{c}
      \mathbf{f}_{j\langle}-A_{j\langle j-1\thetaint}\mathbf{u}_{j-1\thetaint}
      +\left(\tilde{S}_{j\langle}^{\langle}-A_{j\langle}\right)
      \mathbf{u}_{j-1[} \\
      \mathbf{f}_{j\bullet} \\
      \mathbf{f}_{j\rangle}-A_{j\rangle j+1\thetaint}\mathbf{u}_{j+1\thetaint}
      +\left(\tilde{S}_{j\rangle}^{\rangle}-A_{j\rangle}\right)
      \mathbf{u}_{j+1]}
    \end{array}
  \right],
\end{equation}
where
$\tilde{S}_{j\langle}^{\langle}=A_{j\langle}^{\langle}+P_{j\langle}^{\langle}$
and
$\tilde{S}_{j\rangle}^{\rangle}=A_{j\rangle}^{\rangle}+P_{j\rangle}^{\rangle}$.
If we first specify $\tilde{S}_{j\langle}^{\langle}$ and
$\tilde{S}_{j\rangle}^{\rangle}$, then it is not necessary to introduce
a splitting of $A_{j\langle}$ and $A_{j\rangle}$ to use \Cref{eq:matrt2}.

We now consider a special case of \Cref{eq:pde}: we assume that the
data $f$ and $g$ is supported only on $\overline{\Omega}_j$ and
vanishes elsewhere.
Suppose we are interested in the solution of \Cref{eq:pde} in
$\Omega_j$ only.  Then it would be desirable to have a problem
equivalent to the original problem in \Cref{eq:pde} but defined just on
the truncated domain $\Omega_j.$ This can be done by setting a {\it
  transparent boundary condition} on the truncation boundary
$\partial\Omega_j-\partial\Omega$ and solving
\begin{equation}\label{eq:trunc}
  \begin{array}{r@{\hspace{0.2em}}c@{\hspace{0.2em}}ll}
    \mathcal{L}\,u_j&=&f & \mbox{ in }\Omega_j,\\
    \mathcal{B}\,u_j&=&g & \mbox{ on }\partial\Omega_j\cap\partial\Omega,\\
    \mathbf{n}_j^T(\alpha\nabla u_j)+\mathrm{DtN}_ju_j&=&0
                         & \mbox{ on }\partial\Omega_j-\partial\Omega,
  \end{array}
\end{equation}
where $\mathrm{DtN}_j$ is a linear operator defined as follows:

\begin{definition}
  \label{def:dtn}
  The {\it Dirichlet-to-Neumann} (DtN) operator exterior to $\Omega_j\subset\Omega$ for
  \Cref{eq:pde} is
  \begin{equation}\label{eq:dtn}
  \begin{array}{r@{\hspace{0.2em}}c@{\hspace{0.2em}}ll}
    \mathrm{DtN}_j:\;d\rightarrow -\mathbf{n}_j^T(\alpha\nabla v),
    \mbox{ s.t.}\qquad
    \mathcal{L}\,v&=&0\quad & \mbox{in }\,\Omega-\Omega_j,\\
    \mathcal{B}\,v&=&0\quad &\mbox{on }\,
    \partial\Omega-\partial\Omega_j,\\
    v&=&d\quad &\mbox{on }\,\partial\Omega_j-\partial\Omega,
  \end{array}
  \end{equation}
  where $\mathbf{n}_j$ is the unit outward normal vector of $\Omega_j.$
\end{definition}

\begin{remark}
  The DtN operator is an example of a Poincar{\'e}-Steklov operator
  referring to maps between different boundary data.  According to
  \cite[p. VI]{KWbook}, this class of operators was first studied by
  V. A. Steklov \cite{Steklov} and H. Poincar{\'e}
  \cite{Poincare}.  They are also related to the Calderon projectors,
  see e.g. \cite{Nedelec}.
\end{remark}

\begin{lemma}\label{lem:TC}
  Assume that $f$ and $g$ in \Cref{eq:pde} vanish outside
  $\overline{\Omega}_j.$ If \Cref{eq:dtn} has a unique solution $v\in
  H^1(\Omega-\Omega_j)$ for $d\in
  H^{1/2}(\partial\Omega_j-\partial\Omega)$ and \Cref{eq:pde} has a
  unique solution $u\in H^1(\Omega)$, then \Cref{eq:trunc} also has a
  unique solution $u_j\in H^1(\Omega_j)$, and $u_j=u|_{\Omega_j}.$
\end{lemma}
\begin{proof}
  In \Cref{def:dtn}, we substitute
  $u_j|_{\partial\Omega_j-\partial\Omega}$ for $d$ and combine it with
  \Cref{eq:trunc} to find
  \begin{equation}\label{eq:exin}
    \begin{array}{r@{\hspace{0.2em}}c@{\hspace{0.2em}}llr@{\hspace{0.2em}}c@{\hspace{0.2em}}ll}
      \mathcal{L}u_j&=&f&\;\mbox{in }\Omega_j, &
      \quad\mathcal{L}v&=&0&\;\mbox{in }\Omega-\Omega_j,\\
      \mathcal{B}u_j&=&g&\;\mbox{on }\partial\Omega_j\cap\partial\Omega, &
      \mathcal{B}v&=&0&\;\mbox{on }\partial\Omega-\partial\Omega_j,\\
      \mathbf{n}_j^T(\alpha\nabla u_j)&=&\mathbf{n}_j^T(\alpha\nabla v)
      &\;\mbox{on }\partial\Omega_j-\partial\Omega, & \quad
      v&=&u_j&\;\mbox{on }\partial\Omega_j-\partial\Omega.
    \end{array}
  \end{equation}
  This coupled system for $(u_j,v)$ has at least one solution
  $(u|_{\Omega_j},u|_{\Omega-\Omega_j}).$ For uniqueness, we
  set $f:=0$ and $g:=0$, and show that $u_j\equiv0$ if
  $u_j\in H^1(\Omega_j)$: similar to \Cref{eq:weak}, we test the
  PDE satisfied by $u_j$ and $v$ separately with arbitrary $w\in V$,
  \begin{displaymath}
    \begin{array}{rcl}
      a_j(u_j,w) + b_j(u_j,w) &=& \int_{\partial\Omega_j-\partial\Omega}
      \mathbf{n}_j^T(\alpha\nabla u_j)\bar{w},\\
      a_j^c(v,w) + b_j^c(v,w) &=& \int_{\partial\Omega_j-\partial\Omega}
      -\mathbf{n}_j^T(\alpha\nabla v)\bar{w},
    \end{array}
  \end{displaymath}
  where the bilinear forms $a_j$ and $b_j$ correspond to $a$ and $b$ in
  \Cref{eq:weak} with the integration domains restricted to
  $\overline{\Omega}_j$, and similarly for $a_j^c$ and $b_j^c$ with the
  corresponding restriction to $\overline{\Omega}-\Omega_j$.  Adding the two
  equations above, and using the Neumann transmission condition from
  \Cref{eq:exin}, the integral terms cancel and we obtain
  \begin{displaymath}
    a_j(u_j,w) + a_j^c(v,w) + b_j(u_j,w) + b_j^c(v,w) = 0.
  \end{displaymath} 
  Now we need to recombine these bilinear forms into the entire
  ones defined on the original function space on $\Omega$.  Given
  $u_j\in H^{1}(\Omega_j)$ and
  $u_j|_{\partial\Omega_j-\partial\Omega}\in H^{1/2},$ we have $v\in
  H^1(\Omega-\Omega_j)$ by assumption. If we define $\tilde{u}:=u_j$
  in $\Omega_j$ and $\tilde{u}:=v$ in $\Omega-\Omega_j$, we know that
  $\tilde{u}\in H^1(\Omega)$ because $u_j=v$ on
  $\partial\Omega_j-\partial\Omega$ from \Cref{eq:exin}. Hence we
  found a $\tilde{u}\in H^1(\Omega)$ satisfying $ a(\tilde{u},w) +
  b(\tilde{u},w) = 0.$ By uniqueness of the solution of \Cref{eq:pde},
  we then conclude that $\tilde{u}\equiv0.$
\end{proof}


The matrix analogue of the exact truncation is simply Gaussian elimination: if we
consider \Cref{eq:tri} with $\mathbf{f}$ non-zero {\it only} in
$\mathbf{f}_j$, then we can rewrite \Cref{eq:tri} as
\begin{equation}\label{eq:mexin}
  \left[
    \begin{array}{ccccc}
      A_{\ctilde j} & A_{\ctilde j,j\langle} & & &\\
      A_{j\langle, \ctilde j} & A_{j\langle} & A_{j\langle j\bullet} & &\\
      & A_{j\bullet j\langle} & A_{j\bullet} & A_{j\bullet j\rangle} &\\
      & & A_{j\rangle j\bullet} & A_{j\rangle} & A_{j\rangle, j\ctilde}\\
      & & & A_{j\ctilde, j\rangle} & A_{j\ctilde}
    \end{array}
  \right]
  \left[
    \begin{array}{c}
      \mathbf{u}_{\ctilde j} \\ \mathbf{u}_{j\langle} \\ \mathbf{u}_{j\bullet}\\
      \mathbf{u}_{j\rangle} \\ \mathbf{u}_{j\ctilde}
    \end{array}
  \right] =
  \left[
    \begin{array}{c}
      0 \\ \mathbf{f}_{j\langle} \\ \mathbf{f}_{j\bullet}\\
      \mathbf{f}_{j\rangle} \\ 0
    \end{array}
  \right].
\end{equation}
To get the truncated model for $\mathbf{u}_j$ only, we eliminate
$\mathbf{u}_{\ctilde j}$ and $\mathbf{u}_{j\ctilde}$ and find
\begin{equation}\label{eq:schur}
  \left[
    \begin{array}{ccc}
      S_{j\langle}^{\langle} & A_{j\langle j\bullet} &\\
      A_{j\bullet j\langle} & A_{j\bullet} & A_{j\bullet j\rangle}\\
      & A_{j\rangle j\bullet} & S_{j\rangle}^{\rangle}
    \end{array}
  \right]
  \left[
    \begin{array}{c}
      \mathbf{u}_{j\langle}\vphantom{_{j\langle}^{\langle}} \\ \mathbf{u}_{j\bullet}\\ \mathbf{u}_{j\rangle}\vphantom{_{j\langle}^{\langle}}
    \end{array}
  \right] =
  \left[
    \begin{array}{c}
      \mathbf{f}_{j\langle}\vphantom{_{j\langle}^{\langle}} \\ \mathbf{f}_{j\bullet}\\ \mathbf{f}_{j\rangle}\vphantom{_{j\langle}^{\langle}}
    \end{array}
  \right],
\end{equation}
where $S_{j\langle}^{\langle}=A_{j\langle}-A_{j\langle, \ctilde
  j}A_{\ctilde j}^{-1}A_{\ctilde j,j\langle}$ and
$S_{j\rangle}^{\rangle}=A_{j\rangle}-A_{j\rangle,
  j\ctilde}A_{j\ctilde}^{-1}A_{j\ctilde,j\rangle}$ are known as Schur
complements, which are usually dense matrices.  Similar to
\Cref{lem:TC}, we have
\begin{lemma}\label{lem:schur}
  If in \Cref{eq:mexin} the coefficient matrix and the diagonal blocks
  $A_{\ctilde j},$ $A_{j\ctilde}$ are invertible, then \Cref{eq:schur} is
  uniquely solvable and its solution is part of the solution of
  \Cref{eq:mexin}.
\end{lemma}

\begin{remark}
  If $j=1$, then there is no $\mathbf{u}_{j\langle}$, and similarly if
  $j=J$, then there is no $\mathbf{u}_{j\rangle}$, so the
  corresponding rows and columns in \Cref{eq:mexin} and
  \Cref{eq:schur} should be deleted. In this case, \Cref{lem:schur}
  still holds.  From now on, we will treat $j=1$ and $j=J$ like
  the other $j$'s, and just assume that the non-existent blocks
  are deleted.
\end{remark}

Recalling the splitting
$A_{j\langle}=A_{j\langle}^{\rangle}+A_{j\langle}^{\langle}$, we can interpret
\Cref{eq:schur} as containing generalized Robin boundary conditions similar to
\Cref{eq:matrt} by writing
$S_{j\langle}^{\langle}=A_{j\langle}^{\langle}+P_{j\langle}^{\langle}$ and
$S_{j\rangle}^{\rangle}=A_{j\rangle}^{\rangle}+P_{j\rangle}^{\rangle}$ with
\begin{equation}\label{eq:Pdtn}
  P_{j\langle}^{\langle}:=A_{j\langle}^{\rangle}-A_{j\langle, \ctilde j}
  A_{\ctilde j}^{-1}A_{\ctilde j,j\langle},\quad
  P_{j\rangle}^{\rangle}:=A_{j\rangle}^{\langle}-A_{j\rangle, j\ctilde}
  A_{j\ctilde}^{-1}A_{j\ctilde,j\rangle}.
\end{equation}
Comparing \Cref{eq:trunc} and \Cref{eq:schur}, both for the exact
truncation, we find $P_{j\langle}^{\langle}$ should be the analogue of
the Dirichlet-to-Neumann operator $\mathrm{DtN}_j$ restricted to its
input argument on $\Gamma_{j,j-1}.$ We can also observe directly from
\Cref{eq:Pdtn} that $P_{j\langle}^{\langle}$ acts on Dirichlet data
$\mathbf{v}_{j\langle}$ by $-A_{\ctilde j,j\langle}$ (negative sign
for moving to the r.h.s.), solves for $\mathbf{v}_{\ctilde j}$ in the
exterior by $A_{\ctilde j}^{-1}$ and then evaluates the Neumann data
by
$A_{j\langle}^{\rangle}\mathbf{v}_{j\langle}+A_{j\langle, \ctilde j}\mathbf{v}_{\ctilde j}.$ %

\begin{remark}\label{rem:impml}
  PML is a popular technique to approximate the transparent boundary
  condition, and it is appropriate to make a connection here between
  the practical implementation of PML and our present discussion. The
  PML technique replaces the original problem exterior to $\Omega_j$
  by a modified one on another exterior domain $\Omega_j^{pml}$
  surrounding $\Omega_j$ along the truncation boundary
  $\partial\Omega_j-\partial\Omega$.  A DtN operator for the modified
  problem in $\Omega_j^{pml}$ can be defined as\footnote{In our
    setting, except for $j=1,J$, $\Omega_j^{pml}$ has two 
    disconnected parts -- one on the left side of $\Gamma_{j,j-1}$ and
    one on the right side of $\Gamma_{j,j+1}$. So
    $\mathrm{DtN}_j^{pml}$ is block diagonal in the sense that
    $\mathbf{n}_j^T(\tilde{\alpha}\nabla v)$ on $\Gamma_{j,j-1}$
    depends only on $d_1$, and on  
    $\Gamma_{j,j+1}$ only on $d_2$.}
  \begin{equation}\label{eq:dtnpml}
    \begin{array}{r@{\hspace{0.2em}}c@{\hspace{0.2em}}ll}
      \mathrm{DtN}_j^{pml}:\;(d_1,d_2)&\rightarrow& -\mathbf{n}_j^T(\tilde{\alpha}\nabla v)\quad&\mbox{on }(\Gamma_{j,j-1}, \Gamma_{j,j+1}),\\
      \mbox{ s.t.}\qquad
      \widetilde{\mathcal{L}}\,v&=&0\quad & \mbox{in }\,\Omega_j^{pml},\\
      \widetilde{\mathcal{B}}\,v&=&0\quad &\mbox{on }\,
      \partial\Omega_j^{pml}-\partial\Omega_j,\\
      v&=&d_1\quad &\mbox{on }\,\Gamma_{j,j-1}\subset(\partial\Omega_j^{pml}\cap\partial\Omega_j),\\
      v&=&d_2\quad &\mbox{on }\,\Gamma_{j,j+1}\subset(\partial\Omega_j^{pml}\cap\partial\Omega_j),
    \end{array}
  \end{equation}
  where $\widetilde{\mathcal{L}}:=-\nabla^T\tilde{\alpha}\nabla- {\omega^2}/\tilde{\kappa}$.
  Then an approximate transparent boundary condition can be used in the case
  stated in \Cref{lem:TC} to obtain $\tilde{u}_j\approx u_j$:
  \begin{equation}\label{eq:truncpml}
    \begin{array}{r@{\hspace{0.2em}}c@{\hspace{0.2em}}ll}
      \mathcal{L}\,\tilde{u}_j&=&f & \mbox{ in }\Omega_j,\\
      \mathcal{B}\,\tilde{u}_j&=&g & \mbox{ on }\partial\Omega_j\cap\partial\Omega,\\
      \mathbf{n}_j^T(\alpha\nabla \tilde{u}_j)+
      \mathrm{DtN}_j^{pml}\tilde{u}_j&=&0 &
      \mbox{ on }\partial\Omega_j-\partial\Omega.
    \end{array}
  \end{equation}
  To actually solve \Cref{eq:truncpml}, as we did in the proof of
  \Cref{lem:TC}, we substitute with \Cref{eq:dtnpml} and compose a problem
  defined on
  $\overline{\tilde{\Omega}}_j:=\overline{\Omega}_j\cup\overline{\Omega}_j^{pml}$,
  \begin{displaymath}
    \begin{array}{r@{\hspace{0.2em}}c@{\hspace{0.2em}}ll}
      \widetilde{\mathcal{L}}\,\tilde{u}_j&=&\tilde{f} &
      \mbox{ in }\tilde{\Omega}_j,\\
      \widetilde{\mathcal{B}}\,\tilde{u}_j&=&\tilde{g} &
      \mbox{ on }\partial\tilde{\Omega}_j,
    \end{array}
  \end{displaymath}
  where $\widetilde{\mathcal{L}}=\mathcal{L}$ in $\Omega_j$,
  $\widetilde{\mathcal{B}}=\mathcal{B}$ on $\partial\Omega_j\cap\partial\Omega$
  and $\tilde{f}$, $\tilde{g}$ are the zero extensions of $f$, $g$.
\end{remark}

\begin{remark}\label{rem:imschur}
  At the matrix level, the PML technique corresponds to replacing
  \Cref{eq:mexin} with
  \begin{equation}\label{eq:mexinpml}
    \left[
      \begin{array}{ccccc}
        \tilde{A}_{\ctilde j} & \tilde{A}_{\ctilde j,j\langle} & & &\\
        \tilde{A}_{j\langle, \ctilde j} & \tilde{A}_{j\langle} & A_{j\langle j\bullet} & &\\
        & A_{j\bullet j\langle} & A_{j\bullet} & A_{j\bullet j\rangle} &\\
        & & A_{j\rangle j\bullet} & \tilde{A}_{j\rangle} & \tilde{A}_{j\rangle, j\ctilde}\\
        & & & \tilde{A}_{j\ctilde, j\rangle} & \tilde{A}_{j\ctilde}
      \end{array}
    \right]
    \left[
      \begin{array}{c}
        \mathbf{\tilde{u}}_{\ctilde j} \\ \mathbf{u}_{j\langle} \\ \mathbf{u}_{j\bullet}\\
        \mathbf{u}_{j\rangle} \\ \mathbf{\tilde{u}}_{j\ctilde}
      \end{array}
    \right] =
    \left[
      \begin{array}{c}
        0 \\ \mathbf{f}_{j\langle} \\ \mathbf{f}_{j\bullet}\\
        \mathbf{f}_{j\rangle} \\ 0
      \end{array}
    \right],
  \end{equation}
  where the entries with tildes (except $\tilde{A}_{j\langle}$ and
  $\tilde{A}_{j\rangle}$) are typically of much smaller dimension than the
  original ones.  The Schur complemented system of \Cref{eq:mexinpml},
  intended to approximate \Cref{eq:schur}, is
  \begin{equation}\label{eq:schurpml}
    \left[
      \begin{array}{ccc}
        \tilde{S}_{j\langle}^{\langle} & A_{j\langle j\bullet} &\\
        A_{j\bullet j\langle} & A_{j\bullet} & A_{j\bullet j\rangle}\\
        & A_{j\rangle j\bullet} & \tilde{S}_{j\rangle}^{\rangle}
      \end{array}
    \right]
    \left[
      \begin{array}{c}
        \mathbf{u}_{j\langle}\vphantom{_{j\langle}^{\langle}} \\ \mathbf{u}_{j\bullet}\\ \mathbf{u}_{j\rangle}\vphantom{_{j\langle}^{\langle}}
      \end{array}
    \right] =
    \left[
      \begin{array}{c}
        \mathbf{f}_{j\langle}\vphantom{_{j\langle}^{\langle}} \\ \mathbf{f}_{j\bullet}\\ \mathbf{f}_{j\rangle}\vphantom{_{j\langle}^{\langle}}
      \end{array}
    \right],
  \end{equation}
  where
  $\tilde{S}_{j\langle}^{\langle}:=\tilde{A}_{j\langle}-\tilde{A}_{j\langle,
    j\ctilde} \tilde{A}_{j\ctilde}^{-1}\tilde{A}_{j\ctilde,j\langle}$
  and
  $\tilde{S}_{j\rangle}^{\rangle}:=\tilde{A}_{j\rangle}-\tilde{A}_{j\rangle,
    j\ctilde} \tilde{A}_{j\ctilde}^{-1}\tilde{A}_{j\ctilde,j\rangle}$.
  As before, we see that
  $\tilde{P}_{j\langle}^{\langle}:=\tilde{S}_{j\langle}^{\langle}-A_{j\langle}^{\langle}$
  is the matrix version of the PML--DtN operator $\mathrm{DtN}_j^{pml}$
  restricted to its input argument on $\Gamma_{j,j-1}$.  For
  implementation, one usually does not solve \Cref{eq:schurpml} directly, 
  one solves instead \Cref{eq:mexinpml}.
\end{remark}

\subsection{Green's function}

We have prepared the reader for the new Helmholtz solvers so far only
based on the concept of domain decomposition and transmission
conditions.  There are however also formulations of these new
Helmholtz solvers based on Green's functions, which we introduce
next. By definition, a fundamental solution $G(\mathbf{x},\mathbf{y})$
of the partial differential operator $\mathcal{M}$ is a solution of
the PDE in a domain $Y$ without consideration of boundary conditions,
\begin{displaymath}
  \mathcal{M}_{\mathbf{y}}G(\mathbf{x},\mathbf{y}) = \delta(\mathbf{y}-\mathbf{x}),
  \quad\forall\,\mathbf{x}\in X\subseteq Y,
\end{displaymath}
where $\mathcal{M}_y$ is the operator $\mathcal{M}$ acting on the
$\mathbf{y}$ variable, and $\delta(\mathbf{y}-\mathbf{x})$ is the
Dirac delta function representing a point source and satisfying
$\int_Y\delta(\mathbf{y}-\mathbf{x})v(\mathbf{y})~d\mathbf{y}=v(\mathbf{x}).$
Let $u$ be a solution of $\mathcal{L}u=f$ in $Y,$ see \Cref{eq:pde}
without boundary condition, and
$\mathcal{M}:=-\nabla^T(\alpha^T\nabla \cdot)-\frac{\omega^2}{\kappa}.$ %
Using integration by parts, we have formally for $\mathbf{x}\in X-\partial Y$
\begin{equation}\label{eq:rep}
  \medmuskip=-1mu
  \thinmuskip=-1mu
  \thickmuskip=-1mu
  \nulldelimiterspace=-0pt
  \scriptspace=0pt 
  u(\mathbf{x}) = \int_{Y} G(\mathbf{x},\mathbf{y})f(\mathbf{y})\,
  \mathrm{d}\mathbf{y}+\int_{\partial Y}\mathbf{n}_{\mathbf{y}}^T(\alpha
  \nabla u(\mathbf{y})){G}(\mathbf{x},\mathbf{y})
  -\mathbf{n}_{\mathbf{y}}^T\left(\alpha^T\nabla_{\mathbf{y}}
    G(\mathbf{x},\mathbf{y})\right)u(\mathbf{y})\,
  \mathrm{d}\sigma(\mathbf{y}),
\end{equation}
which is a representation formula for the solution, and the three summands are
called volume potential, single layer potential and double layer potential; see
\cite{Nedelec}.  A justification of \Cref{eq:rep} involves existence, regularity
and singularity of the fundamental solution and the solution, which can be found
in the literature if $\alpha$, $\kappa$, $f$ and $\partial X$ are bounded and
smooth; see e.g. \cite{Sauvigny, Ramm}.  We note that the r.h.s. of
\Cref{eq:rep} uses both Neumann and Dirichlet traces of $u$ while usually a
well-posed boundary condition only tells us one of them or a generalized linear
combination of them.  For example, let $u=0$ on $\partial Y.$ To get a usable
representation, we can require $G$ also to satisfy $G(\mathbf{x},\mathbf{y})=0$
for $\mathbf{y}\hspace{-0.2em}\in\hspace{-0.2em}\partial Y$,
$\mathbf{x}\hspace{-0.2em}\in\hspace{-0.2em} X.$ Then, the single layer and the
double layer potentials in \Cref{eq:rep} vanish, and we get the simple
representation formula
\begin{equation}\label{eq:repf}
  u(\mathbf{x}) = \int_{Y} G(\mathbf{x},\mathbf{y})f(\mathbf{y})\,
  \mathrm{d}\mathbf{y}.
\end{equation}
We call a {\it fundamental solution} satisfying a homogeneous boundary
condition {\it Green's function}. People however sometimes use the two
terms in an exchangeable way.

\begin{remark}
  If $u$ satisfies an inhomogeneous boundary condition $\mathcal{B}u|_{\partial{Y}}=g$, we
  can lift (extend) the boundary data into $Y$, i.e. find a function $v$ on $Y$
  such that $\mathcal{B}v|_{\partial Y}=g$, and subtract it from $u$ so that the
  boundary condition becomes homogeneous for the new unknown $\tilde{u}:=u-v.$
  (We will see this trick is useful also for the implementation of the
  transmission condition in \Cref{eq:rt}.)  For the Green's function $G$, we impose
  $\mathcal{B}_{\mathbf{y}}^TG(\mathbf{x},\mathbf{y})=0$ where $\mathcal{B}^T$
  corresponds to $\mathcal{B}$ but with $\alpha$ replaced by $\alpha^T$.
\end{remark}

\begin{remark}
  Another convention is to define the Green's function
  $G(\mathbf{x},\mathbf{y})$ through
  $\mathcal{L}_{\mathbf{x}}G(\mathbf{x},\mathbf{y})=\delta(\mathbf{x}-\mathbf{y})$
  equipped with homogeneous boundary conditions as for $u$. Then,
  \Cref{eq:repf} can be obtained by the superposition principle.
  Similarly, we may define $H(\mathbf{x},\mathbf{y})$ through
  $\mathcal{L}_{\mathbf{y}}H(\mathbf{x},\mathbf{y})=\delta(\mathbf{y}-\mathbf{x})$.
  We then have $H(\mathbf{y},\mathbf{x})=G(\mathbf{x},\mathbf{y})$.
  Furthermore, if $\alpha=\alpha^T$, we have
  $G(\mathbf{x},\mathbf{y})=G(\mathbf{y},\mathbf{x})$.
\end{remark}

We now point out an analogy between \Cref{eq:repf} and the solution
$\mathbf{u} = A^{-1} \mathbf{f}$ of \Cref{eq:las}.  For a particular
value of $\mathbf{x}$, $u(\mathbf{x})$ in \Cref{eq:repf} corresponds
to a particular (say, the $m$-th) entry of $\mathbf{u},$ and
$G(\mathbf{x},\mathbf{y})$ corresponds then also to the $m$-th row of
$A^{-1},$ and the integral in \Cref{eq:repf} becomes the inner product
of the row of $A^{-1}$ with $\mathbf{f}.$ Similarly, for a particular
$\mathbf{y},$ $G(\mathbf{x},\mathbf{y})$ corresponds to a particular
column of $A^{-1}.$ We now take a closer look again at the Schur
complement $S_{j\langle}^{\langle}$ in \Cref{eq:schur}, which is
essentially derived from the 2-by-2 block matrix by Gaussian
elimination,
\begin{equation}\label{eq:2bl}
  \left[
    \begin{array}{cc}
      A_{\ctilde j} & A_{\ctilde j,j\langle}\\
      A_{j\langle, \ctilde j} & A_{j\langle}\vphantom{_{j\langle}^{\langle}}
    \end{array}
  \right] =
  \left[
    \begin{array}{cc}
      I_{\ctilde j} & \\
      A_{j\langle, \ctilde j}A_{\ctilde j}^{-1} & I_{j\langle}\vphantom{_{j\langle}^{\langle}}
    \end{array}
  \right]
  \left[
    \begin{array}{cc}
      A_{\ctilde j} & A_{\ctilde j,j\langle}\\
       & S_{j\langle}^{\langle}
    \end{array}
  \right].
\end{equation}
Taking the inverse of both sides, we find
\begin{displaymath}\small
  \begin{array}{rl}
    &\left[
      \begin{array}{cc}
        A_{\ctilde j} & A_{\ctilde j,j\langle}\\
        A_{j\langle, \ctilde j} & A_{j\langle}\vphantom{_{j\langle}^{\langle}}
      \end{array}
    \right]^{-1} =
    \left[
      \begin{array}{cc}
        A_{\ctilde j} & A_{\ctilde j,j\langle}\\
        & S_{j\langle}^{\langle}
      \end{array}
    \right]^{-1}
    \left[
      \begin{array}{cc}
        I_{\ctilde j} & \\
        -A_{j\langle, \ctilde j}A_{\ctilde j}^{-1} & I_{j\langle}\vphantom{_{j\langle}^{\langle}}
      \end{array}
    \right]\\
    = &
    \left[
      \begin{array}{cc}
        I_{\ctilde j} & -A_{\ctilde j}^{-1}A_{\ctilde j,j\langle}\\
        & I_{j\langle}^{\langle}
      \end{array}
    \right]
    \left[
      \begin{array}{cc}
        A_{\ctilde j}^{-1} & \\
        & {S_{j\langle}^{\langle}}^{-1}
      \end{array}
    \right]
    \left[
      \begin{array}{cc}
        I_{\ctilde j} & \\
        -A_{j\langle, \ctilde j}A_{\ctilde j}^{-1} & I_{j\langle}\vphantom{_{j\langle}^{\langle}}
      \end{array}
    \right]=
    \left[
      \begin{array}{cc}
        * & *\\
        * & {S_{j\langle}^{\langle}}^{-1}
      \end{array}
    \right],
  \end{array}
\end{displaymath}
where we omit the terms marked by $*$. Recalling the analogy between
the matrix inverse and the Green's function, we can identify
${S_{j\langle}^{\langle}}^{-1}$ as a diagonal part of the Green's
function $G(\mathbf{x},\mathbf{y}).$ Here, $G$ satisfies for
$\mathbf{x}, \mathbf{y}\in \Omega_{\ctilde j}^{+h}$
\begin{displaymath}
  \begin{array}{r@{\hspace{0.2em}}c@{\hspace{0.2em}}ll}
    \mathcal{M}_{\mathbf{y}} G(\mathbf{x},\mathbf{y})&=&\delta(\mathbf{y}-\mathbf{x})
    & \mbox{in }\Omega_{\ctilde j}^{+h},\\
    \mathcal{B}_{\mathbf{y}}^TG(\mathbf{x},\mathbf{y})&=&0&
    \mbox{on } \partial\Omega\cap\partial\Omega_{\ctilde j}^{+h},\\
    G(\mathbf{x},\mathbf{y})&=&0& \mbox{on } \Gamma_{j,j-1}^{+h},
  \end{array}
\end{displaymath}
where $\Omega_{\ctilde j}^{+h}$ is the domain covering the left part of
$\Omega-\Omega_j$ but with one grid layer further into $\Omega_j,$ and
$\Gamma_{j,j-1}^{+h}$ is the interface of $\Omega_{\ctilde j}^{+h}$ in
$\Omega_j.$ We see that ${S_{j\langle}^{\langle}}^{-1}$ corresponds to
$G(\mathbf{x},\mathbf{y})$ with $\mathbf{x}, \mathbf{y}$ both restricted to
$\Gamma_{j,j-1}.$ This identification was first given in \cite{EY1, EY2}.
\section{Optimized Schwarz methods}
\label{sec:osm}

If we have $u_j=u|_{\Omega_j},$ $j=1,..,J$ with $u$ the solution of
\Cref{eq:pde}, we must have consistency: a) $u_j$ and $u_l$
matching\footnote{Usually, `match' means `coincide'.  But there are exceptions,
  e.g. the original problem can enforce a jump of the solution across a surface,
  or at discrete level non-conforming discretization is used in overlaps and
  interfaces.} on $\overline{\Omega}_j\cap\overline{\Omega}_l$, and b) the
original equations are satisfied in neighborhoods of interfaces
$\partial\Omega_j\cap\overline{\Omega}_l$, for all $j\neq l$,
$j,l\in\{1,..,J\}$.  Conversely, if $u_j$, $j=1,..,J$ solves \Cref{eq:pde}
restricted to $\overline{\Omega}_j$, then a) and b) together
imply $u_j=u|_{\Omega_j}$.  For second-order elliptic PDEs like
\Cref{eq:pde}, b) is equivalent to say the Neumann traces
$\mathbf{n}^T\alpha\nabla u_j$, $\mathbf{n}^T\alpha\nabla u_l$ match\footnote{We
  assume all surface/line/point sources on the interface have been split and
  incorporated into subdomain problems.}  on
$\partial{\Omega}_j\cap\overline{\Omega}_l$ for all $j,l\in\{1,..,J\}$.  Hence,
when $\{\Omega_j\}_{j=1}^{J}$ are non-overlapping, a) and b) reduce to both
Dirichlet and Neumann (or any other equivalent pair of) traces to match on every
interface. If the subdomains $\{\Omega_j\}_{j=1}^{J}$ overlap, a) and b) as a
whole can be further condensed as: a') {\it one} transmission condition
(matching one of Dirichlet/Neumann/generalized Robin traces) on every interface,
and b') the transmission conditions ensure $u_j=u_l$ on
$\overline{\Omega}_{j}\cap\overline{\Omega}_l$ if $u_j$ and $u_l$ both solve the
original equations restricted to overlaps.  Therefore, no matter the
decomposition being overlapping or non-overlapping, the original problem can be
rewritten as a system of subdomain problems coupled through transmission
conditions on interfaces; c.f. \cite{QVbook}.

Schwarz methods split the coupling between subdomains by taking the
interface data from the already available iterates and solve subdomain
problems to get the new iterate\footnote{Our description is
  applicable also to single-trace methods such as BDD, FETI and FETI-H
  which are based on non-overlapping decompositions and use the same
  data for neighboring subdomains on each interface. In contrast, for
  Schwarz methods each subdomain is equipped with its own interface
  data which is provided by (but not used by) the other
  subdomains.}. Historically the first Schwarz method was the {\it
  alternating Schwarz method} introduced by Schwarz himself
\cite{Schwarz:1870:UGA}, where one subdomain is solved at a time, and
then the newest data is passed on to the neighboring subdomains. This
is analogous to the Gauss-Seidel iteration in linear algebra. More
than a century later, Lions introduced the so called {\it parallel
  Schwarz method} \cite{Lions:1988:SAM}, where each subdomain solves
its local problem at the same time, and data is only exchanged
afterward. This is analogous to the Jacobi iteration in linear
algebra. In the alternating Schwarz method in the presence of many
subdomains, one also needs to specify an ordering, and for the
Helmholtz solvers we are interested in here with the decomposition
into a one dimensional sequence of subdomains, the particular ordering
of sweeping from the first subdomain to the last one and then back,
like in the symmetrized Gauss-Seidel iteration in linear algebra, is
important, and we call these `double sweep' methods.

One also has to decide on which unknowns to write the iteration: one
can choose subdomain approximations (for the equivalent coupled system),
global approximations (for the original problem), interface data, and
residuals. We explain now in detail these formulations and their relations.

\subsection{Subdomain transmission form of Schwarz methods}

In this formulation, the iterates represent approximate solutions on
the subdomains. The corresponding {\it double sweep optimized Schwarz
  method} (DOSM) was first proposed in \cite{Nataf93, NN97}.  Based on
the decomposition defined in \Cref{eq:dd}, we state the DOSM in
\Cref{alg:DOSMPDE} at the PDE
\begin{algorithm}
  \caption{DOSM in the {\bf subdomain transmission} form at the {\bf PDE} level
   }
  \label{alg:DOSMPDE}
  Given the last iterate $\left\{u_j^{(n-1)}\mbox{ in }\Omega_j,\
    j=1,..,J\right\}$, solve successively for $j=1,..,J-1,$
  \begin{displaymath}
  \small
  \begin{array}{r@{\hspace{0.2em}}c@{\hspace{0.2em}}ll}
    \mathcal{L}\,u_j^{(n-\frac{1}{2})}&=&f & \mbox{ in }\Omega_j,\\
    \mathcal{B}\,u_j^{(n-\frac{1}{2})}&=&g & \mbox{ on }
    \partial\Omega\cap\partial\Omega_j,\\
    \mathcal{Q}_{j}^{\langle}\left(\mathbf{n}_j^T\alpha\nabla
      u_j^{(n-\frac{1}{2})}\right)+\mathcal{P}_{j}^{\langle}u_j^{(n-\frac{1}{2})}&=&
    \mathcal{Q}_{j}^{\langle}\left(\mathbf{n}_j^T\alpha\nabla
      u_{j-1}^{(n-\frac{1}{2})}\right)+\mathcal{P}_{j}^{\langle}u_{j-1}^{(n-\frac{1}{2})}
     & \mbox{ on }\Gamma_{j,j-1},\\
     \mathcal{Q}_{j}^{\rangle}\left(\mathbf{n}_j^T\alpha\nabla
       u_j^{(n-\frac{1}{2})}\right)+\mathcal{P}_{j}^{\rangle}u_j^{(n-\frac{1}{2})}&=&
     \mathcal{Q}_{j}^{\rangle}\left(\mathbf{n}_j^T\alpha\nabla
       u_{j+1}^{(n-1)}\right)+\mathcal{P}_{j}^{\rangle}u_{j+1}^{(n-1)}
     & \mbox{ on }\Gamma_{j,j+1},
  \end{array}
  \end{displaymath}
  where $\mathcal{Q}_j^{\langle},$ $\mathcal{P}_j^{\langle}$
  and $\mathcal{Q}_j^{\langle},$ $\mathcal{P}_j^{\langle}$ are some
  possibly non-local operators on the interfaces $\Gamma_{j,j-1}$
  and $\Gamma_{j,j+1}$, and $\mathbf{n}_j$ is the unit outward normal vector of
  $\Omega_j$. We call this process the forward sweep.

  Then, the backward sweep consists in solving successively for $j=J,..,1,$
  \begin{displaymath}
    \small
  \begin{array}{r@{\hspace{0.2em}}c@{\hspace{0.2em}}ll}
    \mathcal{L}\,u_j^{(n)}&=&f & \mbox{ in }\Omega_j,\\
    \mathcal{B}\,u_j^{(n)}&=&g & \mbox{ on }
    \partial\Omega\cap\partial\Omega_j,\\
    \mathcal{Q}_{j}^{\langle}\left(\mathbf{n}_j^T\alpha\nabla
      u_j^{(n)}\right)+\mathcal{P}_{j}^{\langle}u_j^{(n)}&=&
    \mathcal{Q}_{j}^{\langle}\left(\mathbf{n}_j^T\alpha\nabla
      u_{j-1}^{(n-\frac{1}{2})}\right)+\mathcal{P}_{j}^{\langle}u_{j-1}^{(n-\frac{1}{2})}
    & \mbox{ on }\Gamma_{j,j-1},\\
    \mathcal{Q}_{j}^{\rangle}\left(\mathbf{n}_j^T\alpha\nabla
      u_j^{(n)}\right)+\mathcal{P}_{j}^{\rangle}u_j^{(n)}&=&
    \mathcal{Q}_{j}^{\rangle}\left(\mathbf{n}_j^T\alpha\nabla
      u_{j+1}^{(n)}\right)+\mathcal{P}_{j}^{\rangle}u_{j+1}^{(n)}
    & \mbox{ on }\Gamma_{j,j+1}.
  \end{array}
  \end{displaymath}
\end{algorithm}
level for \Cref{eq:pde} and in \Cref{alg:DOSMMAT} at the matrix level
\begin{algorithm}
  \caption{DOSM in the {\bf subdomain transmission} form at the {\bf matrix}
    level}
  \label{alg:DOSMMAT}
  Given the last iterate \{$\mathbf{u}_j^{(n-1)},$ $j=1,..,J$\}, solve
  successively for $j=1,..,J-1,$
  \begin{displaymath}
    \small
    \begin{array}{ll}
      &\left[
        \begin{array}{ccc}
          Q_{j\langle}^{\langle}A_{j\langle}^{\langle}+P_{j\langle}^{\langle} &
          Q_{j\langle}^{\langle}A_{j\langle j\bullet} &\vphantom{_{j\langle}^{(n-\frac{1}{2})}}\\
          A_{j\bullet j\langle} & A_{j\bullet} & A_{j\bullet j\rangle}\vphantom{_{j\langle}^{(n-\frac{1}{2})}}\\
          & Q_{j\rangle}^{\rangle}A_{j\rangle j\bullet} &
          Q_{j\rangle}^{\rangle}A_{j\rangle}^{\rangle}+P_{j\rangle}^{\rangle}\vphantom{_{j\langle}^{(n-\frac{1}{2})}}
        \end{array}
      \right]
      \left[
        \begin{array}{c}
          \mathbf{u}_{j\langle}^{(n-\frac{1}{2})} \\
          \mathbf{u}_{j\bullet}^{(n-\frac{1}{2})} \\
          \mathbf{u}_{j\rangle}^{(n-\frac{1}{2})}
        \end{array}
      \right]\\ =&
      \left[
        \begin{array}{c}
          Q_{j\langle}^{\langle}\left(\mathbf{f}_{j\langle}-
            A_{j\langle j-1\thetaint}\mathbf{u}_{j-1\thetaint}^{(n-\frac{1}{2})}\right)
          +\left(P_{j\langle}^{\langle}-Q_{j\langle}^{\langle}A_{j\langle}^{\rangle}\right)
          \mathbf{u}_{j-1[}^{(n-\frac{1}{2})} \\
          \mathbf{f}_{j\bullet} \\
          Q_{j\rangle}^{\rangle}\left(\mathbf{f}_{j\rangle}-
            A_{j\rangle j+1\thetaint}\mathbf{u}_{j+1\thetaint}^{(n-1)}\right)
          +\left(P_{j\rangle}^{\rangle}-Q_{j\rangle}^{\rangle}A_{j\rangle}^{\langle}\right)
          \mathbf{u}_{j+1]}^{(n-1)}
        \end{array}
      \right],
    \end{array}
  \end{displaymath}
  which constitutes the forward sweep.

  Then perform the backward sweep: solve successively for $j=J,..,1,$
  \begin{displaymath}
    \small
    \begin{array}{ll}
      &\left[
        \begin{array}{ccc}
          Q_{j\langle}^{\langle}A_{j\langle}^{\langle}+P_{j\langle}^{\langle} &
          Q_{j\langle}^{\langle}A_{j\langle j\bullet} &\\
          A_{j\bullet j\langle} & A_{j\bullet} & A_{j\bullet j\rangle}\\
          & Q_{j\rangle}^{\rangle}A_{j\rangle j\bullet} &
          Q_{j\rangle}^{\rangle}A_{j\rangle}^{\rangle}+P_{j\rangle}^{\rangle}
        \end{array}
      \right]
      \left[
        \begin{array}{c}
          \mathbf{u}_{j\langle}^{(n)} \\
          \mathbf{u}_{j\bullet}^{(n)} \\
          \mathbf{u}_{j\rangle}^{(n)}
        \end{array}
      \right]\\ =&
      \left[
        \begin{array}{c}
          Q_{j\langle}^{\langle}\left(\mathbf{f}_{j\langle}-
            A_{j\langle j-1\thetaint}\mathbf{u}_{j-1\thetaint}^{(n-\frac{1}{2})}\right)
          +\left(P_{j\langle}^{\langle}-Q_{j\langle}^{\langle}A_{j\langle}^{\rangle}\right)
          \mathbf{u}_{j-1[}^{(n-\frac{1}{2})} \\
          \mathbf{f}_{j\bullet} \\
          Q_{j\rangle}^{\rangle}\left(\mathbf{f}_{j\rangle}-
            A_{j\rangle j+1\thetaint}\mathbf{u}_{j+1\thetaint}^{(n)}\right)
          +\left(P_{j\rangle}^{\rangle}-Q_{j\rangle}^{\rangle}A_{j\rangle}^{\langle}\right)
          \mathbf{u}_{j+1]}^{(n)}
        \end{array}
      \right].
    \end{array}
  \end{displaymath}

  If $Q_{j\langle}^{\langle}=I$ and $Q_{j\rangle}^{\rangle}=I$, one
  can rewrite the subproblems like in \Cref{eq:matrt2}.
\end{algorithm}
for \Cref{eq:tri}. Note that the transmission conditions on the
interfaces can be changed in the process, e.g. from the forward sweep
to the backward sweep or from one iteration to the next. Note also
that in the double sweep, the subproblem on the last subdomain
$\Omega_J$ is solved only once.  If the transmission conditions on
$\Gamma_{1,2}$ are the same in the forward and the backward sweeps, we
find the same problem on $\Omega_1$ is solved in the backward sweep of
the current iteration and in the forward sweep of the next iteration,
so one can also solve it only once.

For the {\it parallel optimized Schwarz method} (POSM), where all
subdomains are solved simultaneously and data is exchanged afterward,
it was shown in \cite{NRS94} that if optimal transmission conditions
based on the DtN operators are used, then the algorithm converges in a
finite number of steps, equal to the number of subdomains, and thus
the iteration operator is nilpotent of degree equal to the number of
subdomains.  We present now an optimal choice for DOSM, where the
operators $\mathcal{Q}_{j}^{\rangle}$ and $\mathcal{P}_{j}^{\rangle}$
can still be arbitrary, as long as the subdomain problems are
well-posed.

\begin{theorem}\label{thm:1steppde}
  If in the forward and the backward sweeps $\mathcal{Q}_j^{\langle}$ is the
  identity,
  $\mathcal{P}_j^{\langle}=\mathrm{DtN}_j^{\langle}:=\mathrm{DtN}_j|_{\Gamma_{j,j-1}}$
  is well-defined for $j=2,..,J$ as in \Cref{def:dtn}, and the original problem
  in \Cref{eq:pde} and the subdomain problems in \Cref{alg:DOSMPDE} are uniquely
  solvable, then \Cref{alg:DOSMPDE} converges in one double sweep for an
  arbitrary initial guess, and $u_j^{(1)}=u|_{\Omega_j},$ $j=1,..,J$ with $u$
  the solution of \Cref{eq:pde}. This means that the iteration operator of
    DOSM is nilpotent of degree one.
\end{theorem}

\begin{proof}
  We note that the subdomain problems in \Cref{alg:DOSMPDE} are
  satisfied by the solution $u.$ By linearity, it is thus sufficient to prove
  $u_j^{(1)}\equiv0$ when $f\equiv0$ and $g\equiv 0.$ We first consider
  $u_2^{(\frac{1}{2})}$ which satisfies the transmission condition
  \begin{equation}\label{eq:u2u1}
  \mathbf{n}_2^T\alpha\nabla u_2^{(\frac{1}{2})}
  +\mathrm{DtN}_{2}^{\langle}u_2^{(\frac{1}{2})}=
  \mathbf{n}_2^T\alpha\nabla u_{1}^{(\frac{1}{2})}
  +\mathrm{DtN}_{2}^{\langle}u_{1}^{(\frac{1}{2})}\mbox{ on }\Gamma_{2,1}.\\
  \end{equation}
  Since $\mathcal{L}u_1^{(\frac{1}{2})}=0$ in $\Omega_1\supset\Omega_{\ctilde
    2},$ $\mathcal{B}u_1^{(\frac{1}{2})}=0$ on
  $(\partial\Omega\cap\partial\Omega_1)\supset
  (\partial\Omega\cap\partial\Omega_{\ctilde 2})$, from \Cref{def:dtn}
  we have
  $\mathrm{DtN}_{2}^{\langle}u_{1}^{(\frac{1}{2})}=-\mathbf{n}_2^T\alpha\nabla
  u_{1}^{(\frac{1}{2})}.$ Substituting this into \Cref{eq:u2u1} we obtain
  \begin{displaymath}
  \mathbf{n}_2^T\alpha\nabla u_2^{(\frac{1}{2})}
  +\mathrm{DtN}_{2}^{\langle}u_2^{(\frac{1}{2})}=0\mbox{ on }\Gamma_{2,1}.
  \end{displaymath}
  Now assuming that
  \begin{equation}\label{eq:uj0}
  \mathbf{n}_j^T\alpha\nabla u_j^{(\frac{1}{2})}
  +\mathrm{DtN}_{j}^{\langle}u_j^{(\frac{1}{2})}=0\mbox{ on }\Gamma_{j,j-1},
  \end{equation}
  we will show that this also holds for $j+1$ instead of $j.$ In fact, by the
  assumption that $\mathrm{DtN}_{j+1}^{\langle}$ is well-defined, we have a
  unique solution $v_{\ctilde j+1}$ of the problem
  \begin{displaymath}
  \begin{array}{r@{\hspace{0.2em}}c@{\hspace{0.2em}}ll}
    \mathcal{L}v_{\ctilde j+1}&=&0 & \mbox{ in }\Omega_{\ctilde j+1},\\
    \mathcal{B}v_{\ctilde j+1}&=&0 & \mbox{ on }
    \partial\Omega\cap\partial\Omega_{\ctilde j+1},\\
    v_{\ctilde j+1}&=&u_j^{(\frac{1}{2})} & \mbox{ on }\Gamma_{j+1,j}.\\
  \end{array}
  \end{displaymath}
  By \Cref{lem:TC}, we have from \Cref{eq:uj0} that %
  $u_j^{(\frac{1}{2})}=v_{\ctilde j+1}$ in $\Omega_j\cap\Omega_{\ctilde j+1}$. %
  Therefore, %
  $\mathrm{DtN}_{j+1}^{\langle}u_j^{(\frac{1}{2})}=-\mathbf{n}_{j+1}^T
  \alpha\nabla v_{\ctilde j+1}=-\mathbf{n}_{j+1}^T\alpha\nabla
  u_{j}^{(\frac{1}{2})}$ %
  on $\Gamma_{j+1,j}.$ Substituting this into the transmission condition for
  $u_{j+1}^{(\frac{1}{2})}$ we find
  \begin{displaymath}
  \mathbf{n}_{j+1}^T\alpha\nabla u_{j+1}^{(\frac{1}{2})}
  +\mathrm{DtN}_{j+1}^{\langle}u_{j+1}^{(\frac{1}{2})}=0 \mbox{ on
  }\Gamma_{j+1,j}.
  \end{displaymath}
  By induction, \Cref{eq:uj0} holds for all $j=2,..,J-1$ and also
  $j=J$ except that we write $u_J^{(1)}$ instead of
  $u_J^{(\frac{1}{2})}.$ By \Cref{lem:TC}, and recalling
  that $f\equiv0$ and $g\equiv0,$ we obtain $u_J^{(1)}=u|_{\Omega_J}\equiv0.$
  Now assuming that
  \begin{equation}\label{eq:uj00}
    u_{j+1}^{(1)}\equiv 0\mbox{ in }\Omega_{j+1},
  \end{equation}
  we have to show that $u_{j}\equiv0$ in $\Omega_{j}.$ This follows
  directly from \Cref{eq:uj00} and \Cref{eq:uj0}, which imply
  that all the data in the problem for $u_{j}^{(1)}$ vanish, and 
  by the assumption that the subdomain problem is uniquely solvable.
\end{proof}
We also have the equivalent result of convergence in one step for the 
discrete case:
\begin{theorem}\label{thm:1stepmat}
  If in the forward and the backward sweeps
  $Q_{j\langle}^{\langle}=I_{j\langle},$
  $P_{j\langle}^{\langle}=A_{j\langle}^{\rangle}-A_{j\langle, \ctilde j}
  A_{\ctilde j}^{-1}A_{\ctilde j,j\langle}$ is well-defined, for $j=2,..,J,$ and
  the original problem in \Cref{eq:tri} and the subdomain problems in
  \Cref{alg:DOSMMAT} are uniquely solvable, then \Cref{alg:DOSMMAT} converges in
  one step and $\mathbf{u}_j^{(1)}=R_j\mathbf{u}$ with $\mathbf{u}$ the solution
  of \Cref{eq:tri}. This means the iteration matrix of DOSM is nilpotent of
    degree one.
\end{theorem}
\begin{proof}
  First, the subdomain problems are consistent, i.e. neglecting the iteration
  numbers and substituting $\mathbf{u}_j=R_j\mathbf{u}$, we find the equations
  are satisfied by $\mathbf{u}$.  Hence, by considering the errors, we only need
  to show that $\mathbf{u}_j^{(1)}=0$ if $\mathbf{f}=0$. In the problem for
  $\mathbf{u}_2^{(\frac{1}{2})},$ the r.h.s. corresponding to $\Gamma_{2,1}$
  becomes
  \begin{displaymath}
    -A_{2\langle1\thetaint}\mathbf{u}_{1\thetaint}^{(\frac{1}{2})} -A_{2\langle1\thetaint}
    A_{1\thetaint}^{-1}A_{1\thetaint 2\langle}
    \mathbf{u}_{1[}^{(\frac{1}{2})}
    =-A_{2\langle1\thetaint}\left(\mathbf{u}_{1\thetaint}^{(\frac{1}{2})} -
      A_{1\thetaint}^{-1}A_{1\thetaint 2\langle} \mathbf{u}_{1[}^{(\frac{1}{2})}
    \right)=0,
  \end{displaymath}
  since $\mathbf{u}_1^{(\frac{1}{2})}$ satisfies %
  $A_{1\thetaint}\mathbf{u}_{1\thetaint}^{(\frac{1}{2})} +
  A_{1\thetaint2\langle}\mathbf{u}_{1[}^{(\frac{1}{2})}=0.$ %
  In other words, we have for $j=2,$
  \begin{equation}\label{eq:sj0}
    S_{j\langle}^{\langle}\mathbf{u}_{j\langle}^{(\frac{1}{2})}
    + A_{j\langle j\bullet} \mathbf{u}_{j\bullet}^{(\frac{1}{2})} = 0.
  \end{equation}
  Assuming \Cref{eq:sj0} holds for one $j$ in $\{2,..,J-1\},$ we will
  show it will also hold for $j+1$ instead of $j$. In the theorem we are
  proving, we have assumed that $A_{\ctilde j+1}$ is invertible, so we can
  introduce $\mathbf{v}_{\ctilde j+1}:=[\mathbf{v}_{\ctilde j};\mathbf{v}_{j-1\rovlp};\mathbf{v}_{j\thetaint}]$ as the solution of
  \begin{displaymath}
  \left[
    \begin{array}{ccc}
      A_{\ctilde j} & A_{\ctilde j,j-1\rovlp}\\
      A_{j-1\rovlp\ctilde j} & A_{j-1\rovlp} & A_{j-1\rovlp j\thetaint}\\
      & A_{j\thetaint j-1\rovlp} & A_{j\thetaint}\vphantom{_{j\langle}^{(n-\frac{1}{2})}}
    \end{array}
  \right]
  \left[
    \begin{array}{c}
      \mathbf{v}_{\ctilde j}\\ \mathbf{v}_{j-1\rovlp} \\ \mathbf{v}_{j\thetaint}\vphantom{_{j\langle}^{(n-\frac{1}{2})}}
    \end{array}
  \right] =
  \left[
    \begin{array}{c}
      0 \\ 0 \\ -A_{\ctilde j+1,j+1\langle}\mathbf{u}_{j[}^{(\frac{1}{2})}
    \end{array}
  \right].
  \end{displaymath}
  By Gaussian elimination, we have
  \begin{displaymath}
  \left[
    \begin{array}{cc}
      S_{j-1\rovlp}^{\langle} & A_{j-1\rovlp j\thetaint}\\
      A_{j\thetaint j-1\rovlp} & A_{j\thetaint}\vphantom{_{j\langle}^{(n-\frac{1}{2})}}
    \end{array}
  \right]
  \left[
    \begin{array}{c}
      \mathbf{v}_{j-1\rovlp} \\ \mathbf{v}_{j\thetaint}\vphantom{_{j\langle}^{(n-\frac{1}{2})}}
    \end{array}
  \right] =
  \left[
    \begin{array}{c}
      0 \\ -A_{\ctilde j+1,j+1\langle}\mathbf{u}_{j[}^{(\frac{1}{2})}
    \end{array}
  \right],
  \end{displaymath}
  which is also satisfied by the restriction of $\mathbf{u}_{j}^{(\frac{1}{2})}$
  because of \Cref{eq:sj0}, $\mathbf{f}\equiv0$ and
  \Cref{alg:DOSMMAT}. By \Cref{lem:schur}, we have that
  $\mathbf{v}_{j-1\rovlp}=\mathbf{u}_{j\lovlp}^{(\frac{1}{2})}$ and
  $\mathbf{v}_{j\thetaint}=\mathbf{u}_{j\thetaint}^{(\frac{1}{2})}.$ In
  \Cref{alg:DOSMMAT}, the r.h.s. for
  $\mathbf{u}_{j+1\langle}^{(\frac{1}{2})}$ then becomes
  \begin{displaymath}
  \begin{array}{ll}
    &-A_{j+1\langle j\thetaint}\mathbf{u}_{j\thetaint}^{(\frac{1}{2})}
    -A_{j+1\langle,\ctilde j+1}A_{\ctilde j+1}^{-1}A_{\ctilde j+1,j+1\langle}
    \mathbf{u}_{j[}^{(\frac{1}{2})}\\
    &\qquad =-A_{j+1\langle j\thetaint}\mathbf{u}_{j\thetaint}^{(\frac{1}{2})}
    +A_{j+1\langle,\ctilde j+1}\mathbf{v}_{\ctilde j+1}\\
    &\qquad =-A_{j+1\langle j\thetaint}\mathbf{u}_{j\thetaint}^{(\frac{1}{2})}
    +A_{j+1\langle j\thetaint}\mathbf{v}_{j\thetaint}\\
    &\qquad=-A_{j+1\langle j\thetaint}\mathbf{u}_{j\thetaint}^{(\frac{1}{2})}
    +A_{j+1\langle j\thetaint}\mathbf{u}_{j\thetaint}^{(\frac{1}{2})}=0.
  \end{array}
  \end{displaymath}
  Now we know \Cref{eq:sj0} holds for $j=2,..,J-1$ and $j=J$ except
  that we write $\mathbf{u}_J^{(1)}$ instead of
  $\mathbf{u}_J^{(\frac{1}{2})},$ and we see that the r.h.s. for
  $\mathbf{u}_J^{(1)}$ vanishes, so $\mathbf{u}_J^{(1)}=0.$ In the
  backward sweep, the r.h.s. on $\Gamma_{j,j-1}$ for
  $\mathbf{u}_j^{(1)}$ is by \Cref{alg:DOSMMAT} the same as for
  $\mathbf{u}_j^{(\frac{1}{2})}$ and hence is zero by
  \Cref{eq:sj0}, and the r.h.s. on
  $\Gamma_{j,j+1}$ vanishes, given $\mathbf{u}_{j+1}^{(1)}=0.$ By
  induction, we thus conclude that $\mathbf{u}_j^{(1)}=0,$ $j=1,..,J.$
\end{proof}

\Cref{alg:DOSMPDE} and \Cref{alg:DOSMMAT} use the subdomain
approximations as iterates.  If we want to have a global approximation
for the original problem as a final result, we can just glue subdomain
approximations after stopping the iteration. This can be done by
setting $u^{(n)}:=\sum_{j=1}^J\mathcal{E}_j(\phi_ju_j^{(n)})$ at the
PDE level and
$\mathbf{u}^{(n)}=\sum_{j=1}^JR_j^T\Phi_j\mathbf{u}_j^{(n)}$ at the
matrix level, where $\mathcal{E}_j$ is the extension by zero from
$\Omega_j$ to $\Omega,$ and $\phi_j$ is a weighting function and
correspondingly $\Phi_j$ a diagonal matrix. For consistency, when
$u_j^{(n)}=u|_{\Omega_j}$ with $u$ the solution of the original
problem, we want $u^{(n)}=u$, or
$\sum_{j=1}^J\mathcal{E}_j(\phi_ju|_{\Omega_j})=\sum_{j=1}^J(\mathcal{E}_j\phi_j)u=u.$
To ensure this for arbitrary data of \Cref{eq:pde}, we must have
$\sum_{j=1}^J\mathcal{E}_j\phi_j\equiv1.$ At the matrix level, we must
have $\sum_{j=1}^JR_j^T\Phi_jR_j=I.$ In particular, for
a non-overlapping decomposition, we must have $\phi_j\equiv1$ in $\Omega_j.$

\subsection{Global deferred correction form of Schwarz methods}

If we want to use global approximations as iterates, i.e. input the
last iterate $u^{(n-1)}$ to DOSM and get $u^{(n)}$ as output, we need
to be very careful with the weighting functions introduced in the last paragraph.
This is because
\Cref{alg:DOSMPDE} relies essentially on the interface data, and when
inputting $u^{(n-1)}$ instead of $\{u_j^{(n-1)}\}$, we must ensure
that the values of $\{u_j^{(n-1)}\}$ necessary for the evaluation of
the interface data in \Cref{alg:DOSMPDE} can still be found in
$u^{(n-1)}$. We thus need a variant of \Cref{alg:DOSMPDE} that
generates the iterate $u^{(n)}$ without storing $u_j^{(n-1)}$, just
storing $u^{(n-1)}$, and satisfying
$u^{(n)}=\sum_{j=1}^J\mathcal{E}_j\phi_ju_j^{(n)}$, given that this
relation holds for $n=0$. The result of this algorithm should be
identical to the glued result from the iterates $\{u_j^{(n)}\}$
generated by \Cref{alg:DOSMPDE}. The equivalence of the new variant to
\Cref{alg:DOSMPDE} and \Cref{alg:DOSMMAT} at the discrete level is
important, because then \Cref{thm:1steppde}, and \Cref{thm:1stepmat} at
the discrete level will also hold for the variant. We present the DOSM
version with global approximations as iterates in \Cref{alg:GDCPDE} at
\begin{algorithm}
  \caption{DOSM in the {\bf global deferred correction} form at the {\bf PDE}
    level}
  \label{alg:GDCPDE}
  Given the last iterate $u^{(n-1)}$, solve successively for $j=1,..,J-1,$
  \begin{displaymath}
    \small
  \begin{array}{r@{\hspace{0.2em}}c@{\hspace{0.2em}}ll}
    \mathcal{L}\,v_j^{(n-\frac{1}{2})}&=&f-\mathcal{L}\,u^{(n-1+\frac{j-1}{2J-1})}
    & \mbox{ in }\Omega_j,\\
    \mathcal{B}\,v_j^{(n-\frac{1}{2})}&=&g & \mbox{ on }
    \partial\Omega\cap\partial\Omega_j,\\
    \mathcal{Q}_{j}^{\langle}\left(\mathbf{n}_j^T\alpha\nabla
      v_j^{(n-\frac{1}{2})}\right)+\mathcal{P}_{j}^{\langle}v_j^{(n-\frac{1}{2})}&=&
    0 & \mbox{ on }\Gamma_{j,j-1},\\
    \mathcal{Q}_{j}^{\rangle}\left(\mathbf{n}_j^T\alpha\nabla
      v_j^{(n-\frac{1}{2})}\right)+\mathcal{P}_{j}^{\rangle}v_j^{(n-\frac{1}{2})}&=&
    0 & \mbox{ on }\Gamma_{j,j+1},
  \end{array}
  \end{displaymath}
  and each solve is followed by
  $u^{(n-1+\frac{j}{2J-1})}\gets u^{(n-1+\frac{j-1}{2J-1})}+\mathcal{E}_j(\phi_j
  v_j^{(n-\frac{1}{2})})$ with $\mathcal{E}_j\phi_j$ forming a
    non-overlapping partition of unity of $\Omega$, i.e. $\phi_j=1$ in its
  support contained in $\overline\Omega_j,$
  $(\mathcal{E}_j\phi_j)(\mathcal{E}_l\phi_l)\equiv0$ for $j\neq l$ and
  $\sum_{j=1}^J\mathcal{E}_j\phi_j\equiv1$ in $\Omega$.

  Then, in the backward sweep solve successively for $j=J,..,1,$
  \begin{displaymath}
    \small
  \begin{array}{r@{\hspace{0.2em}}c@{\hspace{0.2em}}ll}
    \mathcal{L}\,v_j^{(n)}&=&f-\mathcal{L}\,u^{(n-1+\frac{2J-j-1}{2J-1})}
    & \mbox{ in }\Omega_j,\\
    \mathcal{B}\,v_j^{(n)}&=&g & \mbox{ on }
    \partial\Omega\cap\partial\Omega_j,\\
    \mathcal{Q}_{j}^{\langle}\left(\mathbf{n}_j^T\alpha\nabla
      v_j^{(n)}\right)+\mathcal{P}_{j}^{\langle}v_j^{(n)}&=&
    0 & \mbox{ on }\Gamma_{j,j-1},\\
    \mathcal{Q}_{j}^{\rangle}\left(\mathbf{n}_j^T\alpha\nabla
      v_j^{(n)}\right)+\mathcal{P}_{j}^{\rangle}v_j^{(n)}&=&
    0 & \mbox{ on }\Gamma_{j,j+1},
  \end{array}
  \end{displaymath}
  and each solve is followed by $u^{(n-1+\frac{2J-j}{2J-1})}\gets
  u^{(n-1+\frac{2J-j-1}{2J-1})}+\mathcal{E}_j(\phi_j v_j^{(n)}).$
\end{algorithm}
the continuous level, and in \Cref{alg:GDCMAT} for the discrete case.
\begin{algorithm}
  \caption{DOSM in the {\bf global deferred correction} form at the {\bf matrix}
    level}
  \label{alg:GDCMAT}
  Given the last iterate $\mathbf{u}^{(n-1)},$ solve successively for
  $j=1,..,J-1,$ (note that $\mathbf{v}_j$ changes with $n$)
  \begin{equation}\label{eq:gdcmat}
    \arraycolsep0.3em
    \small
    \left[
      \begin{array}{ccc}
        Q_{j\langle}^{\langle}A_{j\langle}^{\langle}+P_{j\langle}^{\langle} &
        Q_{j\langle}^{\langle}A_{j\langle j\bullet} &\\
        A_{j\bullet j\langle} & A_{j\bullet} & A_{j\bullet j\rangle}\\
        & Q_{j\rangle}^{\rangle}A_{j\rangle j\bullet} &
        Q_{j\rangle}^{\rangle}A_{j\rangle}^{\rangle}+P_{j\rangle}^{\rangle}
      \end{array}
    \right]
    \left[
      \begin{array}{c}
        \mathbf{v}_{j\langle}\vphantom{_{j\rangle}^{\rangle}} \\
        \mathbf{v}_{j\bullet} \\
        \mathbf{v}_{j\rangle}\vphantom{_{j\rangle}^{\rangle}}
      \end{array}
    \right]=
    \left[
      \begin{array}{c}
        Q_{j\langle}^{\langle}\\
        I_{j\bullet}\\
        Q_{j\rangle}^{\rangle}
      \end{array}
    \right]
    R_j\left(\mathbf{f}-A\mathbf{u}^{(n-1+\frac{j-1}{2J-1})}\right),
  \end{equation}
  and each solve is followed by
  \begin{equation}\label{eq:gdcmat2}
    \mathbf{u}^{(n-1+\frac{j}{2J-1})}\gets
    \mathbf{u}^{(n-1+\frac{j-1}{2J-1})}+R_j^T\Phi_j\mathbf{v}_j
  \end{equation}
  with the constraints: $\Phi_j$ is a diagonal matrix with its diagonal ones on
  $\Theta_j\cup\Gamma_{j+1,j}$ and zeros on $\Gamma_{j,j-1},$
  $(R_j^T\Phi_j)R_j(R_l^T\Phi_l)=0$ as $j\neq l$ and
  $\sum_{j=1}^JR_j^T\Phi_jR_j=I_{1,..,J}.$

  Then solve successively for $j=J,..,1,$
  \begin{displaymath}
    \arraycolsep0.3em
    \small
    \left[
      \begin{array}{ccc}
        Q_{j\langle}^{\langle}A_{j\langle}^{\langle}+P_{j\langle}^{\langle} &
        Q_{j\langle}^{\langle}A_{j\langle j\bullet} &\\
        A_{j\bullet j\langle} & A_{j\bullet} & A_{j\bullet j\rangle}\\
        & Q_{j\rangle}^{\rangle}A_{j\rangle j\bullet} &
        Q_{j\rangle}^{\rangle}A_{j\rangle}^{\rangle}+P_{j\rangle}^{\rangle}
      \end{array}
    \right]
    \left[
      \begin{array}{c}
        \mathbf{v}_{j\langle}\vphantom{_{j\rangle}^{\rangle}} \\
        \mathbf{v}_{j\bullet} \\
        \mathbf{v}_{j\rangle}\vphantom{_{j\rangle}^{\rangle}}
      \end{array}
    \right]=
    \left[
      \begin{array}{c}
        Q_{j\langle}^{\langle}\\
        I_{j\bullet}\\
        Q_{j\rangle}^{\rangle}
      \end{array}
    \right]
    R_j\left(\mathbf{f}-A\mathbf{u}^{(n-1+\frac{2J-j-1}{2J-1})}\right),
  \end{displaymath}
  and each solve is followed by $\mathbf{u}^{(n-1+\frac{j}{2J-1})}\gets
  \mathbf{u}^{(n-1+\frac{j-1}{2J-1})}+R_j^T\Phi_j\mathbf{v}_j.$ This time
  the diagonal matrix $\Phi_j$ has diagonal values ones on
  $\Theta_j\cup\Gamma_{j-1,j}$ and zeros on $\Gamma_{j,j+1},$ and the last two
  constraints are the same as before.
\end{algorithm}
For the parallel form of the algorithms, i.e. POSM, the situation is
even more delicate and has been studied at length in \cite{Efstathiou,
  St-Cyr07}. For example, the well known preconditioner
\cite{Dryja:1987:AVS}, called the {\it additive Schwarz method} (AS),
is designed to be symmetric but then loses the equivalence to POSM.
AS can also not be used as a standalone iterative method, since it is not
convergent, for a discussion, see \cite{Efstathiou,gander2008schwarz}.

\begin{remark}
  \Cref{alg:GDCPDE} and \Cref{alg:GDCMAT} require the weighting
  functions to take values either 0 or 1, and to constitute a
  partition of unity. The resulting way of gluing subdomain
  approximations into a global approximation was introduced in
  \cite{Cai99} under the name {\it restricted additive Schwarz method}
  (RAS), and made the method equivalent to the underlying parallel
  Schwarz method, but at the price of sacrificing symmetry. The
  restricted weighting never adds two subdomain approximations at the
  same location so that the current subdomain approximation can be
  subtracted and updated through a correction. One can then use a
  global deferred correction at every substep. One could also evaluate
  the global residual at the beginning of the $n$-th iteration and do
  the gluing and the global correction at the end of the $n$-th
  iteration, while carrying out the intermediate substeps in the subdomain
  transmission form or a local deferred correction form (see
  \Cref{sec:stddm}).  Then, weighting functions from the backward
  sweep of \Cref{alg:GDCMAT} without the restricted constraint
  $(R_j^T\Phi_j)R_j(R_l^T\Phi_l)=0$ could be used for gluing the
  global correction.
\end{remark}

\begin{remark}\label{rem:OSHM}
  \Cref{alg:GDCPDE} (\Cref{alg:GDCMAT}) uses a restricted extension
  $\mathcal{E}_j\phi_j$ ($R_j^T\Phi_j$) of the local approximations,
  but a full restriction $\cdot|_{\Omega_j}$ ($R_j$) of the global
  residuals.  A variant of each algorithm can be obtained by using a
  restricted restriction $\cdot|_{\Omega_j}\phi_j$ ($\Phi_jR_j$) of
  the global residuals but a full extension $\mathcal{E}_j$ ($R_j^T$)
  of the local approximations.  For example, in \Cref{alg:GDCMAT},
  $\Phi_j$ could be moved from the right of $R_j^T$ in
  \Cref{eq:gdcmat2} to the left of $R_j$ in \Cref{eq:gdcmat}.  This
  idea was first introduced in \cite{Cai99} to transform RAS into
    the {\it additive Schwarz method with harmonic extension} (ASH),
  and later adopted and studied in \cite{Kwok} for optimized Schwarz
  methods with overlap.  In \cite{Kwok}, a close relation was proved
  between the iterates of optimized ASH and POSM.  Moreover, the
  harmonic extension variant of \Cref{alg:GDCPDE} and
  \Cref{alg:GDCMAT} can be shown to converge in one step under the
  assumptions of \Cref{thm:1steppde} and \Cref{thm:1stepmat}.  When
  the coefficient matrices of the original problem and the subproblems
  are complex symmetric (i.e.  $A^T=A$), then RAS and ASH lead to
  preconditioned systems that are the transpose of each other, and
  hence they have the same spectra.
\end{remark}

\begin{theorem}\label{thm:GDCPDE}
  Suppose the subproblems in \Cref{alg:GDCPDE} are well-posed.  For
  an overlapping decomposition, if $\phi_j$ of \Cref{alg:GDCPDE}
  satisfies also $\mathbf{n}_j^T\alpha\nabla\phi_j=0$ on
  $\Gamma_{j\pm1,j},$ $j=1,..,J,$ then \Cref{alg:GDCPDE} is equivalent
  to \Cref{alg:DOSMPDE}.  That is, given
  $u^{(0)}=\sum_{j=1}^J\mathcal{E}_j(\phi_ju_j^{(0)})$, the iterates
  generated by the two algorithms satisfy for all $n$ that
  $u^{(n)}=\sum_{l=1}^{J}\mathcal{E}_l(\phi_lu_l^{(n)})$.
\end{theorem}
\begin{proof}
  Suppose $u^{(0)}=\sum_{j=1}^J\mathcal{E}_j(\phi_ju_j^{(0)})$ and the
  sequence $u^{(*)}$ is generated by \Cref{alg:GDCPDE}, and the
  sequence $u_j^{(*)}$ is generated by \Cref{alg:DOSMPDE}. Assuming
  that
  $u^{(n-1+\frac{j-1}{2J-1})}=\sum_{l=1}^{j-1}\mathcal{E}_l(\phi_lu_l^{(n-\frac{1}{2})})
  +\sum_{l=j}^{J}\mathcal{E}_l(\phi_lu_l^{(n-1)}),$ we will show that
  the same relation also holds for $j+1$ instead of $j$. First, we
  have $u^{(n-1+\frac{j-1}{2J-1})}=u_{j-1}^{(n-\frac{1}{2})}$ on
  $\Gamma_{j,j-1}$ because $\phi_{j-1}=1$ in $\Theta_{j-1}$ and
  $\phi_j$ is compactly supported in $\Omega_j$. Second, we have
  $\mathbf{n}_j^T\alpha\nabla u^{(n-1+\frac{j-1}{2J-1})} =
  \mathbf{n}_j^T\alpha\nabla u_{j-1}^{(n-\frac{1}{2})}$ on
  $\Gamma_{j,j-1}$ because $\phi_{j-1}=1$ in $\Theta_{j-1},$ $\phi_j$
  is compactly supported in $\Omega_j,$ $\phi_{j-1}$ is smooth and
  $\mathbf{n}_j^T\alpha\nabla\phi_{j-1}=0$ on $\Gamma_{j,j-1}.$
  Combining the two traces, we have that the transmission condition on
  $\Gamma_{j,j-1}$ in \Cref{alg:DOSMPDE} is equivalent to
  \begin{displaymath}
    \mathcal{Q}_{j}^{\langle}\left(\mathbf{n}_j^T\alpha\nabla
      u_{j}^{(n-\frac{1}{2})}\right)
    +\mathcal{P}_{j}^{\langle}u_{j}^{(n-\frac{1}{2})} =
    \mathcal{Q}_{j}^{\langle}\left(\mathbf{n}_j^T\alpha\nabla
      u^{(n-1+\frac{j-1}{2J-1})}\right)
    +\mathcal{P}_{j}^{\langle}u^{(n-1+\frac{j-1}{2J-1})}.
  \end{displaymath}
  By the same argument, the other transmission condition in the forward sweep
  can also be rewritten using $u^{(n-1+\frac{j-1}{2J-1})}.$ We introduce
  $\tilde{v}_j:=u_{j}^{(n-\frac{1}{2})}-u^{(n-1+\frac{j-1}{2J-1})}|_{\Omega_j}$
  and we find from \Cref{alg:DOSMPDE} that $\tilde{v}_j$ solves the
  forward sweeping problem of \Cref{alg:GDCPDE}. Hence,
  $v_j^{(n-\frac{1}{2})}=\tilde{v}_j=u_{j}^{(n-\frac{1}{2})}-u^{(n-1+\frac{j-1}{2J-1})}|_{\Omega_j}.$
  By this equation, \Cref{alg:GDCPDE} and our assumptions, we obtain
  \begin{displaymath}
  \begin{array}{ll}
    & u^{(n-1+\frac{j}{2J-1})}= u^{(n-1+\frac{j-1}{2J-1})}+
      \mathcal{E}_j(\phi_jv_j^{(n-\frac{1}{2})})\\
    &\qquad= (1-\mathcal{E}_j\phi_j)u^{(n-1+\frac{j-1}{2J-1})}+
    \mathcal{E}_j(\phi_ju_{j}^{(n-\frac{1}{2})})\\
    &\qquad= \sum_{s\neq j}\mathcal{E}_s\phi_s
    \left(\sum_{l=1}^{j-1}\mathcal{E}_l(\phi_lu_l^{(n-\frac{1}{2})})
      +\sum_{l=j}^{J}\mathcal{E}_l(\phi_lu_l^{(n-1)})\right)
    +\mathcal{E}_j(\phi_ju_{j}^{(n-\frac{1}{2})})\\
    &\qquad=\sum_{l=1}^{j}\mathcal{E}_l(\phi_lu_l^{(n-\frac{1}{2})})
    +\sum_{l=j+1}^{J}\mathcal{E}_l(\phi_lu_l^{(n-1)}).
  \end{array}
  \end{displaymath}
  By induction, we know this holds for all $j=1,..,J-1.$ Similarly, we can
  prove for the backward sweep
  \begin{displaymath}
  u^{(n-1+\frac{2J-j}{2J-1})}=
  \sum_{l=1}^{j-1}\mathcal{E}_l(\phi_lu_l^{(n-\frac{1}{2})})
  +\sum_{l=j}^{J}\mathcal{E}_l(\phi_lu_l^{(n)}),\,\mbox{ for all }
  j=J,..,1.
  \end{displaymath}
  In particular, $u^{(n)}=\sum_{l=1}^{J}\mathcal{E}_l(\phi_lu_l^{(n)})$, and
  the result follows by induction.
\end{proof}

\begin{remark}
  The assumption that $\mathbf{n}_j^T\alpha\nabla\phi_j=0$ on
  $\Gamma_{j\pm1,j}$ was first introduced in \cite{Chen13a}. We think
  this assumption can be removed from \Cref{thm:GDCPDE}. In the matrix
  version (see \Cref{thm:GDCMAT}) we find no counterpart of this
  assumption. For the same reason, we think that \Cref{thm:GDCPDE}
  also holds for non-overlapping decompositions.  There are
  however some difficulties for the justification at the PDE
  level: as noted in \cite{Stolk}, ``$u^{(*)}$ is generally
      discontinuous at the interface $\Gamma_{j,j-1}$, which results
      in a most singular residual, and the restriction of the residual
      to the subdomain $\Omega_j$ is not well defined''.
\end{remark}

\begin{theorem}\label{thm:GDCMAT}
  If the subproblems in \Cref{alg:GDCMAT} are well-posed, \Cref{alg:GDCMAT} is
  equivalent to \Cref{alg:DOSMMAT}.  That is, given
  $\mathbf{u}^{(0)}=\sum_{j=1}^JR_j^T\Phi_j\mathbf{u}_j^{(0)}$, the iterates
  generated by the two algorithms satisfy
  $\mathbf{u}^{(n)}=\sum_{l=1}^{J}R_l^T\Phi_l\mathbf{u}_l^{(n)}$ for all $n$.
\end{theorem}

\begin{proof}
  Suppose $\mathbf{u}^{(0)}=\sum_{j=1}^JR_j^T\Phi_j\mathbf{u}_j^{(0)}$
  and the iterates $\mathbf{u}^{(*)}$ are generated by
  \Cref{alg:GDCMAT}, and the iterates $\mathbf{u}_j^{*}$,
  $j=1,..,J$ by \Cref{alg:DOSMMAT}. Assuming that
  \begin{equation}\label{Ass13}
  \mathbf{u}^{(n-1+\frac{j-1}{2J-1})}=\sum_{l=1}^{j-1}R_l^T\Phi_l
  \mathbf{u}_l^{(n-\frac{1}{2})}+
  \sum_{l=j}^{J}R_l^T\Phi_l\mathbf{u}_l^{(n-1)},
  \end{equation}
  we will show that the same relation also holds for $j+1$
  instead of $j$. By \Cref{Ass13} and the constraints on $\Phi_l$,
  we have
  \begin{displaymath}
  \begin{array}{ll}
    R_{j\langle}\mathbf{u}^{(n-1+\frac{j-1}{2J-1})}
    =\mathbf{u}_{j-1[}^{(n-\frac{1}{2})},&
    R_{j-1\thetaint}\mathbf{u}^{(n-1+\frac{j-1}{2J-1})}
    =\mathbf{u}_{j-1\thetaint}^{(n-\frac{1}{2})},\\
    R_{j\rangle}\mathbf{u}^{(n-1+\frac{j-1}{2J-1})}
    =\mathbf{u}_{j+1]}^{(n-1)},&
    R_{j+1\thetaint}\mathbf{u}^{(n-1+\frac{j-1}{2J-1})}
    =\mathbf{u}_{j+1\thetaint}^{(n-1)}.
  \end{array}
  \end{displaymath}
  Substituting these into the forward sweep problem in
  \Cref{alg:DOSMMAT}, we find
  $\mathbf{u}_j^{(n-\frac{1}{2})}-R_j\mathbf{u}^{(n-1+\frac{j-1}{2J-1})}$ solves
  the forward sweep problem of \Cref{alg:GDCMAT},
  i.e. $\mathbf{v}_j=\mathbf{u}_j^{(n-\frac{1}{2})}-R_j\mathbf{u}^{(n-1+\frac{j-1}{2J-1})}.$
  Substituting this into the update relation of \Cref{alg:GDCMAT} and
  using the constraints on $\Phi_l,$ we obtain
  \begin{displaymath}
  \begin{array}{ll}
    &\mathbf{u}^{(n-1+\frac{j}{2J-1})}=(I-R_j^T\Phi_jR_j)
      \mathbf{u}^{(n-1+\frac{j-1}{2J-1})}+
      R_j^T\Phi_j\mathbf{u}_j^{(n-\frac{1}{2})}\\
    &\qquad=\sum_{s\neq j}R_s^T\Phi_sR_s\left(\sum_{l=1}^{j-1}R_l^T
       \Phi_l\mathbf{u}_l^{(n-\frac{1}{2})}
       +\sum_{l=j}^{J}R_l^T\Phi_l\mathbf{u}_l^{(n-1)}\right)
       +R_j^T\Phi_j\mathbf{u}_j^{(n-\frac{1}{2})}\\
    &\qquad=\sum_{l=1}^{j}R_l^T\Phi_l\mathbf{u}_l^{(n-\frac{1}{2})}
       +\sum_{l=j+1}^{J}R_l^T\Phi_l\mathbf{u}_l^{(n-1)}.
  \end{array}
  \end{displaymath}
  By induction, we know this holds for all $j=1,..,J-1$. Similarly, the
  backward sweep gives 
  \begin{displaymath}
  \mathbf{u}^{(n-1+\frac{2J-j}{2J-1})}=
  \sum_{l=1}^{j-1}R_l^T\Phi_l\mathbf{u}_l^{(n-\frac{1}{2})}
  +\sum_{l=j}^{J}R_l^T\Phi_l\mathbf{u}_l^{(n)},\quad\mbox{ for all }j=J,..,1.
  \end{displaymath}
  In particular, we have
  $\mathbf{u}^{(n)}=\sum_{l=1}^{J}R_l^T\Phi_l\mathbf{u}_l^{(n)}$, and
  the result follows by induction.
\end{proof}

\begin{remark}
  DOSM can be used also as a preconditioner for the original problem.
  The DOSM preconditioner as a linear operator acts on a given
  r.h.s. and outputs an approximate solution by one or multiple steps
  of the DOSM iteration with {\it zero} initial guess. As
  we showed above, choosing an appropriate gluing scheme is very
  important for the preconditioned Richardson iteration
  $\mathbf{u}^{(l)}=
  \mathbf{u}^{(l-1)}+M^{-1}(\mathbf{f}-A\mathbf{u}^{(l-1)})$ to be
  equivalent to the DOSM in the subdomain transmission form.  If we use
  the preconditioner with a Krylov method, e.g. GMRES, the convergence
  can not be slower than the Richardson iteration.  Hence, the
  equivalence tells us that if the optimal DOSM from \Cref{thm:1stepmat} is
  used as a preconditioner for a Krylov method, the latter also
  converges in one step.  This is not the case for the gluing scheme
  of AS which uses for the weights all ones on each subdomain. An
  advantage of using DOSM as a {\it preconditioner} is also that even if one
  uses inexact subdomain solves, i.e. $\tilde{M}^{-1}\approx M^{-1}$,
  the global consistency is retained, i.e. the converged iterate is
  always the solution of the original problem, while the {\it plain} DOSM
  iterates have a consistency error.
\end{remark}

\begin{remark}
  From the subdomain transmission form (\Cref{alg:DOSMPDE} or
  \Cref{alg:DOSMMAT}) to the deferred correction form
  (\Cref{alg:GDCPDE} or \Cref{alg:GDCMAT}), we see that the interface
  conditions are becoming homogeneous and the evaluation of the
  r.h.s. becomes unrelated to the transmission operators $\mathcal{P},
  \mathcal{Q}$ or $P, Q.$ This can be an advantage when the action of
  those operators is expensive.  For example, the PML technique (see
  \Cref{rem:impml}) leads to the choice $\mathcal{Q}=\mathcal{I},$
  $\mathcal{P}=\mathrm{DtN}_j^{pml}$ on $\Gamma_{j,j\pm1}.$ In this
  case, the action of $\mathcal{P}$ involves solving a problem in
  $\Omega_{j}^{pml}$ which one might want to avoid.  In the deferred
  correction form, the action of $\mathcal{P}$ is not required for the
  evaluation of the r.h.s., but $\mathcal{P}$ still appears
  acting on the unknown function in the interface condition
  \begin{displaymath}
  \mathbf{n}_j^T\alpha\nabla v_j + \mathrm{DtN}_j^{pml}v_j=0\quad \mbox{on
  }\Gamma_{j,j\pm1},
  \end{displaymath}
  where we omitted the superscripts for simplicity.  For the
  implementation, one usually unfolds the PML--DtN operator and composes a
  bigger problem defined in $\Omega_j\cup\Omega_j^{pml}$ as explained in
  \Cref{rem:impml}. The first use of PML in a Schwarz method is due
  to Toselli (see \cite{Toselli}) who seemed to use the full extension
  including the PML regions for updating the global iterates, so that
  his algorithm deviates from OSM and may be interpreted as an
  overlapping Schwarz method with the PML equations used in the
  overlap. This resembles in the overlap the shifted-Laplace based 
  Schwarz method recently proposed and studied in \cite{GSV}.
\end{remark}

\begin{remark}\label{rem:schadle}
  There are many other ways of implementing the PML transmission
  conditions.  One was proposed in \cite{Schadle, SZBKS}.  First, we
  rewrite the condition as
  \begin{displaymath}
  \mathbf{n}_j^T\alpha\nabla u_j + \mathrm{DtN}_j^{pml}
  (u_j-u_{j-1})=\mathbf{n}_j^T\alpha\nabla u_{j-1}\quad\mbox{on }\Gamma_{j,j-1}.
  \end{displaymath}
  Then, we unfold $\mathrm{DtN}_j^{pml}$ and compose a coupled problem in
  $\Omega_j\cup\Omega_{j}^{pml}$ as follows:
  \begin{displaymath}
    \begin{array}{r@{\hspace{0.2em}}c@{\hspace{0.2em}}ll}
      \mathcal{L}u_j&=&f&\;\mbox{in }\Omega_j,\\
      \mathcal{B}u_j&=&g&\;\mbox{on }\partial\Omega_j\cap\partial\Omega,\\
      \mathbf{n}_j^T\alpha\nabla u_j-\mathbf{n}_j^T\alpha\nabla v
      &=&\mathbf{n}_j^T\alpha\nabla u_{j-1}
                        &\;\mbox{on }\partial\Omega_j-\partial\Omega, \\[1em]
      \widetilde{\mathcal{L}}v&=&0&\;\mbox{in }\Omega_j^{pml},\\
      \widetilde{\mathcal{B}}v&=&0&\;\mbox{on }
                                    \partial\Omega_j^{pml}-\partial\Omega_j,\\
      v-u_j&=&-u_{j-1}&\;\mbox{on }\partial\Omega_j^{pml}-\partial\Omega.
    \end{array}
  \end{displaymath}
  A straightforward discretization requires two sets of d.o.f. on the interfaces
  $\Gamma_{j,j\pm1},$ one for $u_j$ and the other for $v.$ In this way, we need
  only to solve one coupled problem in $\Omega_j\cup\Omega_j^{pml}$ and we avoid
  to solve an extra problem in $\Omega_j^{pml}.$ We can further remove the
  Dirichlet jump from the coupled problem by extending (lifting) $u_{j-1}$ from
  the interfaces into $\Omega_j^{pml}$. Let the extended function be
  $\tilde{u}_{j-1}^{j,pml}$. We change the unknown in the PML to
  $\tilde{v}:=v+\tilde{u}_{j-1}^{j,pml}$ so that the coupled problem for
  $(u_j,\tilde{v})$ enforces no jump of Dirichlet traces across the interfaces.
\end{remark}

\subsection{Substructured form of Schwarz methods}

The substructured form of OSM consists in taking interface data as
iterates (unknowns)\footnote{If exact subdomain solvers are used,
  the glued global approximation has compact residual for the original
  problem near interfaces.  Making use of this property leads to
  yet another substructured form; see \Cref{alg:spyres} later in
  \Cref{sec:stolk}.}. These iterates are substantially smaller than the
volume iterates and thus can save memory and flops for Krylov subspace
methods. This form was first introduced in \cite{NRS95} for one-way
domain decompositions like \Cref{fig:dd}.  In particular, for
Helmholtz problems, it was used in \cite{Chevalier, GMN02}.  Later,
the substructured form was generalized to the case of domain
decompositions with cross-points (points shared by three or more
subdomains), see e.g. \cite{Bendali, FMLRMB, GK2, GZ1, Loisel} for
various approaches (some are called FETI-2LM).  Here, we consider only
the sequential one-way domain decomposition from \Cref{fig:dd}.  The
substructured form of DOSM is given in \Cref{alg:ISPDE} at the PDE
level, and in \Cref{alg:ISMAT} at the matrix level.  
\begin{algorithm}
  \caption{DOSM in the {\bf substructured} form at the {\bf PDE} level
    }
  \label{alg:ISPDE}
  Given the last iterate $\lambda_{\rangle}^{(n-1)}=\left\{\lambda_{j\rangle}^{(n-1)}\mbox{
      on }\Gamma_{j,j+1},\ j=1,..,J-1\right\}$, solve successively for
  $j=1,..,J-1,$
  \begin{displaymath}
    \small
  \begin{array}{r@{\hspace{0.2em}}c@{\hspace{0.2em}}ll}
    \mathcal{L}\,u_j^{(n-\frac{1}{2})}&=&f & \mbox{ in }\Omega_j,\\
    \mathcal{B}\,u_j^{(n-\frac{1}{2})}&=&g & \mbox{ on }
    \partial\Omega\cap\partial\Omega_j,\\
    \mathcal{Q}_{j}^{\langle}\left(\mathbf{n}_j^T\alpha\nabla
      u_j^{(n-\frac{1}{2})}\right)+\mathcal{P}_{j}^{\langle}u_j^{(n-\frac{1}{2})}&=&
    \lambda_{j\langle}^{(n)} & \mbox{ on }\Gamma_{j,j-1},\\
    \mathcal{Q}_{j}^{\rangle}\left(\mathbf{n}_j^T\alpha\nabla
      u_j^{(n-\frac{1}{2})}\right)+\mathcal{P}_{j}^{\rangle}u_j^{(n-\frac{1}{2})}&=&
    \lambda_{j\rangle}^{(n-1)} & \mbox{ on }\Gamma_{j,j+1},
  \end{array}
  \end{displaymath}
  and each solve is followed by
  $\lambda_{j+1\langle}^{(n)}\gets\mathcal{Q}_{j+1}^{\langle}\left(\mathbf{n}_{j+1}
    ^T\alpha\nabla u_{j}^{(n-\frac{1}{2})}\right)
  +\mathcal{P}_{j+1}^{\langle}u_{j}^{(n-\frac{1}{2})}.$

  Then solve successively for $j=J,..,1,$
  \begin{displaymath}
    \small
  \begin{array}{r@{\hspace{0.2em}}c@{\hspace{0.2em}}ll}
    \mathcal{L}\,u_j^{(n)}&=&f & \mbox{ in }\Omega_j,\\
    \mathcal{B}\,u_j^{(n)}&=&g & \mbox{ on }
    \partial\Omega\cap\partial\Omega_j,\\
    \mathcal{Q}_{j}^{\langle}\left(\mathbf{n}_j^T\alpha\nabla
      u_j^{(n)}\right)+\mathcal{P}_{j}^{\langle}u_j^{(n)}&=&
    \lambda_{j\langle}^{(n)} & \mbox{ on }\Gamma_{j,j-1},\\
    \mathcal{Q}_{j}^{\rangle}\left(\mathbf{n}_j^T\alpha\nabla
      u_j^{(n)}\right)+\mathcal{P}_{j}^{\rangle}u_j^{(n)}&=&
    \lambda_{j\rangle}^{(n)} & \mbox{ on }\Gamma_{j,j+1},
  \end{array}
  \end{displaymath}
  and each solve is followed by $\lambda_{j-1\rangle}^{(n)}\gets
  \mathcal{Q}_{j-1}^{\rangle}\left(\mathbf{n}_{j-1}^T\alpha\nabla
    u_{j}^{(n)}\right)+\mathcal{P}_{j-1}^{\rangle}u_{j}^{(n)}$.

  We obtain $\lambda_{\rangle}^{(n)}\gets\left\{\lambda_{j\rangle}^{(n)}\mbox{ on
    }\Gamma_{j,j+1},\ j=1,..,J-1\right\}$.
\end{algorithm}
\begin{algorithm}
  \caption{DOSM in the {\bf substructured} form at the {\bf matrix}
    level}
  \label{alg:ISMAT}
  Given the last iterate
  $\small\boldsymbol\lambda_{\rangle}^{(n-1)}=\left\{\boldsymbol{\lambda}_{j\rangle}^{(n-1)},
    j=1,..,J-1\right\}$, solve successively for $j=1,..,J-1,$
  \begin{displaymath}
    \small
    \left[
      \begin{array}{ccc}
        Q_{j\langle}^{\langle}A_{j\langle}^{\langle}+P_{j\langle}^{\langle} &
        Q_{j\langle}^{\langle}A_{j\langle j\bullet} &\vphantom{_{j\langle}^{(n-\frac{1}{2})}}\\
        A_{j\bullet j\langle} & A_{j\bullet} & A_{j\bullet j\rangle}\\
        & Q_{j\rangle}^{\rangle}A_{j\rangle j\bullet} &
        Q_{j\rangle}^{\rangle}A_{j\rangle}^{\rangle}+P_{j\rangle}^{\rangle}\vphantom{_{j\langle}^{(n-\frac{1}{2})}}
      \end{array}
    \right]
    \left[
      \begin{array}{c}
        \mathbf{u}_{j\langle}^{(n-\frac{1}{2})} \\
        \mathbf{u}_{j\bullet}^{(n-\frac{1}{2})} \\
        \mathbf{u}_{j\rangle}^{(n-\frac{1}{2})}
      \end{array}
    \right]=
    \left[
      \begin{array}{c}
        Q_{j\langle}^{\langle}\mathbf{f}_{j\langle}^{\langle}
        +\boldsymbol\lambda_{j\langle}^{(n)}\vphantom{_{j\langle}^{(n-\frac{1}{2})}}\\
        \mathbf{f}_{j\bullet} \\
        Q_{j\rangle}^{\rangle}\mathbf{f}_{j\rangle}^{\rangle}+\boldsymbol\lambda_{j\rangle}^{(n-1)}\vphantom{_{j\langle}^{(n-\frac{1}{2})}}
      \end{array}
    \right],
  \end{displaymath}
  and each solve is followed by
  \begin{displaymath}
    \small
    \boldsymbol\lambda_{j+1\langle}^{(n)}\gets
    Q_{j+1\langle}^{\langle}\left(\mathbf{f}_{j+1\langle}^{\rangle} -
      A_{j+1\langle j\thetaint} \mathbf{u}_{j\thetaint}^{(n-\frac{1}{2})}\right)
    +\left(P_{j+1\langle}^{\langle}-Q_{j+1\langle}^{\langle}A_{j+1\langle}^{\rangle}\right)
    \mathbf{u}_{j[}^{(n-\frac{1}{2})}.
  \end{displaymath}
  Note that it does not matter what splits
  $\mathbf{f}_{j\langle}=\mathbf{f}_{j\langle}^{\rangle}+\mathbf{f}_{j\langle}^{\langle}$
  and
  $\mathbf{f}_{j\rangle}=\mathbf{f}_{j\rangle}^{\rangle}+\mathbf{f}_{j\rangle}^{\langle}$
  are used, the only difference will be the definition of the
  interface data.  For example, one can use the simple splits
  $\mathbf{f}_{j\langle}^{\rangle}=0$ and
  $\mathbf{f}_{j\rangle}^{\langle}=0.$

  Then solve successively for $j=J,..,1,$
  \begin{displaymath}
    \small
    \left[
      \begin{array}{ccc}
        Q_{j\langle}^{\langle}A_{j\langle}^{\langle}+P_{j\langle}^{\langle} &
        Q_{j\langle}^{\langle}A_{j\langle j\bullet} &\\
        A_{j\bullet j\langle} & A_{j\bullet} & A_{j\bullet j\rangle}\\
        & Q_{j\rangle}^{\rangle}A_{j\rangle j\bullet} &
        Q_{j\rangle}^{\rangle}A_{j\rangle}^{\rangle}+P_{j\rangle}^{\rangle}
      \end{array}
    \right]
    \left[
      \begin{array}{c}
        \mathbf{u}_{j\langle}^{(n)} \\
        \mathbf{u}_{j\bullet}^{(n)} \\
        \mathbf{u}_{j\rangle}^{(n)}
      \end{array}
    \right]=
    \left[
      \begin{array}{c}
        Q_{j\langle}^{\langle}\mathbf{f}_{j\langle}^{\langle}+\boldsymbol\lambda_{j\langle}^{(n)}\\
        \mathbf{f}_{j\bullet} \\
        Q_{j\rangle}^{\rangle}\mathbf{f}_{j\rangle}^{\rangle}+\boldsymbol\lambda_{j\rangle}^{(n)}
      \end{array}
    \right],
  \end{displaymath}
  and each solve is followed by
  \begin{displaymath}
    \small
    \boldsymbol\lambda_{j-1\rangle}^{(n)}\gets
    Q_{j-1\rangle}^{\rangle}\left(\mathbf{f}_{j-1\rangle}^{\langle} -
      A_{j-1\rangle j\thetaint}\mathbf{u}_{j\thetaint}^{(n)}\right)
    +\left(P_{j-1\rangle}^{\rangle}-Q_{j-1\rangle}^{\rangle}A_{j-1\rangle}^{\langle}\right)
    \mathbf{u}_{j]}^{(n)}.
  \end{displaymath}

  We obtain
  $\small\boldsymbol\lambda_{\rangle}^{(n)}=\left\{\boldsymbol{\lambda}_{j\rangle}^{(n)},
    j=1,..,J-1\right\}$.
\end{algorithm}
\Cref{thm:ISPDE} and \Cref{thm:ISMAT} give the equivalence of
the substructured formulations to the formulations with subdomain
iterates. Their proofs are simple and we thus omit them here.

\begin{theorem}\label{thm:ISPDE}
  \Cref{alg:ISPDE} is equivalent to \Cref{alg:DOSMPDE}.  That is, given
  $\lambda_{j\rangle}^{(0)}=\mathcal{Q}_{j}^{\rangle}\left(\mathbf{n}_{j}^T\alpha\nabla
    u_{j}^{(0)}\right)+\mathcal{P}_{j}^{\rangle}u_{j}^{(0)}$, the iterates
  generated by the two algorithms satisfy
  $\lambda_{j\rangle}^{(n)}=\mathcal{Q}_{j}^{\rangle}\left(\mathbf{n}_{j}^T\alpha\nabla
    u_{j}^{(n)}\right)+\mathcal{P}_{j}^{\rangle}u_{j}^{(n)}$.
\end{theorem}

\begin{theorem}\label{thm:ISMAT}
  \Cref{alg:ISMAT} is equivalent to \Cref{alg:DOSMMAT}.  That is, given
  $\boldsymbol{\lambda}_{j\rangle}^{(0)}=
  (Q_{j\rangle}^{\rangle}A_{j\rangle}^{\rangle}+P_{j\rangle}^{\rangle})\mathbf{u}_{j\rangle}^{(0)}
  +Q_{j\rangle}^{\rangle}A_{j\rangle
    j\bullet}\mathbf{u}_{j\bullet}^{(0)}-Q_{j\rangle}^{\rangle}\mathbf{f}_{j\rangle}^{\rangle}$,
  the iterates generated by the two algorithms satisfy
  $\boldsymbol{\lambda}_{j\rangle}^{(n)}=
  (Q_{j\rangle}^{\rangle}A_{j\rangle}^{\rangle}+P_{j\rangle}^{\rangle})\mathbf{u}_{j\rangle}^{(n)}
  +Q_{j\rangle}^{\rangle}A_{j\rangle
    j\bullet}\mathbf{u}_{j\bullet}^{(n)}-Q_{j\rangle}^{\rangle}\mathbf{f}_{j\rangle}^{\rangle}$.
\end{theorem}

To explain how to use Krylov acceleration for the substructured
formulations, we denote by
$\lambda_{\langle}^{(n)}:=\left\{\lambda_{j\langle}^{(n)},
j=2,..,J\right\}$.  The forward and backwards sweeps of
\Cref{alg:ISPDE} define a linear forward map $\mathcal{F}_{\langle}$
and backward map $\mathcal{F}_{\rangle}$ such that
$\lambda^{(n)}_{\langle}=\mathcal{F}_{\langle}(\lambda^{(n-1)}_{\rangle},f,g)$ and
$\lambda^{(n)}_{\rangle}=\mathcal{F}_{\rangle}(\lambda^{(n)}_{\langle},f,g)$. The
corresponding fixed point equation
$\lambda_{\rangle}=\mathcal{F}_{\rangle}(\mathcal{F}_{\langle}(\lambda_{\rangle},f,g),f,g)$
can be rewritten as a linear system,
\begin{displaymath}
\left(\mathcal{I}-\mathcal{F}_\rangle(\mathcal{F}_{\langle}(\cdot,0,0),0,0)\right)
\lambda_\rangle=\mathcal{F}_\rangle(\mathcal{F}_\langle(0,f,g),f,g),
\end{displaymath}
which can now be solved by polynomial methods e.g. Chebyshev iterations and/or
Krylov methods.

\begin{remark}\label{rem:iab}
If we look at each solve and the following update, we have a linear
map $\mathcal{F}_{j+1\langle}$ such that
$\lambda_{j+1\langle}^{(n)}=\mathcal{F}_{j+1\langle}(\lambda_{j\langle}^{(n)},\lambda_{j\rangle}^{(n-1)},f,g)$
and $\mathcal{F}_{j-1\rangle}$ such that
$\lambda_{j-1\rangle}^{(n)}=\mathcal{F}_{j-1\rangle}(\lambda_{j\langle}^{(n)},\lambda_{j\rangle}^{(n)},f,g)$.
Considering the converged solution (i.e. removing the superscripts),
we find the linear system
\begin{displaymath}\small
\left[
  \begin{array}{cccc|cccc}
    \mathfrak{1} & & & & -\mathfrak{b}_{1}^- & & & \\
    -\mathfrak{a}_1^- & \mathfrak{1} & & & & -\mathfrak{b}_2^- & &\\
    & \ddots & \ddots & & & & \ddots &\\
    & & -\mathfrak{a}_{J-2}^- & \mathfrak{1} & & & & -\mathfrak{b}_{J-1}^-\\ \hline
    -\mathfrak{b}_1^+ & & & & \mathfrak{1} & -\mathfrak{a}_1^+ & &\\
    & -\mathfrak{b}_2^+ & & & & \mathfrak{1} & \ddots &\\
    & & \ddots & & & & \ddots & -\mathfrak{a}_{J-2}^+\\
    & & & -\mathfrak{b}_{J-1}^+ & & & & \mathfrak{1}
  \end{array}
\right]
\left[
  \begin{array}{c}
    \lambda_{2\langle} \\\vspace{.1em} \lambda_{3\langle} \\ \vdots \\
    \lambda_{J\langle} \\ \hline
    \lambda_{1\rangle} \\ \vspace{.3em} \lambda_{2\rangle}\vspace{.2em} \\ \vdots \\
    \vspace{.2em}\lambda_{J-1\rangle}
  \end{array}
\right]
=
\left[
  \begin{array}{c}
    \mathfrak{g}_{2\langle}\\ \mathfrak{g}_{3\langle}\\
    \vdots\\ \mathfrak{g}_{J\langle}\\ \hline
    \mathfrak{g}_{1\rangle}\\ \vspace{.2em} \mathfrak{g}_{2\rangle}\\
    \vspace{.2em}\vdots\\ \vspace{.2em} \mathfrak{g}_{J-1\rangle}
  \end{array}
\right],
\end{displaymath}
where $\mathfrak{1}$'s are the identity operators,
$\mathfrak{a}_j^-:=\mathcal{F}_{j+2\langle}(\cdot,0,0,0)$,
$\mathfrak{b}_j^-:=\mathcal{F}_{j+1\langle}(0,\cdot,0,0)$,
$\mathfrak{a}_j^+:=\mathcal{F}_{j+1\rangle}(0,\cdot,0,0)$,
$\mathfrak{b}_j^+:=\mathcal{F}_{j+1\rangle}(\cdot,0,0,0)$ and
$\mathfrak{g}_{*}=\mathcal{F}_*(0,0,f,g)$.  If we regard the above system as a
2-by-2 block system, the block Gauss-Seidel method
(see \Cref{rem:jacobi} for block Jacobi)
leads to \Cref{alg:ISPDE} with each block solved exactly by forward or backward
substitution.  The operators $\mathfrak{a}_j$ and $\mathfrak{b}_j$ can also be
represented using Green's functions based on \Cref{eq:rep}, which
we will see in more detail in \Cref{sec:phys}.
\end{remark}

\section{AILU and sweeping preconditioners}
\label{sec:lu}

We now explain the Analytic Incomplete LU (AILU) and sweeping
preconditioners. To do so, we denote in \Cref{eq:tri} by
$\mathfrak{u}_1:=\mathbf{u}_{1\thetaint}$,
$\mathfrak{f}_1:=\mathbf{f}_{1\thetaint}$, $D_1:=A_{1\thetaint}$,
$L_1:=[A_{1\rovlp1\thetaint}; 0]$, $U_1:=[A_{1\thetaint1\rovlp},\; 0]$,
and for $j\geq2$, $\mathfrak{u}_j:=[\mathbf{u}_{j-1\rovlp};
  \mathbf{u}_{j\thetaint}]$,
$\mathfrak{f}_{j}:=[\mathbf{f}_{j-1\rovlp}; \mathbf{f}_{j\thetaint}]$,
\begin{displaymath}
  D_j :=
  \left[
    \begin{array}{cc}
      A_{j-1\rovlp} & A_{j-1\rovlp j\thetaint}\\
      A_{j\thetaint j-1\rovlp} & A_{j\thetaint}
    \end{array}
  \right],\quad
  L_{j} :=
  \left[
    \begin{array}{cc}
      0 & A_{j\rovlp j\thetaint}\\
      0 & 0
    \end{array}
  \right],\quad
  U_{j} :=
  \left[
    \begin{array}{cc}
      0 & 0\\
      A_{j\thetaint j\rovlp} & 0
    \end{array}
  \right].
\end{displaymath}
With this notation, \Cref{eq:tri} becomes \Cref{eq:trilP}. Then,
based on the factorization given in \Cref{Factors}, we can solve
\Cref{ForwardSubstitution,BackwardSubstitution} by forward and
backward substitution, which leads to \Cref{alg:lusol}.
\begin{algorithm}
  \caption{Block LU solve for the block tridiagonal system in \Cref{eq:trilP}}
  \label{alg:lusol}
  Compute $T_j$'s according to \Cref{eq:TjP}.

  {Forward sweep}: solve successively the sub-problems
  \begin{displaymath}
    \small
  \begin{array}{r@{\hspace{0.2em}}c@{\hspace{0.2em}}ll}
    T_1\mathfrak{v}_1&=&\mathfrak{f}_1,\quad &\\
    T_j\mathfrak{v}_j&=&\mathfrak{f}_j-L_{j-1}\mathfrak{v}_{j-1},\quad & j=2,\ldots,J.
  \end{array}
  \end{displaymath}

  {Backward sweep}: let $\mathfrak{u}_J\gets\mathfrak{v}_J;$ solve
  successively the sub-problems
  \begin{displaymath}
    \small
  \begin{array}{r@{\hspace{0.2em}}c@{\hspace{0.2em}}ll}
    \mbox{ \quad\ \ }T_j\mathfrak{u}_j&=&T_j\mathfrak{v}_j-U_{j}\mathfrak{u}_{j+1},\quad
    & j=J-1,\ldots,1.
  \end{array}
  \end{displaymath}
\end{algorithm}

\begin{theorem}\label{thm:lu}
  If $T_j$, $j=1,..,J$ are invertible, then \Cref{alg:lusol} is
  equivalent to \Cref{alg:DOSMMAT} with a non-overlapping
  decomposition, zero initial guess and
  $Q_{j\langle}^{\langle}=I_{j\langle}$,
  $P_{j\langle}^{\langle}=A_{j\langle}^{\rangle}-A_{j\langle, \ctilde
    j} A_{\ctilde j}^{-1} A_{\ctilde j,j\langle}$,
  $Q_{j\rangle}^{\rangle}=0$,
  $P_{j\rangle}^{\rangle}=I_{j\rangle}$. That is, the iterates
  generated by the two algorithms satisfy
  $\mathfrak{v}_j=\left[\mathbf{u}_{j\langle}^{(\frac{1}{2})};
    \mathbf{u}_{j\thetaint}^{(\frac{1}{2})}\right]$ and
  $\mathfrak{u}_j=\left[\mathbf{u}_{j\langle}^{(1)};
    \mathbf{u}_{j\thetaint}^{(1)}\right]$.
\end{theorem}

\begin{proof}
  We first show that the $T_j$'s defined by \Cref{eq:TjP} satisfy for $j\geq 2$,
  \begin{equation}\label{eq:Tj2}
    T_j = D_j - \tilde{L}_{j-1}A_{\ctilde j}^{-1}\tilde{U}_{j-1},
  \end{equation}
  where $\tilde{L}_1:=L_1$, $\tilde{U}_1:=U_1$, $\tilde{L}_{j-1}:=[0,L_{j-1}]$
  and $\tilde{U}_{j-1}:=[0;U_{j-1}]$ for $j\geq3$.  The case of $j=2$ follows
  directly by definition.  Assuming \Cref{eq:Tj2} holds for one $j$, we now show
  that it also holds for $j+1$ instead of $j$. First, by Gaussian elimination we
  have from \Cref{eq:Tj2} that
  \begin{displaymath}
  A_{\ctilde j+1}^{-1}=
  \left[
    \begin{array}{cc}
      A_{\ctilde j} & \tilde{U}_{j-1}\\
      \tilde{L}_{j-1} & D_{j}
    \end{array}
  \right]^{-1}=
  \left[
    \begin{array}{cc}
      * & *\\
      * & T_j^{-1}
    \end{array}
  \right],
  \end{displaymath}
  where $*$ represents terms not interesting to us. Therefore,
  \begin{displaymath}
  D_{j+1}-\tilde{L}_jA_{\ctilde j+1}^{-1}\tilde{U}_j
  =D_{j+1}-
  \left[
    \begin{array}{cc}
      0 & L_j
    \end{array}
  \right]
  \left[
    \begin{array}{cc}
      * & *\\
      * & T_j^{-1}
    \end{array}
  \right]
  \left[
    \begin{array}{cc}
      0 \\ U_j
    \end{array}
  \right] = T_{j+1},
  \end{displaymath}
  which is \Cref{eq:Tj2} with $j$ replaced by $j+1$. By induction, 
  \Cref{eq:Tj2} holds for all $j\geq 2$. Note that we are considering a
  non-overlapping decomposition so we can write
  \begin{displaymath}
  \tilde{L}_{j-1}=
  \left[
    \begin{array}{cc}
      A_{j\langle,\ctilde j}\\ 0
    \end{array}
  \right],\quad
  \tilde{U}_{j-1}=
  \left[
    \begin{array}{cc}
      A_{\ctilde j,j\langle} & 0
    \end{array}
  \right],\quad
  D_{j}= \left[
    \begin{array}{cc}
      A_{j\langle} & A_{j\langle j\bullet}\\
      A_{j\bullet j\langle} & A_{j\bullet}
    \end{array}
  \right].
  \end{displaymath}
  Substituting the above equations into \Cref{eq:Tj2}, we obtain for $j\geq2$,
  \begin{equation}\label{eq:Tj3}
  T_j=
  \left[
    \begin{array}{cc}
      A_{j\langle}-A_{j\langle,\ctilde j}A_{\ctilde j}^{-1}A_{\ctilde j,j\langle} & A_{j\langle j\bullet}\\
      A_{j\bullet j\langle} & A_{j\bullet}
    \end{array}
  \right].
  \end{equation}

  Let the initial guess of \Cref{alg:DOSMMAT} be
  $\mathbf{u}_j^{(0)}=0$, $j=1,..,J$. Substituting this and the
  specified matrices $P$ and $Q$ into \Cref{alg:DOSMMAT}, we find
  \begin{equation}\label{eq:u1sol}
  \left[
    \begin{array}{cc}
      A_{1\bullet} & A_{1\bullet 1\rangle}\vphantom{_{1\bullet}^{(\frac{1}{2})}}\\
      0 & I_{1\rangle}\vphantom{_{1\bullet}^{(\frac{1}{2})}}
    \end{array}
  \right]
  \left[
    \begin{array}{c}
      \mathbf{u}_{1\bullet}^{(\frac{1}{2})}\\
      \mathbf{u}_{1\rangle}^{(\frac{1}{2})}
    \end{array}
  \right] =
  \left[
    \begin{array}{c}
      \mathbf{f}_{1\bullet}\vphantom{_{1\bullet}^{(\frac{1}{2})}}\\
      0\vphantom{_{1\bullet}^{(\frac{1}{2})}}
    \end{array}
  \right].
  \end{equation}
  By definition, we know $T_1=D_1=A_{1\thetaint}=A_{1\bullet}$ and
  $\mathfrak{f}_1=\mathbf{f}_{1\bullet}$. Hence, from \Cref{eq:u1sol},
  we have $T_1\mathbf{u}_{1\bullet}^{(\frac{1}{2})}=\mathfrak{f}_1$,
  which is satisfied also by $\mathfrak{v}_1$ of
  \Cref{alg:lusol}. Since $T_1$ is invertible and from the fact that
  the decomposition is non-overlapping, we have
  $\mathfrak{v}_1=\mathbf{u}_{1\bullet}^{(\frac{1}{2})}
  =\mathbf{u}_{1\thetaint}^{(\frac{1}{2})}$.  From \Cref{eq:u1sol} and
  again using that the decomposition is non-overlapping, we have
  $\mathbf{u}_{1[}^{(\frac{1}{2})}=\mathbf{u}_{1\rangle}^{(\frac{1}{2})}=0$.
    Now assume that\footnote{when $j=2$, we need to remove the
      non-existent block}
  \begin{equation}\label{eq:y1}
    \mathfrak{v}_{j-1}=\left[
        \mathbf{u}_{j-1\langle}^{(\frac{1}{2})};
        \mathbf{u}_{j-1\thetaint}^{(\frac{1}{2})}
    \right],
    \quad \mathbf{u}_{j-1[}^{(\frac{1}{2})}=0,
  \end{equation}
  with $\mathfrak{v}_{j-1}$ obtained by \Cref{alg:lusol} and
  $\mathbf{u}_{j-1}^{(\frac{1}{2})}$ obtained by \Cref{alg:DOSMMAT}.
  In the next substep of \Cref{alg:DOSMMAT}, we substitute the
  specified matrices $P$ and $Q$, and the second equation of
  \Cref{eq:y1}, to find
  \begin{equation}\label{eq:u2sol}
  \left[
    \arraycolsep0.2em
    \begin{array}{cc|c}
      A_{j\langle}-A_{j\langle,\ctilde j}A_{\ctilde j}^{-1}A_{\ctilde j, j\langle}
      & A_{j\langle j\bullet} & 0\vphantom{u_{j}^{(\frac{1}{2})}}\\
      A_{j\bullet}  & A_{j\bullet} & A_{j\bullet j\rangle}
      \vphantom{u_{j}^{(\frac{1}{j})}}\\ \hline
      0 & 0 & I_{j\rangle}\vphantom{u_{j}^{(\frac{1}{2})}}
    \end{array}
  \right]
  \left[
    \begin{array}{c}
      \mathbf{u}_{j\langle}^{(\frac{1}{2})}\\
      \mathbf{u}_{j\bullet}^{(\frac{1}{2})}\\ \hline
      \mathbf{u}_{j\rangle}^{(\frac{1}{2})}
    \end{array}
  \right] =
  \left[
    \begin{array}{c}
      \mathbf{f}_{j\langle}-A_{j\langle j-1\thetaint}\mathbf{u}_{j-1\thetaint}^{(\frac{1}{2})}\\
      \mathbf{f}_{j\bullet}\vphantom{u_{j}^{(\frac{1}{2})}} \\ \hline
      0\vphantom{u_{j}^{(\frac{1}{2})}}
    \end{array}
  \right].
  \end{equation}
  By \Cref{eq:Tj3}, we know that the upper-left 2-by-2 block matrix in
  \Cref{eq:u2sol} equals $T_j$.  From the first equation of
  \Cref{eq:y1}, we see that the first two rows of the r.h.s. of
  \Cref{eq:u2sol} equal $\mathfrak{f}_2-L_1\mathfrak{v}_1$.  Given
  that $T_j$ is invertible and that the decomposition is
  non-overlapping, we have from \Cref{eq:u2sol} and \Cref{alg:lusol}
  \begin{equation}\label{eq:yj}
    \mathfrak{v}_j=\left[\mathbf{u}_{j\langle}^{(\frac{1}{2})};
        \mathbf{u}_{j\thetaint}^{(\frac{1}{2})}\right],\quad
    \mathbf{u}_{j[}^{(\frac{1}{2})}=0.
  \end{equation}
  Thus, by induction, \Cref{eq:yj} holds for all $j=1,..,J-1$ and
  \begin{displaymath}
  \mathfrak{u}_J=\mathfrak{v}_J=
  \left[
      \mathbf{u}_{J\langle}^{(1)};
      \mathbf{u}_{J\thetaint}^{(1)}
  \right].
  \end{displaymath}
  In \Cref{alg:lusol} we substitute $T_j\mathfrak{v}_j$ from the forward
  sweep to the backward sweep and we get the equivalent backward solve
  \begin{displaymath}
  T_j\mathfrak{u}_j=\mathfrak{f}_j-L_{j-1}\mathfrak{v}_{j-1}-U_{j}\mathfrak{u}_{j+1}.
  \end{displaymath}
  First, we note that the coefficient matrix $T_j$ is the same as in
  the forward solve.  Second, compared to the forward solve, the
  present r.h.s. has an extra term $-U_{j}\mathfrak{u}_{j+1}$ which
  corresponds to use the new Dirichlet data taken from the neighboring
  subdomain on the right. So \Cref{alg:lusol} and the specified case
  of \Cref{alg:DOSMMAT} remain equivalent in their backward solves,
  and we have for $j=J,..,1$\footnote{when $j=1$ we need to remove the
    non-existent block}
  \begin{displaymath}
  \mathfrak{u}_j=
  \left[
    \mathbf{u}_{j\langle}^{(1)}; \mathbf{u}_{j\thetaint}^{(1)}
  \right].
  \end{displaymath}
  Thus the equivalence of the two algorithms is proved.
\end{proof}


Based on \Cref{thm:lu}, the PDE analogue of \Cref{alg:lusol} can
be stated as \Cref{alg:DOSMPDE} with a non-overlapping decomposition,
zero initial guess and $\mathcal{Q}_{j\langle}^\langle=\mathcal{I}$,
$\mathcal{P}_{j\langle}^\langle=\mathrm{DtN}_j^{\langle}$,
$\mathcal{Q}_{j\rangle}^{\rangle}=0$,
$\mathcal{P}_{j\rangle}^{\rangle}=\mathcal{I}$. Next, we revisit the AILU and
the sweeping preconditioners. Based on \Cref{thm:lu}, it is
straightforward to get the following corollaries.

\begin{corollary}
  The AILU preconditioner in \cite{GanderAILU05} is equivalent to one step of
  DOSM with a non-overlapping decomposition such that the subdomain interiors
  are empty (i.e. $\Theta_j=\emptyset$, see \Cref{fig:ailu}), zero initial
  guess, and second-order absorbing transmission conditions on the left
  interfaces and Dirichlet transmission conditions on the right interfaces of
  subdomains.
\end{corollary}

\begin{corollary}
  The sweeping preconditioners in \cite{EY1, EY2} are equivalent to one step of
  DOSM with a non-overlapping decomposition such that the subdomain interiors
  are empty (i.e. $\Theta_j=\emptyset$, see \Cref{fig:ailu}), zero initial
  guess, and PML or $\mathcal{H}$-matrix transmission conditions on the left
  interfaces and Dirichlet transmission conditions on the right interfaces of
  subdomains.
\end{corollary}

\begin{figure}
  \centering
  \includegraphics[scale=.50,trim=30 115 90 85,clip]{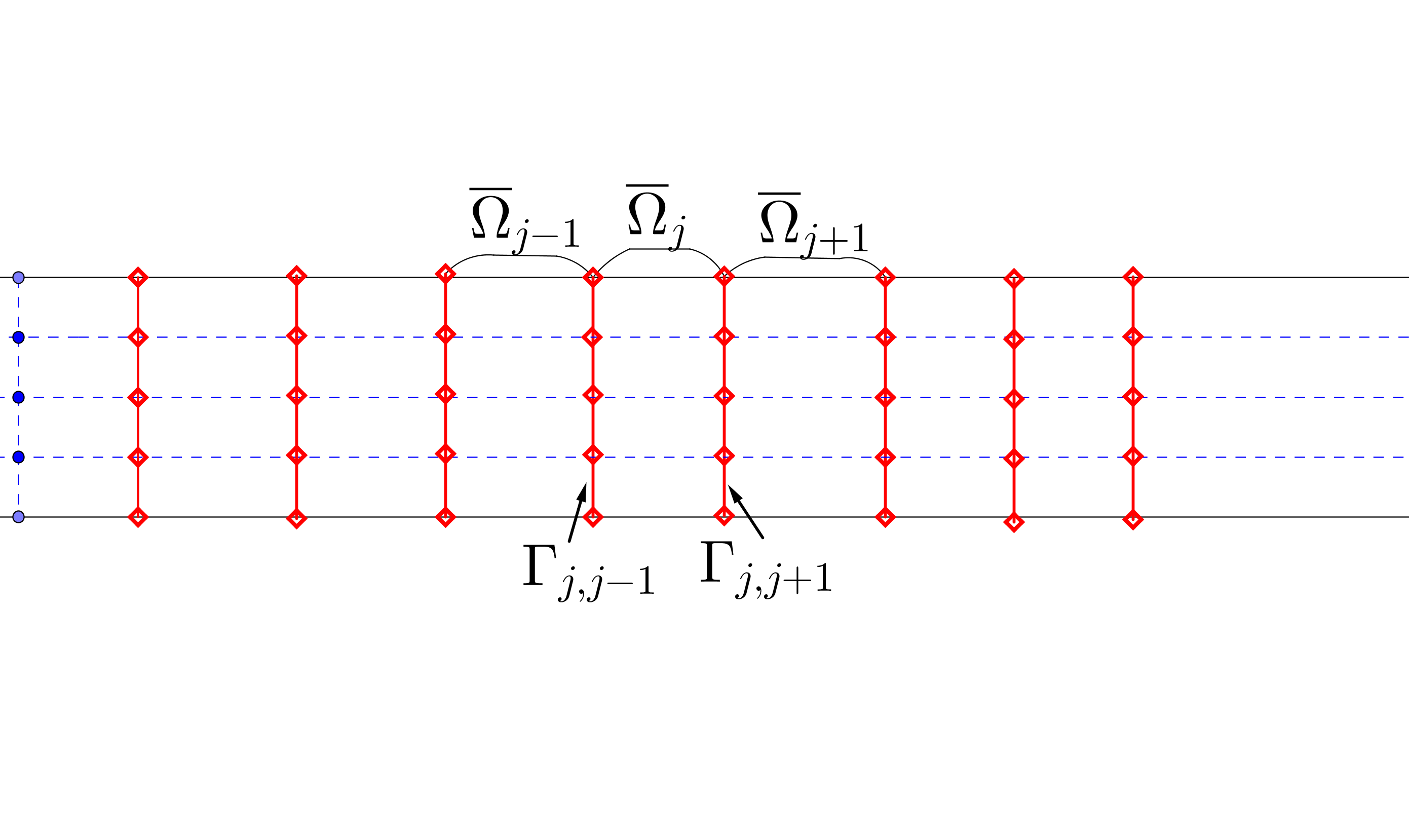}%
  \caption{Non-overlapping domain decomposition with empty interior,
    ${\color{red}\Diamond}\in \Gamma_{*,\#}$.}
  \label{fig:ailu}
\end{figure}
Here we finish our discussions of the algorithms. Now we point out
some analogy between the LU factorization of the matrix from the
discretization of $\partial_{xx}+\partial_{yy}+k^2$ and the formal
analytic factorization of the differential operator,
\begin{equation}\label{eq:anafact}
\partial_{xx}+\partial_{yy}+k^2 =
(\partial_{x} - \I\sqrt{\partial_{yy} + k^2})
(\partial_{x} + \I\sqrt{\partial_{yy} + k^2}).
\end{equation}
This analogy was first drawn in \cite{GanderAILU00}.  The
factorization in \Cref{eq:anafact} represents a reformulation of the
two-way wave (Helmholtz) equation as two one-way wave
equations. Hence, the original {\it boundary value problem} (BVP)
given by the Helmholtz \Cref{eq:pde} can be solved in terms of certain
{\it initial value problems} (IVPs) of the one-way wave equations.
The forward substitution of the lower triangular system and the
backward solution of the upper triangular system at the matrix level
are related to solving these IVPs. Such analytic factorizations have
been studied by many researchers, see e.g. \cite{HalpTref, Lu06}.
Next, we give a derivation of the factorization based on a family of
DtN operators, and explain its relation to \Cref{alg:DOSMPDE}. We will
do it for the more general case with curved domains and curved
interfaces.

We assume that the domain $\Omega$ can be described by the curvilinear
coordinates,
\begin{displaymath}
\Omega=\left\{\mathbf{x}\in\R^d~|~\mathbf{x}=(x_l)_{l=1}^d,
  x_l=x_l(\xi_1,..,\xi_d),\xi_1\in(a,b)\subset\R,(\xi_2,..,\xi_d)\in Y\subset\R^{d-1}
\right\},
\end{displaymath}
with $(\frac{\partial x_j}{\partial\xi_l})$ non-singular and $x_j$ sufficiently
smooth in $\boldsymbol\xi$.  We may view $\Omega$ as a topological cylinder with
the axial variable $\xi_1$.  For $s\in[a,b]$, we denote by
\begin{displaymath}
\begin{array}{l}
  \Omega^s:=\left\{\mathbf{x}\in\Omega~|~x_l=x_l(\xi_1,..,\xi_d),\xi_1\in(a,s),
    (\xi_2,..,\xi_d)\in Y\right\},\\
  \Gamma^s:=\left\{\mathbf{x}\in\Omega~|~x_l=x_l(s,\xi_2,..,\xi_d),
    (\xi_2,..,\xi_d)\in Y\right\}.
\end{array}
\end{displaymath}
Let $d=2$ for simplicity.  In the curvilinear coordinates, the operator
$\mathcal{L}$ of \Cref{eq:pde} becomes
\begin{displaymath}
  \medmuskip=-0mu
  \thinmuskip=-0mu
  \thickmuskip=-0mu
  \nulldelimiterspace=-0pt
  \scriptspace=0pt 
  \mathcal{L}v=-\left(\frac{\partial\xi_1}{\partial x_1}
  \frac{\partial}{\partial\xi_1} + \frac{\partial\xi_2}{\partial
    x_1}\frac{\partial}{\partial\xi_2}\right) \left(\beta^1_1\frac{\partial
    v}{\partial\xi_1} + \beta^1_2 \frac{\partial v}{\partial\xi_2}\right) -
\left(\frac{\partial\xi_1}{\partial x_2} \frac{\partial}{\partial\xi_1} +
  \frac{\partial\xi_2}{\partial x_2}\frac{\partial}{\partial\xi_2}\right)
\left(\beta^2_1\frac{\partial v}{\partial\xi_1} + \beta^2_2 \frac{\partial
    v}{\partial\xi_2}\right) - \frac{\omega^2}{\kappa}v,
\end{displaymath}
where $\beta^j_l=\sum_{m=1}^d\alpha^j_m\frac{\partial\xi_l}{\partial
  x_m}$ and $(\alpha_l^j)$ is the coefficient matrix $\alpha$ in
Cartesian coordinates.  We assume that $\mathcal{L}$ can be rewritten
in the form
\begin{displaymath}
\mathcal{L}v=\gamma_3\left\{ \frac{\partial}{\partial\xi_1}
  \left(\gamma_1\frac{\partial v}{\partial\xi_1} + \gamma_2\frac{\partial
      v}{\partial\xi_2} + \gamma_0v\right) - \mathcal{L}_2v\right\},
\end{displaymath}
with $\gamma_j=\gamma_j(\xi_1,\xi_2)$, $j=1,2,3$, $\gamma_3$ and
$\gamma_1$ nowhere zero and $\mathcal{L}_2$ a partial differential
operator w.r.t.  $\xi_2$ only.  This assumption can be verified if
$\beta^1_1$, $\beta^2_1$ are sufficiently smooth or if $\xi_1=x_1$,
$\xi_2=x_2$ and $\alpha^2_1$ is sufficiently smooth.  We define the
operator-valued function $\mathcal{P}(s)$ for $s\in [a,b]$ as the
$s$-dependent generalized DtN operator (we assume the problem below is
well-posed)
\begin{displaymath}
\begin{array}{r@{\hspace{0.2em}}c@{\hspace{0.2em}}ll}
  \mathcal{P}(s):\;d\rightarrow (\gamma_1\frac{\partial v}{\partial\xi_1} +
  \gamma_2\frac{\partial v}{\partial\xi_2}+\gamma_0v)|_{\Gamma^s},
  \mbox{ s.t.}\qquad
  \mathcal{L}\,v&=&0\quad & \mbox{in }\,\Omega^s,\\
  \mathcal{B}\,v&=&0\quad &\mbox{on }\,
  \partial\Omega^s\cap\partial\Omega,\\
  v&=&d\quad &\mbox{on }\,\Gamma^s.
\end{array}
\end{displaymath}
Let us consider how $\mathcal{P}(s)$ changes with $s$.  Let $v$ be the
solution in the definition of $\mathcal{P}(s)$.  We have for $\Delta
s<0$,
\begin{displaymath}
  \medmuskip=-1mu
  \thinmuskip=-1mu
  \thickmuskip=-1mu
  \nulldelimiterspace=-0pt
  \scriptspace=0pt 
\begin{array}{r@{\hspace{0.2em}}c@{\hspace{0.2em}}ll}
  \gamma_1\frac{\partial v}{\partial\xi_1}(s,\xi_2) +
  \gamma_2\frac{\partial v}{\partial\xi_2}(s,\xi_2) + \gamma_0v(s,\xi_2)&=&
  \mathcal{P}(s)v(s,\xi_2),&\xi_2\in Y,\\
  \gamma_1\frac{\partial v}{\partial\xi_1}(s+\Delta s,\xi_2) +
  \gamma_2\frac{\partial v}{\partial\xi_2}(s+\Delta s,\xi_2)
  +\gamma_0v(s+\Delta s,\xi_2) &=&
  \mathcal{P}(s+\Delta s)v(s+\Delta s,\xi_2),&\xi_2\in Y.
\end{array}
\end{displaymath}
Subtracting the two equations and dividing by $\Delta s$, we get
\begin{displaymath}
\begin{array}{ll}
  &1/\Delta s\cdot\left\{(\gamma_1\partial_1v+\gamma_2\partial_2v+\gamma_0v)
    (s+\Delta s,\xi_2)-(\gamma_1\partial_1v+\gamma_2\partial_2v+\gamma_0v)
    (s,\xi_2)\right\}\\
  =&1/\Delta s\cdot\{\mathcal{P}(s+\Delta s)-\mathcal{P}(s)\}v(s+\Delta s,\xi_2)%
  + \mathcal{P}(s)\{v(s+\Delta s,\xi_2)-v(s,\xi_2)\}/\Delta s.
\end{array}
\end{displaymath}
We assume that the difference quotients above converge as $\Delta
s\rightarrow0$, and we find
\begin{equation}\label{eq:analytic}
\frac{\partial}{\partial\xi_1} \left(\gamma_1\frac{\partial v}{\partial\xi_1} +
  \gamma_2\frac{\partial v}{\partial\xi_2} + \gamma_0v\right)(s,\xi_2)
=\mathcal{P}'(s)v(s,\xi_2)+\mathcal{P}(s)\frac{\partial v}{\partial
  \xi_1}(s,\xi_2).
\end{equation}
Using $\mathcal{L}v=0$ and $\gamma_3\neq0$, we can find that
  $\mathcal{L}_2(s,\xi_2)v(s,\xi_2)$ is equal to the left hand side of
  \Cref{eq:analytic} and thus to the r.h.s. too. We further replace
$\frac{\partial v}{\partial \xi_1}(s,\xi_2)$ with the definition of
$\mathcal{P}(s)$ to obtain
\begin{displaymath}
\mathcal{L}_2(s,\cdot)v(s,\cdot)=\mathcal{P}'(s)v(s,\cdot)+
\mathcal{P}(s)\left\{\gamma_1^{-1}(s,\cdot)\left(\mathcal{P}(s)-
  \gamma_2(s,\cdot)\partial_2-\gamma_0(s,\cdot)\right)v(s,\cdot)\right\}.
\end{displaymath}
Removing $v(s,\cdot)$ from the above equation, we obtain the differential
Riccati equation
\begin{equation}\label{eq:Riccati}
\mathcal{P}'(s) =
\mathcal{P}(s)\left\{\gamma_1^{-1}(s,\cdot)\left(-\mathcal{P}(s)+
    \gamma_2(s,\cdot)\partial_2+\gamma_0(s,\cdot)\right)\right\} +
\mathcal{L}_{2}(s,\cdot).
\end{equation}
As mentioned in \cite{Henry}, \Cref{eq:Riccati} can also be obtained
from \Cref{eq:TjP} when the width of the subdomain is equal to the mesh
size and goes to zero.
The initial value $\mathcal{P}(a)$ for \Cref{eq:Riccati} can be
obtained from the boundary condition $\mathcal{B}v=0$ on $\Gamma^a$ if
the condition is not of Dirichlet type (we assume this in the
following derivation; otherwise, we should not use the DtN operator but the
NtD operator for $\mathcal{P}$).  In particular, if $a=-\infty$ and there
exists $a'\in\R$ such that the problem in the definition of
$\mathcal{P}(s)$ is independent of $s\leq a'$, then by letting
$\mathcal{P}'(s)=0$ we get from the differential Riccati equation the
algebraic Riccati equation
\begin{displaymath}
\mathcal{P}(s)\left\{\gamma_1^{-1}(s,\cdot)\left(-\mathcal{P}(s)+
    \gamma_2(s,\cdot)\partial_2+\gamma_0(s,\cdot)\right)\right\} +
\mathcal{L}_{2}(s,\cdot)=0,\quad\forall s\leq a'.
\end{displaymath}
The solution of the algebraic Riccati equation at $s=a'$ gives us an
initial value $\mathcal{P}(a')$ for the differential Riccati equation.
In the following, we assume $\mathcal{P}$ has been precomputed on
$[a,b]$.

We introduce $w:=\left(\gamma_1\frac{\partial u}{\partial\xi_1} +
\gamma_2\frac{\partial u}{\partial\xi_2} +
\gamma_0u\right)-\mathcal{P}u$ with $u$ the solution of \Cref{eq:pde}.
Again, by the definition of $\mathcal{P}$, the initial value
$w(a,\cdot)$ can be acquired from $\mathcal{B}u=g$ on $\Gamma^a$.  For
example, if $\mathcal{B}=\tilde{\gamma}_1\frac{\partial
}{\partial\xi_1} + \tilde{\gamma}_2\frac{\partial }{\partial\xi_2} +
\tilde{\gamma}_0$ and $v$ is from the definition of $\mathcal{P}(a)$,
i.e. $v(a,\cdot)=u(a,\cdot)$, we have
$\mathcal{B}(u-v)=\tilde{\gamma}_1\frac{\partial
  (u-v)}{\partial\xi_1}=g$ on $\Gamma^a$ and
\begin{displaymath}
  \medmuskip=2mu
  \thinmuskip=2mu
  \thickmuskip=2mu
  \nulldelimiterspace=1pt
  \scriptspace=0pt 
w(a,\cdot)=\gamma_1\frac{\partial (u-v)}{\partial\xi_1}(a,\cdot) +
\gamma_2\frac{\partial (u-v)}{\partial\xi_2}(a,\cdot) + \gamma_0(u-v)(a,\cdot)=
\gamma_1\frac{\partial(u-v)}{\partial\xi_1}(a,\cdot)=
\gamma_1\tilde{\gamma}_1^{-1}g.
\end{displaymath}
We calculate the partial derivative of $w,$
\begin{align}
  \frac{\partial w}{\partial\xi_1}
  &=\frac{\partial}{\partial\xi_1}
    \left(\gamma_1\frac{\partial u}{\partial\xi_1} + \gamma_2\frac{\partial
    u}{\partial\xi_2} + \gamma_0u\right) - \mathcal{P}'u -
    \mathcal{P}\frac{\partial u}{\partial\xi_1}\nonumber\\
  &=\gamma_3^{-1}f+\mathcal{L}_2u-\mathcal{P}'u-
    \mathcal{P}\frac{\partial u}{\partial\xi_1}\nonumber\\
  &=\gamma_3^{-1}f+\mathcal{P}\{\gamma_1^{-1}(\mathcal{P}u-\gamma_2\partial_2u-\gamma_0u)\}
    -\mathcal{P}\frac{\partial u}{\partial\xi_1}\nonumber\\
  &=\gamma_3^{-1}f-\mathcal{P}\{\gamma_1^{-1}w\},\label{eq:ivp4w}
\end{align}
where we successively used the partial differential equation satisfied
by $u,$ the differential Riccati equation of $\mathcal{P}$, and the
definition of $w.$ So we have obtained an IVP for $w$ which is the
analogue of the `$L$' system of the LU factorization.  The analytic
analogue of the `U' system from which to recover $u$ is simply the
definition of $w$,
\begin{equation}\label{eq:ivp4u}
\gamma_1\frac{\partial u}{\partial\xi_1} = w- \gamma_2\frac{\partial
  u}{\partial\xi_2} - \gamma_0u+\mathcal{P}u.
\end{equation}
The initial value for this system is set on $\Gamma^b$ and can be
sought again from the boundary condition $\mathcal{B}u=g$ and the
one-way wave equation itself at $\xi_1=b$. We thus have as the
analytic analogue of the LU factorization
\begin{displaymath}
\frac{\partial}{\partial\xi_1} \left(\gamma_1\frac{\partial }{\partial\xi_1} +
  \gamma_2\frac{\partial}{\partial\xi_2} + \gamma_0\right) - \mathcal{L}_2 =
\left(\frac{\partial}{\partial\xi_1}+\mathcal{P}\{\gamma_1^{-1}\cdot\}\right)
\left(\gamma_1\frac{\partial}{\partial\xi_1} +
  \gamma_2\frac{\partial}{\partial\xi_2} + \gamma_0-\mathcal{P}\right),
\end{displaymath}
which can be verified for an arbitrary function $v(\xi_1,\xi_2)$ as follows:
\begin{displaymath}
\begin{aligned}
  &\left(\frac{\partial}{\partial\xi_1}+\mathcal{P}\{\gamma_1^{-1}\cdot\}\right)
  \left(\gamma_1\frac{\partial}{\partial\xi_1} +
    \gamma_2\frac{\partial}{\partial\xi_2} + \gamma_0-\mathcal{P}\right)v\\
  =&\left(\frac{\partial}{\partial\xi_1}+\mathcal{P}\{\gamma_1^{-1}\cdot\}\right)
  \left(\gamma_1\frac{\partial v}{\partial\xi_1} +
    \gamma_2\frac{\partial v}{\partial\xi_2} + \gamma_0v-\mathcal{P}v\right)\\
  =& \frac{\partial}{\partial\xi_1}\left(\gamma_1\frac{\partial v}
    {\partial\xi_1} +\gamma_2\frac{\partial v}{\partial\xi_2} +
    \gamma_0v\right)-\frac{\partial}{\partial\xi_1}(\mathcal{P}v)
  +\mathcal{P}\frac{\partial v}{\partial\xi_1}+\mathcal{P}
  \{\gamma_1^{-1}(\gamma_2\frac{\partial v}{\partial\xi_2}+\gamma_0v-\mathcal{P}v)\}\\
  =&\frac{\partial}{\partial\xi_1}\left(\gamma_1\frac{\partial v}
    {\partial\xi_1} +\gamma_2\frac{\partial v}
    {\partial\xi_2}+\gamma_0v\right)-\mathcal{L}_2v,
\end{aligned}
\end{displaymath}
where to obtain the last identity we substituted the differential Riccati
\Cref{eq:Riccati}.  Note that all the above derivation needs to be justified in
appropriate function spaces, for which we refer to e.g. \cite{Henry}.

Solving the IVP in \Cref{eq:ivp4w} is not the only way to get $w$.  We
can also solve the original problem \Cref{eq:pde} restricted to
$\Omega^{s_1}$ ($s_1\leq b$) complemented with an arbitrary boundary
condition on $\Gamma^{s_1}$ that guarantees well-posedness, which results
in $\tilde{u}$.  From the proof of \Cref{thm:1steppde}, we have
$w=\left(\gamma_1\frac{\partial \tilde{u}}{\partial\xi_1} +
\gamma_2\frac{\partial \tilde{u}}{\partial\xi_2} +
\gamma_0\tilde{u}\right)-\mathcal{P}\tilde{u}$ on $\Gamma^{t}$ for all
$t\in(a, s_1]$.  Suppose $w$ is known in $\Omega^{s_1}$, to get $w$ in
$\Omega^{s_2}-\Omega^{s_1}$ for $s_2>s_1$, we only have to solve the original
problem restricted to $\Omega^{s_2}-\Omega^{s_1}$ where $\tilde{u}$
satisfies the boundary condition $\left(\gamma_1\frac{\partial
  \tilde{u}}{\partial\xi_1} + \gamma_2\frac{\partial
  \tilde{u}}{\partial\xi_2} +
\gamma_0\tilde{u}\right)-\mathcal{P}\tilde{u}=w$ on $\Gamma^{s_1}$ and an
arbitrary boundary condition for well-posedness on $\Gamma^{s_{2}}$;
then, we have $w=\left(\gamma_1\frac{\partial
  \tilde{u}}{\partial\xi_1} + \gamma_2\frac{\partial
  \tilde{u}}{\partial\xi_2} +
\gamma_0\tilde{u}\right)-\mathcal{P}\tilde{u}$ on $\Gamma^{t}$ for all
$t\in(a, s_2]$.  This process continues forward until $\Gamma^b$ and
$w$ is obtained in $\Omega$.  Then, we solve \Cref{eq:pde} restricted
to $\Omega^b-\Omega^{s_{J-1}}$ for $u$ with $w$ providing interface data on
$\Gamma^{s_{J-1}}$.  To find $u$ further backward in $\Omega^{s_{J-1}}-\Omega^{s_{J-2}}$,
we use again $w$ as interface data on $\Gamma^{s_{J-2}}$, while an
arbitrary boundary condition on $\Gamma^{s_{J-1}}$,
as long as the resulting problem is
well-posed, can be extracted from already known $u$ in
$\Omega^b-\Omega^{s_{J-1}}$.  This process continues backward until
$\Gamma^a$.  The forward plus backward processes constitute exactly
\Cref{alg:DOSMPDE}.  In other words, we may view \Cref{alg:DOSMPDE} as
a way of solving the IVPs for $w$ and $u$ in \Cref{eq:ivp4w,eq:ivp4u}.

\section{Methods motivated by physics} \label{sec:phys}

We now present several algorithms motivated by various intuitions from
physics, and developed using Green's function techniques. We start
with the special case we considered in \Cref{sec:notation}, where the
source term vanished outside a subdomain $\Omega_j$, and we showed how
to truncate the original problem to $\Omega_j$ to avoid discretizing
the big domain $\Omega$. To be able to use this as a building block
for a more general solver, we need two further ingredients: first, since
we are now not only interested in the near-field solution $u$ in
$\Omega_j$, but also the far-field $u$ in $\Omega-\Omega_j$, we need
to be able to map the near-field waves to the far-field. This is a
classical engineering problem, see
e.g. \cite[pp. 329--352]{THbook}. Second, we may have also sources
outside $\Omega_j$ which stimulate waves that come into $\Omega_j$.
The question is then how to incorporate the influence of these exterior
sources on the local solution on the subdomain $\Omega_j$.

In the rest of this section we assume that $g=0$ in \Cref{eq:pde} to
simplify our presentation.  From the solution formula we have seen in
\Cref{eq:repf}, namely
\begin{displaymath}
  u(\mathbf{x}) = \int_{\Omega} G(\mathbf{x},\mathbf{y})f(\mathbf{y})\,
  \mathrm{d}\mathbf{y},
\end{displaymath}
we see that if we restrict to $\mathbf{x},\mathbf{y}\in\Omega_j$, then
the corresponding diagonal part of $G(\mathbf{x},\mathbf{y})$ can be
well approximated by a good truncation of the corresponding BVP to
$\Omega_j$.  The second point raised in the last paragraph actually
asks how to approximate the off-diagonal parts of
$G(\mathbf{x},\mathbf{y})$ when $\mathbf{x},\mathbf{y}$ are in {\it
  different} subdomains.  This is a core problem studied in
$\mathcal{H}$-matrix and similar techniques, see the references
mentioned in \Cref{sec:solvers}, but a direct approximation of the
off-diagonal parts of the Green's function is difficult for waves
traveling long-distance in heterogeneous media.  As an alternative,
one can first solve for the diagonal parts, i.e. the near-field waves
stimulated by the sources within each subdomain, and then propagate
the waves gradually subdomain by subdomain from near to far. This is
possible because of the so-called equivalence theorem in
engineering, which states that if we enclose a source by a
surface, referred to as Huygens surface, then the stimulated
waves in the exterior can be determined from the waves on the
  Huygens surface, thought as new equivalent sources that are called currents
  in engineering.  For example, once we have
$u_{1,1}:=u_{,1}|_{\Omega_1}$ with
$u_{,1}(\mathbf{x}):=\int_{\Omega_1}G(\mathbf{x},\mathbf{y})
f(\mathbf{y})\,\mathrm{d}\mathbf{y},$ $\mathbf{x}\in\Omega$, we should
be able to find also $u_{2,1}:=u_{,1}|_{\Omega_2}$.  In fact, $u_{,1}$
solves the BVP
\begin{displaymath}
  \mathcal{L}u_{,1}=f_1\mbox{\ in\ } \Omega,\quad
  \mathcal{B}u_{,1}=0 \mbox{\ on\ } \partial\Omega,
\end{displaymath}
where $f_1:=\mathcal{E}_1(f|_{\Omega_1})$; so we deduce that $u_{2,1}$ can be
obtained from
\begin{equation}\label{eq:u2f1}
  \begin{array}{r@{\hspace{0.2em}}c@{\hspace{0.2em}}ll}
    \mathcal{L}u_{2,1}&=&0&\mbox{\ in\ } \Omega_2,\\
    \mathcal{B}u_{2,1}&=&0&\mbox{\ on\ } \partial\Omega\cap\partial\Omega_2,\\
    \mathcal{B}_2^{\langle}u_{2,1}&=&\mathcal{B}_2^{\langle}u_{1,1}&
    \mbox{\ on\ } \Gamma_{2,1},\\
    \mathbf{n}_2^T\alpha\nabla u_{2,1}+\mathrm{DtN}_2u_{2,1}&=&0&
    \mbox{\ on\ } \Gamma_{2,3},
  \end{array}
\end{equation}
where $\mathcal{B}_2^{\langle}$ is an arbitrary boundary operator so
that the problem is well-posed, and we assume $\mathrm{DtN}_2$ is
well-defined.  We see that the influence of $f_1$ to the waves in
$\Omega_2$ has been transformed to an equivalent surface current
$\mathcal{B}_2^{\langle}u_{1,1}$.  In summary, the near-field waves
$u_{1,1}$ generate a surface current $\mathcal{B}_2^{\langle}u_{1,1}$
from which one can recover the far-field waves $u_{2,1}$, as the
equivalence theorem says.

Since in $\Omega_2$ we want also
$u_{2,2}(\mathbf{x})=\int_{\Omega_2}G(\mathbf{x},\mathbf{y})
f(\mathbf{y})\,\mathrm{d}\mathbf{y},$ $\mathbf{x}\in\Omega_2$, it is convenient
to add the source $f_2$ directly in \Cref{eq:u2f1}, and to solve for
$(u_{2,1}+u_{2,2})(\mathbf{x})=u_{2,1:2}(\mathbf{x}):=
\int_{\Omega_1\cup\Omega_2}G(\mathbf{x},\mathbf{y})
f(\mathbf{y})\,\mathrm{d}\mathbf{y},$ $\mathbf{x}\in\Omega_2$ at once from
\begin{equation}\label{eq:u2f1f2}
  \begin{array}{r@{\hspace{0.2em}}c@{\hspace{0.2em}}ll}
    \mathcal{L}u_{2,1:2}&=&f_2&\mbox{\ in\ } \Omega_2,\\
    \mathcal{B}u_{2,1:2}&=&0&\mbox{\ on\ } \partial\Omega\cap\partial\Omega_2,\\
    \mathbf{n}_2^T\alpha\nabla u_{2,1:2}+\mathrm{DtN}_2u_{2,1:2}&=&
    \mathbf{n}_2^T\alpha\nabla u_{1,1}+\mathrm{DtN}_2u_{1,1}&
    \mbox{\ on\ } \Gamma_{2,1},\\
    \mathbf{n}_2^T\alpha\nabla u_{2,1:2}+\mathrm{DtN}_2u_{2,1:2}&=&0&
    \mbox{\ on\ } \Gamma_{2,3},
  \end{array}
\end{equation}
where we specified $\mathcal{B}_2^{\langle}$ of \Cref{eq:u2f1} as the
transparent boundary operator to simulate the waves generated by
  $f_2$ without spurious reflections.  Using \Cref{eq:rep}, the solution
of \Cref{eq:u2f1f2} can be represented as
\begin{displaymath}
u_{2,1:2}(\mathbf{x})=\int_{\Omega_2} G(\mathbf{x},\mathbf{y})
f(\mathbf{y})\,\mathrm{d}\mathbf{y}+\int_{\Gamma_{2,1}}G(\mathbf{x},\mathbf{y})
\lambda_{2\langle}(\mathbf{y})\,\mathrm{d}\sigma(\mathbf{y}),
\quad\mathbf{x}\in\Omega_2,
\end{displaymath}
where
$\lambda_{2\langle}:=\mathcal{B}_{2}^{\langle}u_{1,1}=\mathbf{n}_2^T\alpha\nabla
u_{1,1}+\mathrm{DtN}_2u_{1,1}$.  Now that $u_{2,1:2}$ contains the
influence of both $f_1$ and $f_2$, this influence can be passed on to
$\Omega_3$ through a transmission condition on $\Gamma_{3,2}$, and
using a transparent transmission condition also permits to include
then the influence of $f_3$ locally in $\Omega_3$.  This process
continues until we obtain $u_{J,1:J}$ which is the exact solution of
\Cref{eq:pde} restricted to $\Omega_J$ i.e.  $u_{J,1:J}(\mathbf{x})=
\int_{\Omega} G(\mathbf{x},\mathbf{y})f(\mathbf{y})d\mathbf{y},$
$\mathbf{x}\in\Omega_J$. Now that we have $u_{j,1:j}$,
$j=1,..,J$ and the interface data
$\lambda_{j+1\langle}:=\mathcal{B}_{j+1}^{\langle}u_{j,1:j}
=\mathbf{n}_{j+1}^T\alpha\nabla
u_{j,1:j}+\mathrm{DtN}_{j+1}u_{j,1:j}$, $j=1,..,J-1$, we want to add
the waves $u_{j,j+1:J}$, stimulated by the sources on the right of
$\Omega_j$, to $u_{j,1:j}$ to get $u_{j,1:J}=u_{j,1:j}+u_{j,j+1:J}$,
the solution of the original BVP in \Cref{eq:pde} restricted to
$\Omega_j$.  We note that
$\mathcal{B}_{j+1}^{\langle}u_{j,1:J}=\mathcal{B}_{j+1}^{\langle}u_{j,1:j}
=\lambda_{j+1\langle}$ because
$\mathcal{B}_{j+1}^{\langle}u_{j,j+1:J}=0$ by \Cref{lem:TC}.  That is,
the waves from $\Omega_{j+1:J}$ pass through $\Gamma_{j+1,j}$
transparently.  For $u_{J-1,1:J}$, the other interface data
$\mathcal{B}_{J-1}^{\rangle}u_{J-1,1:J}$ on $\Gamma_{J-1,J}$ is
available from the known solution $u_{J,1:J}$.  Therefore,
$u_{J-1,1:J}$ satisfies the BVP
\begin{equation}\label{eq:uJ-1}
  \begin{array}{r@{\hspace{0.2em}}c@{\hspace{0.2em}}ll}
    \mathcal{L}v_{J-1}&=&f_{J-1}&\mbox{\ in\ } \Omega_{J-1},\\
    \mathcal{B}v_{J-1}&=&0&\mbox{\ on\ } \partial\Omega\cap\partial\Omega_{J-1},\\
    \mathbf{n}_{J-1}^T\alpha\nabla v_{J-1}+\mathrm{DtN}_{J-1}v_{J-1}&=&
    \lambda_{J-1\langle}&\mbox{\ on\ } \Gamma_{J-1,J-2},\\
    \mathcal{B}_{J-1}^{\rangle}v_{J-1}&=&\mathcal{B}_{J-1}^{\rangle}u_{J,1:J}&
    \mbox{\ on\ } \Gamma_{J-1,J},
  \end{array}
\end{equation}
where all the data is known, and $\mathcal{B}_{J-1}^{\rangle}$ is
arbitrary as long as \Cref{eq:uJ-1} is well-posed, and the first two
equations are just taken from \Cref{eq:pde}.  In other words,
\Cref{eq:uJ-1} is set up according to what the solution $u$ of
\Cref{eq:pde} satisfies, and the unique solvability of \Cref{eq:uJ-1}
justifies that its solution can only be
$v_{J-1}=u_{J-1,1:J}=u|_{\Omega_{J-1}}$.  After having $u_{J-1,1:J}$,
the exact information can be passed backward further through
$\Gamma_{J-2,J-1}$. This process continues until we obtain the exact
solution in all the subdomains.  We can also formulate this process
using the representation formula in \Cref{eq:rep}.  For example, if
$\lambda_{J-1\rangle}:=\mathcal{B}_{J-1}^{\rangle}u_{J,1:J}=\mathbf{n}_{J-1}^T\alpha
\nabla u_{J,1:J}+\mathrm{DtN}_{J-1}u_{J,1:J}$, then solving
\Cref{eq:uJ-1} is equivalent to computing for $\mathbf{x}\in\Omega_{J-1}$,
\begin{displaymath}
  \medmuskip=-1mu
  \thinmuskip=-1mu
  \thickmuskip=-1mu
  \nulldelimiterspace=-0pt
  \scriptspace=0pt 
  u_{J-1,1:J}(\mathbf{x})=\int_{\Omega_{J-1}}\hspace{-1em}G(\mathbf{x},\mathbf{y})
  f(\mathbf{y})\,\mathrm{d}\mathbf{y} +
  \int_{\Gamma_{J-1,J-2}}\hspace{-2em}G(\mathbf{x},\mathbf{y})
  \lambda_{J-1\langle}(\mathbf{y})\,\mathrm{d}\sigma(\mathbf{y}) +
  \int_{\Gamma_{J-1,J}}\hspace{-2em}G(\mathbf{x},\mathbf{y})
  \lambda_{J-1\rangle}(\mathbf{y})\,\mathrm{d}\sigma(\mathbf{y}).
\end{displaymath}
We have now presented the basic ideas digested from \cite{Chen13a,
  Stolk, Zepeda, ZD}, and have again derived \Cref{alg:DOSMPDE}.  To
propagate the waves based on physical insight, we were forced
here to use absorbing transmission conditions on $\Gamma_{j,j+1}$
in the forward sweep.  In the next subsections, we will explain in
detail the different paths that led to the invention of the new
methods in \cite{Chen13a, Stolk, Zepeda, ZD}. We will see that these
new methods were derived in quite different forms, but it will become
clear how closely they are related to the algorithms discussed in
\Cref{sec:osm,sec:lu}.

\subsection{The source transfer method using equivalent volume sources}
\label{sec:stddm}

The source transfer method from \cite{Chen13a} is based on a decomposition of
$\Omega$ into non-overlapping and serially connected layers $O_j$, $j=0,..,J$,
which are combined into subdomains $\Omega_j:=O_{j-1}\cup\Gamma_{j}\cup O_j$
with $\Gamma_j:=\partial O_{j-1}\cap\partial O_j$ for $j=1,..,J$.  We have seen
this decomposition in \Cref{rem:ddst}, and also assume here that the resulting
system is block tridiagonal, as in \Cref{rem:ddst}.  The key idea of the source
transfer method is the physical intuition that it is possible to transfer the
sources before $O_J$ into $O_{J-1}$ without changing the wave field in $O_J$.
This is done layer by layer.  First, the source in $O_0$ is transferred to $O_1$
without changing the wave field to the right of $O_1$, i.e. in $O_j$,
$j=2,..,J$.  In terms of the Green's function, we need to find a map
$\Psi_1$ that transfers $\tilde{f}_{1\lovlpint}:=f|_{O_0}$ to the source
$\Psi_1(\tilde{f}_{1\lovlpint})$ defined in $O_1$ such that
\begin{displaymath}
\int_{O_0}G(\mathbf{x},\mathbf{y})\tilde{f}_{1\lovlpint}(\mathbf{y})\,\mathrm{d}\mathbf{y}
=\int_{O_1}G(\mathbf{x},\mathbf{y})\Psi_1(\tilde{f}_{1\lovlpint})(\mathbf{y})\,\mathrm{d}\mathbf{y},
\quad \forall \mathbf{x}\in O_{l},~ l=2,..,J.
\end{displaymath}
Then we define
$\tilde{f}_{2\lovlpint}:=f|_{O_1}+\Psi_1(\tilde{f}_{1\lovlpint})$.  For
$j=1,..,J-2$, we try to find a map $\Psi_{j+1}$ that transfers
$\tilde{f}_{j+1\lovlpint}$ to the source
$\Psi_{j+1}(\tilde{f}_{j+1\lovlpint})$ in $O_{j+1}$ such that for
$l=j+2,..,J$
\begin{equation}\label{eq:stbase}
\int_{O_j}G(\mathbf{x},\mathbf{y})\tilde{f}_{j+1\lovlpint}(\mathbf{y})\,\mathrm{d}\mathbf{y}
=\int_{O_{j+1}}G(\mathbf{x},\mathbf{y})\Psi_{j+1}(\tilde{f}_{j+1\lovlpint})(\mathbf{y})\,
\mathrm{d}\mathbf{y},\quad \forall \mathbf{x}\in O_{l},
\end{equation}
and define
$\tilde{f}_{j+2\lovlpint}:=f|_{O_{j+1}}+\Psi_{j+1}(\tilde{f}_{j+1\lovlpint})$.
Eventually, we get $\tilde{f}_{J\lovlpint}$ and the wave field in
$O_J$ of the original problem in \Cref{eq:pde} is given by
\begin{displaymath}
u_{J\thetaint}(\mathbf{x})=\int_{O_J}G(\mathbf{x},\mathbf{y})f_{J\thetaint}(\mathbf{y})
\,\mathrm{d}\mathbf{y} + \int_{O_{J-1}}G(\mathbf{x},\mathbf{y})\tilde{f}_{J\lovlpint}
(\mathbf{y})\,\mathrm{d}\mathbf{y},\quad\forall \mathbf{x}\in O_J.
\end{displaymath}
Once $u_{J\thetaint}$ is known, it leaves a Dirichlet trace on
$\Gamma_{J}$ which, together with the transferred source
$\tilde{f}_{J-1\lovlpint}$ and the source in $O_{J-1}$, results in the
wave field in $O_{J-1}$,
\begin{multline*}  
  u_{J-1\rovlpint}(\mathbf{x})=\int_{O_{J-1}}G_{J-1}(\mathbf{x},\mathbf{y})f_{J-1\rovlpint}
  (\mathbf{y})\,\mathrm{d}\mathbf{y} + \int_{O_{J-2}}
  G_{J-1}(\mathbf{x},\mathbf{y})\tilde{f}_{J-1\lovlpint}
  (\mathbf{y})\,\mathrm{d}\mathbf{y} \\
  + \int_{\Gamma_{J}}\left[\mathbf{n}^T\alpha^T\nabla_{\mathbf{y}}
  G_{J-1}(\mathbf{x},\mathbf{y})\right]
  u_{J\thetaint}(\mathbf{y})\,\mathrm{d}\sigma(\mathbf{y}),
  \quad\forall\mathbf{x}\in O_{J-1},
\end{multline*}
where $G_{J-1}$ is the Green's function in $\Omega_{J-1}$ satisfying a
transparent boundary condition on $\Gamma_{J-2}$ and a homogeneous
Dirichlet condition on $\Gamma_{J}$.  This amounts to solve a BVP in
$\Omega_{J-1}$ with Dirichlet boundary condition $u_{J-1}=u_{J}$ on
$\Gamma_{J}$ and the transparent boundary condition on $\Gamma_{J-2}$.
Afterwards, the wave field in $O_{J-1}$ can be again extended backward
to $O_{J-2}$.  This process continues until we have obtained the wave
field in all the layers.

The authors of \cite{Chen13a} emphasize that ``{\it the key step in
  the method is the source transfer algorithm}'' which results in
\Cref{eq:stbase}. So how can one obtain the source transfer functions
$\Psi_{j+1}$~? While the source transfer appeared very naturally in
the block LU factorization in \Cref{SourceTransferIdea}, it is a bit
more involved to find $\Psi_{j+1}$ in the present setting: the authors
substitute the source $\tilde{f}_{j+1\lovlpint}$ in $O_j$ with the
generated wave field $v_{j+1}$ in $\Omega_{j+1}$ by using the PDE
$\mathcal{L}v_{j+1}=\tilde{f}_{j+1}$ in $\Omega_{j+1}$ (let
$\tilde{f}_{j+1}$ be the zero extension of $\tilde{f}_{j+1\lovlpint}$ to $\Omega_{j+1}$).
Substituting this and
$-\frac{\omega^2}{\kappa}G(\mathbf{x},\mathbf{y})=
\nabla_{\mathbf{y}}^T(\alpha^T\nabla_{\mathbf{y}}
G(\mathbf{x},\mathbf{y}))$ at $\mathbf{x}\neq\mathbf{y}$ into the
l.h.s. of \Cref{eq:stbase}, they obtain for $\mathbf{x}\in
O_l,~l=j+2,..,J,$
\begin{displaymath}
  \medmuskip=-1mu
  \thinmuskip=-1mu
  \thickmuskip=-1mu
  \nulldelimiterspace=-0pt
  \scriptspace=0pt
  \int_{O_j}G(\mathbf{x},\mathbf{y})\tilde{f}_{j+1\lovlpint}(\mathbf{y})\,\mathrm{d}\mathbf{y}=
  \int_{O_j}\left[-\nabla^T(\alpha\nabla v_{j+1}(\mathbf{y})) \right]
  G(\mathbf{x},\mathbf{y})+\left[\nabla_{\mathbf{y}}^T(\alpha^T\nabla_{\mathbf{y}}
    G(\mathbf{x},\mathbf{y}))\right]v_{j+1}(\mathbf{y})
  \,\mathrm{d}\mathbf{y}.
\end{displaymath}
Integrating by parts and substituting the boundary conditions (transparent on
$\Gamma_{j}$ relative to $O_j$ for both $v_{j+1}$ and $G(\mathbf{x},\cdot)$
with $\mathbf{x}$ to the right of $O_{j+1}$) leads to
\begin{equation}\label{eq:gammaj+1}
  \medmuskip=-2mu
  \thinmuskip=-2mu
  \thickmuskip=-2mu
  \nulldelimiterspace=-2pt
  \scriptspace=-1pt 
  \int_{O_j}G(\mathbf{x},\mathbf{y})\tilde{f}_{j+1\lovlpint}
  (\mathbf{y})\,\mathrm{d}\mathbf{y}=
  \int_{\Gamma_{j+1}}\hspace{-0.8em}
  -\left[\mathbf{n}_j^T\alpha\nabla v_{j+1}(\mathbf{y})\right]
  G(\mathbf{x},\mathbf{y})+\left[\mathbf{n}_j^T\alpha^T\nabla_{\mathbf{y}}
    G(\mathbf{x},\mathbf{y}) \right]v_{j+1}(\mathbf{y})
  \,\mathrm{d}\sigma(\mathbf{y}).
\end{equation}
The idea for transferring the source in $O_j$ to $O_{j+1}$ consists in
a secondary integration by parts but from $\Gamma_{j+1}$ to $O_{j+1}$.
This will involve another surface integral
\begin{equation}\label{eq:gammaj+2}
\int_{\Gamma_{j+2}}\left[\mathbf{n}_{j+1}^T\alpha\nabla
  v_{j+1}(\mathbf{y})\right]
G(\mathbf{x},\mathbf{y})-\left[\mathbf{n}_{j+1}^T\alpha^T\nabla_{\mathbf{y}}
  G(\mathbf{x},\mathbf{y}) \right]v_{j+1}(\mathbf{y})
\,\mathrm{d}\sigma(\mathbf{y}).
\end{equation}
Since $\Gamma_{j+2}$ is {\it not} a transparent boundary for
$G(\mathbf{x},\mathbf{y})$, $\mathbf{y}\in O_{j+1}$, $\mathbf{x}\in
O_l$ ($l> j+1$), the above surface integral in general does not
vanish.  Note that, however, the information to be transferred is all
from $\Gamma_{j+1}$, which is provided by the Dirichlet and Neumann
traces of $v_{j+1}$; see \Cref{eq:gammaj+1}.  So $v_{j+1}$ can be
modified in $O_{j+1}$ without changing this information while letting
the Dirichlet and Neumann traces on $\Gamma_{j+2}$ vanish to get rid
of \Cref{eq:gammaj+2}. The authors use a function $\beta_{j+1}$ for
this purpose, which smoothly damps $v_{j+1}$ from $\Gamma_{j+1}$ to
$\Gamma_{j+2}$ such that
\begin{equation}\label{eq:beta}
  \medmuskip=-1mu
  \thinmuskip=-1mu
  \thickmuskip=-1mu
  \nulldelimiterspace=-0pt
  \scriptspace=0pt 
\begin{aligned}
  &\beta_{j+1}=1,\mbox{ }
  (\alpha^T\mathbf{n}_j)^T\nabla\beta_{j+1}=0\mbox{ on }\Gamma_{j+1},\quad
  \beta_{j+1}=0,\mbox{ }
  (\alpha^T\mathbf{n}_j)^T\nabla\beta_{j+1}=0\mbox{ on }\Gamma_{j+2},\\
&\mathcal{B}(\beta_{j+1} v_{j+1})=0\mbox{ on }\partial O_{j+1}\cap\partial\Omega.
\end{aligned}
\end{equation}
Otherwise, the precise shape of the function $\beta_{j+1}$ has 
no influence on the algorithm. Substituting $\beta_{j+1} v_{j+1}$ for
$v_{j+1}$ in the r.h.s. of \Cref{eq:gammaj+1} and \Cref{eq:gammaj+2}
and summing up, we find
\begin{displaymath}
  \medmuskip=-3.5mu
  \thinmuskip=-3.5mu
  \thickmuskip=-3.5mu
  \nulldelimiterspace=-1pt
  \scriptspace=-1.6pt
  \int\limits_{O_j}G(\mathbf{x},\mathbf{y})\tilde{f}_{j+1\lovlpint}(\mathbf{y})\,\mathrm{d}\mathbf{y}=\hspace{-0.6em}
  \int\limits_{\partial O_{j+1}}\hspace{-0.4em}\left[\mathbf{n}^T\alpha\nabla (\beta_{j+1}
    v_{j+1})(\mathbf{y})\right]G(\mathbf{x},\mathbf{y})-
  \left[\mathbf{n}^T\alpha^T\nabla_{\mathbf{y}} G(\mathbf{x},\mathbf{y})
  \right](\beta_{j+1} v_{j+1})(\mathbf{y}) \,\mathrm{d}\sigma(\mathbf{y}).
\end{displaymath}
Integrating by parts for the r.h.s. and using
$\mathcal{M}_{\mathbf{y}}G(\mathbf{x},\mathbf{y})=0$ in $O_{j+1}$ for
$\mathbf{x}\not\in O_{j+1}$ yields the identity
\begin{displaymath}
  \medmuskip=-1mu
  \thinmuskip=-1mu
  \thickmuskip=-1mu
  \nulldelimiterspace=-0pt
  \scriptspace=0pt
  \int_{O_j}G(\mathbf{x},\mathbf{y})\tilde{f}_{j+1\lovlpint}
  (\mathbf{y})\,\mathrm{d}\mathbf{y}=
  \int_{O_{j+1}}G(\mathbf{x},\mathbf{y})\left\{\nabla^T\left[\alpha\nabla
      (\beta_{j+1} v_{j+1})(\mathbf{y})\right]+\frac{\omega^2}{\kappa}
    \beta_{j+1}v_{j+1}(\mathbf{y})\right\}\,\mathrm{d}\mathbf{y}.
\end{displaymath}
The sum inside the above curly-braces is thus the source transfer
function $\Psi_{j+1}(\tilde{f}_{j+1\lovlpint})$ we were looking for, see
\Cref{eq:stbase}. We can now define the source transfer method, see
\Cref{alg:STPDE} at the PDE level and \Cref{alg:STMAT} at the matrix
level.  Despite the quite different motivating ideas, these algorithms
look very similar to \Cref{alg:GDCPDE} and \Cref{alg:GDCMAT}. There is
one difference though: in the forward sweep, the source transfer
methods do not use the source in the right overlap $O_j$ for the
subproblem in $\Omega_j$. One can however regard $O_{j-1}$ as the
physical subdomain, and consider $O_j$ to belong already to the
PML region so that $O_j\cup\Omega_j^{pml}$ becomes the PML region for
$O_{j-1}$.  Then the source transfer methods can be derived
also directly as DOSMs in the {\it local} deferred correction form, and
we obtain the following equivalence results.
\begin{algorithm}
  \caption{{\bf Source transfer} preconditioner at the {\bf PDE} level
    (\cite[Algorithm 3.1-3.2]{Chen13a})}
  \label{alg:STPDE}
  Input the source $f$ and $g$.  Let $\tilde{f}_{1}\gets f$ in $O_0$
  and $\tilde{f}_{1}\gets0$ otherwise.  Solve successively for
  $j=1,..,J-1$,
  \begin{displaymath}
    \small
  \begin{array}{r@{\hspace{0.2em}}c@{\hspace{0.2em}}ll}
    \mathcal{L}\,v_j&=&\tilde{f}_{j} & \mbox{ in }\Omega_j,\\
    \mathcal{B}\,v_j&=&g & \mbox{ on }
    \partial\Omega\cap\partial\Omega_j,\\
    \mathbf{n}_j^T\alpha\nabla v_j+
    \mathrm{DtN}_{j}^{pml}v_j&=&0 & \mbox{ on }\Gamma_{j,j-1},\\
    \mathbf{n}_j^T\alpha\nabla v_j+
    \mathrm{DtN}_{j}^{pml}v_j&=&0 & \mbox{ on }\Gamma_{j,j+1},
  \end{array}
  \end{displaymath}
  where ${\mathrm{DtN}}_{j}^{pml}$ is defined by the PML along $\Gamma_{j,j-1}$
  and $\Gamma_{j,j+1}$ (see \Cref{rem:impml} for practical
  implementation), and after each solve we let $\tilde{f}_{j+1}\gets f -
  \mathcal{L}(\beta_jv_j)$
  in $O_{j}$ and $\tilde{f}_{j+1}\gets0$ otherwise, and $\beta_{j}$ satisfies
  \Cref{eq:beta}.

  Let $\tilde{f}_j$ unchanged in $O_{j-1}$ but $\tilde{f}_j\gets f$ in $O_{j}$.
  Solve in order of $j=J,..,1$ the problem
  \begin{displaymath}
    \small
  \begin{array}{r@{\hspace{0.2em}}c@{\hspace{0.2em}}ll}
    \mathcal{L}\,\tilde{u}_j&=&\tilde{f}_{j} & \mbox{ in }\Omega_j,\\
    \mathcal{B}\,\tilde{u}_j&=&g & \mbox{ on }
    \partial\Omega\cap\partial\Omega_j,\\
    \mathbf{n}_j^T\alpha\nabla \tilde{u}_j+
    \mathrm{DtN}_{j}^{pml}\tilde{u}_j&=&0 &
    \mbox{ on }\Gamma_{j,j-1},\\
    \tilde{u}_j&=&\tilde{u}_{j+1} & \mbox{ on }\Gamma_{j,j+1}.
  \end{array}
  \end{displaymath}

  Output the global approximation $\tilde{u}\gets \tilde{u}_j$ in $O_j$, $j=1,..,J$ and
  $\tilde{u}\gets \tilde{u}_1$ in $O_0$.
\end{algorithm}

\begin{algorithm}
  \caption{{\bf Source transfer} preconditioner at the {\bf matrix} level}
  \label{alg:STMAT}
  Input the r.h.s. $\mathbf{f}$.  Let
  $\tilde{\mathbf{f}}_{1\lovlpint}\gets{\mathbf{f}}_{1\thetaint}$.  Solve
  successively for $j=1,..,J-1$,

  \begin{displaymath}
    \small
    \left[
      \begin{array}{ccccc}
        \tilde{S}_{j\langle}^{\langle} & A_{j\langle j-1\rovlpint} & & &\\
        A_{j-1\rovlpint j\langle} & A_{j-1\rovlpint} & A_{j-1\rovlpint j+1\langle} & &\vphantom{_j^j}\\
        & A_{j+1\langle j-1\rovlpint } & A_{j+1\langle} & A_{j+1\langle j\rovlpint } &\vphantom{_j^j}\\
        & & A_{j\rovlpint  j+1\langle} & A_{j\rovlpint } & A_{j\rovlpint j\rangle}\vphantom{_j^j}\\
        & & & A_{j\rangle j\}} & \tilde{S}_{j\rangle}^{\rangle}
      \end{array}
    \right]
    \left[
      \begin{array}{c}
        \mathbf{v}_{j\langle} \\ \mathbf{v}_{j\lovlpint}
        \vphantom{\tilde{S}_{j\langle}^{\langle}} \\
        \mathbf{v}_{j[} \\ \mathbf{v}_{j\rovlpint} \\
        \mathbf{v}_{j\rangle}\vphantom{\tilde{S}_{j\langle}^{\langle}}
      \end{array}
    \right] =
    \left[
      \begin{array}{c}
        0\vphantom{A_j^j} \\ \mathbf{\tilde{f}}_{j\lovlpint }\vphantom{A_j^j} \\
        \mathbf{f}_{j+1\langle}\vphantom{A_j^j} \\
        0\vphantom{A_j^j} \\ 0\vphantom{A_j^j}
      \end{array}
    \right],
  \end{displaymath}
  where $\tilde{S}_{j\langle}^{\langle}$ and $\tilde{S}_{j\rangle}^{\rangle}$
  are defined by the PML (see \Cref{rem:imschur} for practical
  implementation), and after each solve we let
  $\mathbf{\tilde{f}}_{j+1\lovlpint }\gets\mathbf{f}_{j\rovlpint }-A_{j\rovlpint }I^j_{j\rovlpint }(D_j\mathbf{v}_j)-
  A_{j\rovlpint j+1\langle}I^j_{j+1\langle}(D_j\mathbf{v}_j)$, where $D_j$ satisfies
  \begin{equation}\label{eq:STD}
  I^j_{j\rangle}D_j=0,\, A_{j\rangle j\rovlpint }I^j_{j\rovlpint }D_j=0,\,
  I^j_{j+1\langle}(D_j-I_j)=0, A_{j+1\langle j\rovlpint }I^j_{j\rovlpint }(D_j-I_j)=0.
  \end{equation}

  Solve in order of $j=J,..,1$ the following problem
  \begin{displaymath}
    \small
    \left[
      \begin{array}{ccccc}
        \tilde{S}_{j\langle}^{\langle} & A_{j\langle j-1\rovlpint } & & &\\
        A_{j-1\rovlpint j\langle} & A_{j-1\rovlpint } & A_{j-1\rovlpint j+1\langle} & &\vphantom{_j^j}\\
        & A_{j+1\langle j-1\rovlpint } & A_{j+1\langle} & A_{j+1\langle j\rovlpint } &\vphantom{_j^j}\\
        & & A_{j\rovlpint  j+1\langle} & A_{j\rovlpint } & A_{j\rovlpint j\rangle}\vphantom{_j^j}\\
        & & & 0 & I_{j\rangle}
      \end{array}
    \right]
    \left[
      \begin{array}{c}
        \tilde{\mathbf{u}}_{j\langle} \\ \tilde{\mathbf{u}}_{j\lovlpint}
        \vphantom{\tilde{S}_{j\langle}^{\langle}}\\
        \tilde{\mathbf{u}}_{j]} \\ \tilde{\mathbf{u}}_{j\rovlpint} \\
        \tilde{\mathbf{u}}_{j\rangle}\vphantom{\tilde{S}_{j\langle}^{\langle}}
      \end{array}
    \right] =
    \left[
      \begin{array}{c}
        0\vphantom{A_j^j} \\ \mathbf{\tilde{f}}_{j\lovlpint }\vphantom{A_j^j} \\
        \mathbf{f}_{j+1\langle}\vphantom{A_j^j} \\
        \mathbf{f}_{j\rovlpint }\vphantom{A_j^j} \\
        \tilde{\mathbf{u}}_{j+1]}\vphantom{A_j^j}
      \end{array}
    \right].
  \end{displaymath}

  Output the global approximation $\mathbf{\tilde{u}}\gets\tilde{\mathbf{u}}_j$ in
  $\overline{O}_j$, $j=1,..,J$ and $\mathbf{\tilde{u}}\gets\tilde{\mathbf{u}}_1$ in
  $O_0$.
\end{algorithm}

\begin{figure}
  \centering
  \includegraphics[scale=.5,trim=10 95 160 95,clip]{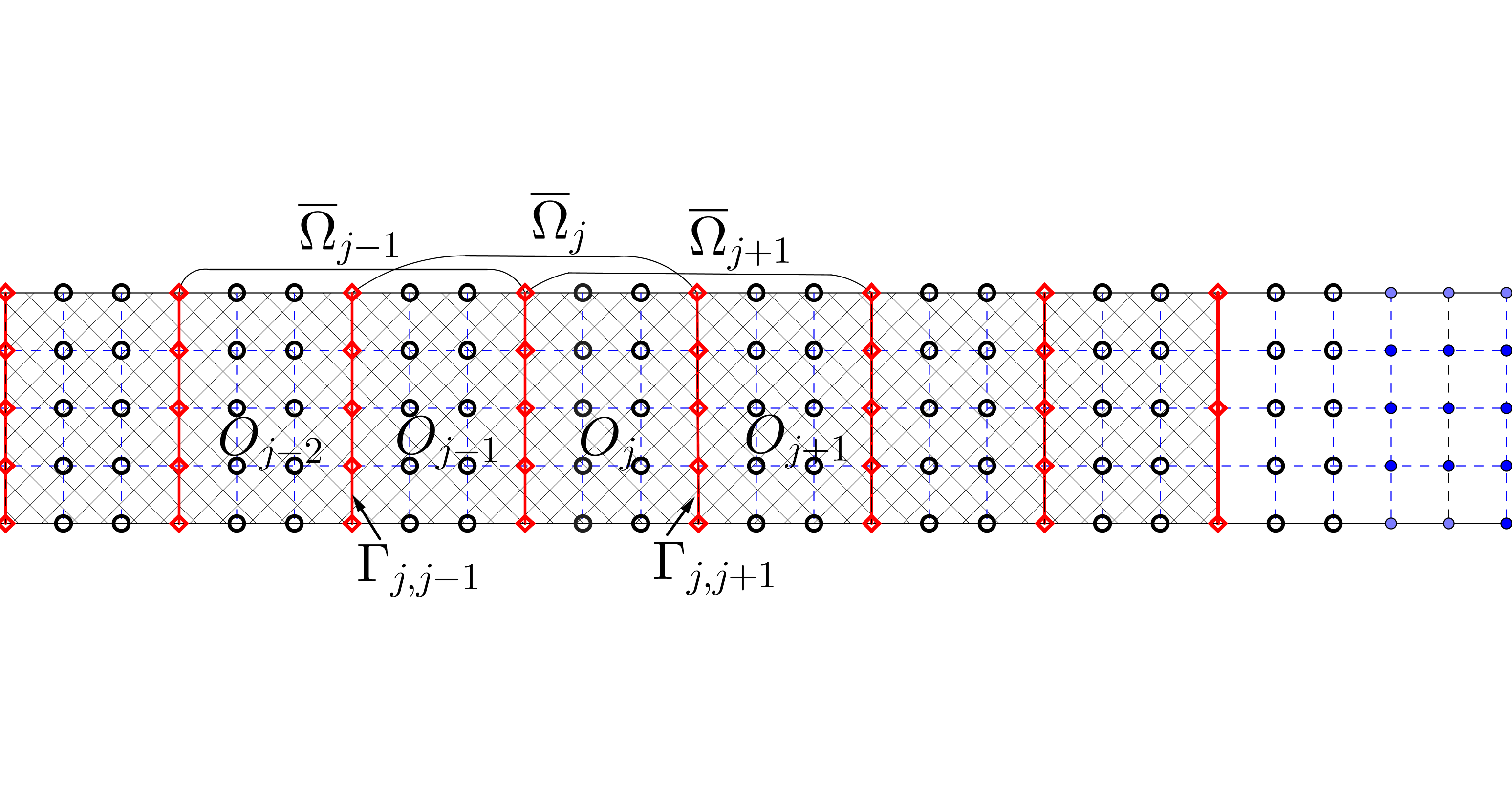}%
  \caption{Overlapping domain decomposition with contacting overlaps,
    $\circ \in O_*$, ${\color{red}\Diamond}\in \Gamma_{*,\#}$.}
  \label{fig:stddm}
\end{figure}  

\begin{theorem}\label{thm:stpde}
  Suppose the subproblems of the source transfer \Cref{alg:STPDE} are
  well-posed. Then \Cref{alg:STPDE} is equivalent to the DOSM
  preconditioner shown in \Cref{alg:DOSMPDE} in the specific case
    where each subdomain consists of two contacting overlaps, see
  \Cref{fig:stddm}, and using PML transmission conditions on the
  interfaces in the forward sweep and Dirichlet instead of PML on the
  right interfaces in the backward sweep; moreover, in the forward
  sweep the source terms in the right overlap of each subdomain are
  put to zero, which turns the right overlap effectively into part of the
  PML on the right of the subdomain.
\end{theorem}

\begin{proof}
  We will prove several identities between the iterates generated by
  the two algorithms. Since we are considering the algorithms as
  preconditioners, we use a zero initial guess for \Cref{alg:DOSMPDE}.
  We start from the iterate of \Cref{alg:DOSMPDE} for the specific
  case stated in the theorem:
  \begin{displaymath}
  \begin{array}{r@{\hspace{0.2em}}c@{\hspace{0.2em}}ll}
    \mathcal{L}\,u_1^{(\frac{1}{2})}&=&\hat{f}_{1} & \mbox{ in }\Omega_1,\\
    \mathcal{B}\,u_1^{(\frac{1}{2})}&=&g & \mbox{ on }
    \partial\Omega\cap\partial\Omega_1,\\
    \mathbf{n}_1^T\alpha\nabla u_1^{(\frac{1}{2})}+
    \mathrm{DtN}_{1}^{pml}u_1^{(\frac{1}{2})}&=&0 & \mbox{ on }\Gamma_{1,2},
  \end{array}
  \end{displaymath}
  where $\hat{f}_{j}=f$ in $O_{j-1}$ and $\hat{f}_{j}=0$
  otherwise, for $j=1,..,J-1$.  This is exactly what we solve for $v_1$ in
  \Cref{alg:STPDE}.  By uniqueness of the solution, we see that
  $u_1^{(\frac{1}{2})}=v_1$. Assuming now that
  $u_{j-1}^{(\frac{1}{2})}=v_{j-1}$ in $O_{j-1}$, we will show that
  this relation also holds for $j$ replacing $j-1$.  In the $j$-th
  forward substep of \Cref{alg:DOSMPDE}, we solve
  \begin{equation}\label{eq:stoseq}
  \begin{array}{r@{\hspace{0.2em}}c@{\hspace{0.2em}}ll}
    \mathcal{L}\,u_j^{(\frac{1}{2})}&=&\hat{f}_{j} & \mbox{ in }\Omega_j,\\
    \mathcal{B}\,u_j^{(\frac{1}{2})}&=&g & \mbox{ on }
    \partial\Omega\cap\partial\Omega_j,\\
    \mathbf{n}_j^T\alpha\nabla u_j^{(\frac{1}{2})}+
    \mathrm{DtN}_{j}^{pml}u_j^{(\frac{1}{2})}&=&
    \mathbf{n}_j^T\alpha\nabla u_{j-1}^{(\frac{1}{2})}+
    \mathrm{DtN}_{j}^{pml}u_{j-1}^{(\frac{1}{2})} & \mbox{ on }\Gamma_{j,j-1},\\
    \mathbf{n}_j^T\alpha\nabla u_j^{(\frac{1}{2})}+
    \mathrm{DtN}_{j}^{pml}u_j^{(\frac{1}{2})}&=&0 & \mbox{ on }\Gamma_{j,j+1}.
  \end{array}
  \end{equation}
  We extend $\beta_{j-1}u_{j-1}^{(\frac{1}{2})}$ by zero into $O_j$, and
  substitute the correction
  $u_j^{(\frac{1}{2})}-\beta_{j-1}u_{j-1}^{(\frac{1}{2})}$ for
  $u_j^{(\frac{1}{2})}$ into \Cref{eq:stoseq}. Using \Cref{eq:beta}, we 
  find that the BVP satisfied by the correction term is
  \begin{displaymath}
  \begin{array}{r@{\hspace{0.2em}}c@{\hspace{0.2em}}ll}
    \mathcal{L}\left(u_j^{(\frac{1}{2})}-\beta_{j-1}u_{j-1}^{(\frac{1}{2})}
    \right)&=&\hat{f}_{j}-\mathcal{L}(\beta_{j-1}u_{j-1}^{(\frac{1}{2})}) &
    \mbox{ in }\Omega_j,\\
    \mathcal{B}\left(u_j^{(\frac{1}{2})}-\beta_{j-1}u_{j-1}^{(\frac{1}{2})}
    \right)&=&g & \mbox{ on }
    \partial\Omega\cap\partial\Omega_j,\\
    \left(\mathbf{n}_j^T\alpha\nabla + \mathrm{DtN}_{j}^{pml}\right)
    \left(u_j^{(\frac{1}{2})}-\beta_{j-1}u_{j-1}^{(\frac{1}{2})}\right)&=&0
    & \mbox{ on }\Gamma_{j,j-1},\\
    \left(\mathbf{n}_j^T\alpha\nabla+\mathrm{DtN}_{j}^{pml}\right)
    \left(u_j^{(\frac{1}{2})}-\beta_{j-1}u_{j-1}^{(\frac{1}{2})}\right)&=&0
    & \mbox{ on }\Gamma_{j,j+1}.
  \end{array}
  \end{displaymath}
  Since we assumed that $u_{j-1}^{(\frac{1}{2})}=v_{j-1}$ in
  $O_{j-1}=\Omega_j\cap\Omega_{j-1}$, the above BVP is exactly the
  same as the BVP for $v_{j}$ in \Cref{alg:STPDE}.  By uniqueness of
  the solution, we thus deduce that
  $v_j=u_j^{(\frac{1}{2})}-\beta_{j-1}u_{j-1}^{(\frac{1}{2})}$ in
  $\Omega_j$, and in particular $v_j=u_j^{(\frac{1}{2})}$ in $O_j$.
  By induction, this last relation then holds for all $j=1,..,J-1$.
  In the backward sweep of \Cref{alg:DOSMPDE}, we solve the subdomain
  problem
  \begin{displaymath}
  \begin{array}{r@{\hspace{0.2em}}c@{\hspace{0.2em}}ll}
    \mathcal{L}\,u_j^{(1)}&=&f & \mbox{ in }\Omega_j,\\
    \mathcal{B}\,u_j^{(1)}&=&g & \mbox{ on }
    \partial\Omega\cap\partial\Omega_j,\\
    \mathbf{n}_j^T\alpha\nabla u_j^{(1)}+
    \mathrm{DtN}_{j}^{pml}u_j^{(1)}&=&
    \mathbf{n}_j^T\alpha\nabla u_{j-1}^{(\frac{1}{2})}+
    \mathrm{DtN}_{j}^{pml}u_{j-1}^{(\frac{1}{2})} & \mbox{ on }\Gamma_{j,j-1},\\
    u_j^{(1)}&=&u_{j+1}^{(1)} & \mbox{ on }\Gamma_{j,j+1}.
  \end{array}
  \end{displaymath}
  By the same argument as before, we can show that $\tilde{u}_j=u_j^{(1)}$ in
  $O_j$ for $j=J,..,1$ and $\tilde{u}_1=u_1^{(1)}$ in $O_0$.
\end{proof}

\begin{theorem}\label{thm:stmat}
  Assume that the subproblems of the discrete source transfer
  \Cref{alg:STMAT} are well-posed.  Then \Cref{alg:STMAT} is
  equivalent to the discrete DOSM preconditioner shown in
  \Cref{alg:DOSMMAT}, in the specific case at the discrete
   level corresponding
    to the case stated in \Cref{thm:stpde} at the continous level.
\end{theorem}

\begin{proof}
  In the specific case, and with zero initial guess for its use as a
  preconditioner, \Cref{alg:DOSMMAT} first solves
  \begin{displaymath}
  \left[
    \begin{array}{cccc}
      A_{1\thetaint} & A_{1\thetaint2\langle} & &\vphantom{_{j\thetaint}^{(\frac{1}{2})}}\\
      A_{2\langle 1\thetaint} & A_{2\langle} & A_{2\langle 1\rovlpint} &\vphantom{_{1[}^{(\frac{1}{2})}}\\
      & A_{1\rovlpint 2\langle} & A_{1\rovlpint} & A_{1\rovlpint1\rangle}\vphantom{_{j\rovlpint}^{(\frac{1}{2})}}\\
      & & A_{1\rangle 1\rovlpint} & \tilde{S}_{1\rangle}^{\rangle}\vphantom{_{1\rangle}^{(\frac{1}{2})}}
    \end{array}
  \right]
  \left[
    \begin{array}{c}
      \mathbf{u}_{1\thetaint}^{(\frac{1}{2})} \\
      \mathbf{u}_{1[}^{(\frac{1}{2})} \\
      \mathbf{u}_{1\rovlpint}^{(\frac{1}{2})} \\
      \mathbf{u}_{1\rangle}^{(\frac{1}{2})}
    \end{array}
  \right] =
  \left[
    \begin{array}{c}
      \mathbf{f}_{1\thetaint}\vphantom{A_{j\thetaint}^{(\frac{1}{2})}} \\
      \mathbf{f}_{2\langle}\vphantom{A_{j\thetaint}^{(\frac{1}{2})}} \\
      0\vphantom{A_{j\rovlpint}^{(\frac{1}{2})}} \\ 0\vphantom{A_{j\rangle}^{(\frac{1}{2})}}
    \end{array}
  \right],
  \end{displaymath}
  which is exactly the same problem for $\mathbf{v}_1$ in
  \Cref{alg:STMAT}; so $\mathbf{u}_1^{(\frac{1}{2})}=\mathbf{v}_1$ by
  uniqueness of the solution. Now assuming that
  \begin{equation}\label{eq:stositer}
    \mathbf{u}_{j-1\rovlpint}^{(\frac{1}{2})}=\mathbf{v}_{j-1\rovlpint},\;
    \mathbf{u}_{j-1[}^{(\frac{1}{2})}=\mathbf{v}_{j-1[},
  \end{equation}
  we will show these relations also hold for $j+1$ replacing $j$.
  In the specific version of \Cref{alg:DOSMMAT},
  $\mathbf{u}_{j}^{(\frac{1}{2})}$ solves
  \begin{displaymath}
    \small
    \arraycolsep0.1em
    \medmuskip=-1mu
    \thinmuskip=-1mu
    \thickmuskip=-1mu
    \nulldelimiterspace=-1pt
    \scriptspace=0pt
    \left[
      \begin{array}{ccccc}
        \tilde{S}_{j\langle}^{\langle} & A_{j\langle j-1\rovlpint} & & & \vphantom{_{j\langle}^{\frac{1}{2}}}\\
        A_{j-1\rovlpint j\langle} & A_{j-1\rovlpint} & A_{j-1\rovlpint j+1\langle} & &\vphantom{_{j\lovlpint}^{(\frac{1}{2})}}\\
      & A_{j+1\langle j-1\rovlpint} & A_{j+1\langle} & A_{j+1\langle j\rovlpint} &\vphantom{_{j[}^{(\frac{1}{2})}}\\
      & & A_{j\rovlpint j+1\langle} & A_{j\rovlpint} & A_{j\rovlpint j\rangle}\vphantom{_{j\rovlpint}^{(\frac{1}{2})}}\\
      & & & A_{j\rangle j\rovlpint} & \tilde{S}_{j\rangle}\vphantom{_{j\rangle}^{(\frac{1}{2})}}
      \end{array}
    \right]\hspace{-0.5em}
    \left[
      \begin{array}{c}
        \mathbf{u}_{j\langle}^{(\frac{1}{2})} \\
        \mathbf{u}_{j\lovlpint}^{(\frac{1}{2})} \\
        \mathbf{u}_{j[}^{(\frac{1}{2})} \\
        \mathbf{u}_{j\rovlpint}^{(\frac{1}{2})} \\
        \mathbf{u}_{j\rangle}^{(\frac{1}{2})}
      \end{array}
    \right]\hspace{-0.2em}=\hspace{-0.2em}
    \left[
      \begin{array}{c}
        \mathbf{f}_{j\langle}+
        (\tilde{S}_{j\langle}^{\langle}-A_{j\langle})\mathbf{u}_{j-1[}^{(\frac{1}{2})}
        -A_{j\langle j-2\rovlpint}\mathbf{u}_{j-1\lovlpint}^{(\frac{1}{2})}
        \vphantom{_{j\langle}^{(\frac{1}{2})}}\\
        \mathbf{f}_{j\lovlpint}\vphantom{A_j^j}\vphantom{_{j\langle}^{(\frac{1}{2})}}\\
        \mathbf{f}_{j[}\vphantom{A_j^j} \vphantom{_{j\langle}^{(\frac{1}{2})}}\\
        0\vphantom{_{j\langle}^{(\frac{1}{2})}} \\ 0\vphantom{_{j\langle}^{(\frac{1}{2})}}
      \end{array}
    \right].
  \end{displaymath}
  We can further rewrite the above system into the equivalent system
  \begin{displaymath}
    \small
    \arraycolsep0.1em
    \begin{array}{l}
      \left[
      \begin{array}{ccccc}
        \tilde{S}_{j\langle}^{\langle} & A_{j\langle j-1\rovlpint} & &
        &\vphantom{_{j\langle}^{(\frac{1}{2})}}\\
        A_{j-1\rovlpint j\langle} & A_{j-1\rovlpint} & A_{j-1\rovlpint j+1\langle} &
        &\vphantom{_{j\langle}^{(\frac{1}{2})}}\\
      & A_{j+1\langle j-1\rovlpint} & A_{j+1\langle} & A_{j+1\langle j\rovlpint} &\vphantom{_{j\langle}^{(\frac{1}{2})}}\\
      & & A_{j\rovlpint j+1\langle} & A_{j\rovlpint} & A_{j\rovlpint j\rangle}\vphantom{_{j\langle}^{(\frac{1}{2})}}\\
      & & & A_{j\rangle j\rovlpint} & \tilde{S}_{j\rangle}^{\rangle}\vphantom{_{j\langle}^{(\frac{1}{2})}}
      \end{array}
    \right]
    \left[
    \begin{array}{c}
      \mathbf{u}_{j\langle}^{(\frac{1}{2})}
      -I^{j-1}_{j\langle}(D_{j-1}\mathbf{u}_{j-1}^{(\frac{1}{2})})
      \\
      \mathbf{u}_{j\lovlpint}^{(\frac{1}{2})}
      -I^{j-1}_{j-1\rovlpint}(D_{j-1}\mathbf{u}_{j-1}^{(\frac{1}{2})})
      \\
      \mathbf{u}_{j[}^{(\frac{1}{2})} \\
      \mathbf{u}_{j\rovlpint}^{(\frac{1}{2})} \\
      \mathbf{u}_{j\rangle}^{(\frac{1}{2})}
    \end{array}
      \right]\\
      \qquad=
      \left[
      \begin{array}{c}
        0\\
        \mathbf{f}_{j-1\rovlpint}-A_{j-1\rovlpint}I^{j-1}_{j-1\rovlpint}
        (D_{j-1}\mathbf{u}_{j-1}^{(\frac{1}{2})})-
        A_{j-1\rovlpint j\langle}I^{j-1}_{j\langle}(D_{j-1}\mathbf{u}_{j-1}^{(\frac{1}{2})})\\
        \mathbf{f}_{j+1\langle}\vphantom{A_j^j} \\
        0\vphantom{A_j^j} \\ 0\vphantom{A_j^j}
      \end{array}
      \right],
    \end{array}
  \end{displaymath}
  where for the first row we have used
    $I^{j-1}_{j\langle}(D_{j-1}\mathbf{u}_{j-1}^{(\frac{1}{2})})=
    \mathbf{u}_{j-1[}^{(\frac{1}{2})}$,
    $A_{j\langle j-1\rovlpint}I^{j-1}_{j-1\rovlpint}(D_{j-1}$
    $\mathbf{u}_{j-1}^{(\frac{1}{2})})=A_{j\langle
      j-1\rovlpint}\mathbf{u}_{j-1\rovlpint}^{(\frac{1}{2})}$ from \Cref{eq:STD}
    and
    $A_{j\langle j-2\rovlpint}\mathbf{u}_{j-1\lovlpint}^{(\frac{1}{2})}+
    A_{j\langle}\mathbf{u}_{j-1[}^{(\frac{1}{2})}+A_{j\langle j-1\rovlpint}$
    $\mathbf{u}_{j-1\rovlpint}^{(\frac{1}{2})}=\mathbf{f}_{j\langle}$ from
    the system for $\mathbf{u}_{j-1}^{(\frac{1}{2})}$.  By \Cref{eq:stositer}
  this is exactly the same problem for $\mathbf{v}_j$ in \Cref{alg:STMAT}.  By
  induction, we have
  $\mathbf{u}_{j\rovlpint}^{(\frac{1}{2})}=\mathbf{v}_{j\rovlpint}$ and
  $\mathbf{u}_{j[}^{(\frac{1}{2})}=\mathbf{v}_{j[}$ for all $j=1,..,J-1$.  By
  similar arguments, we can also show that in the backward sweep,
  $\mathbf{u}_{j\rovlpint}^{(1)}=\tilde{\mathbf{u}}_{j\rovlpint}$ and
  $\mathbf{u}_{j[}^{(1)}=\tilde{\mathbf{u}}_{j[}$ for all $j=J,..,2$, and
  $\mathbf{u}_{1}^{(1)}=\tilde{\mathbf{u}}_{1}$.
\end{proof}

\subsection{The method using single layer potentials}
\label{sec:stolk}

Stolk summarizes in \cite{Stolk} the main ideas for the method based
on single layer potentials as follows: ``{\it A new domain
  decomposition method is introduced for the heterogeneous 2-D and 3-D
  Helmholtz equations.  Transmission conditions based on the perfectly
  matched layer (PML) are derived that avoid artificial reflections
  and match incoming and outgoing waves at the subdomain
  interfaces}''.

To motivate the method based on single layer potentials, we consider
the free space problem in a homogeneous medium.  In this case, the
source in each subdomain generates waves biased in two directions --
forward and backward, which both are outgoing relative to the
subdomain. In each subdomain, we are interested in the sum of the
waves stimulated by {\it all} the sources, including the sources from
the other subdomains. We thus must account for the waves incoming from
all the other subdomains in each subdomain.  The idea of the algorithm
based on single layer potentials is forward propagation and
accumulation of the waves from the first subdomain to the last one so
that the waves in $\Omega_j$, $1\leq j\leq J$, have accounted for all the
forward going waves generated from the sources in $\Omega_l$ for all
$l< j$.  Then, a similar procedure can be carried out backward so that
the waves in $\Omega_j$ contain also the backward going waves
generated from the sources in $\Omega_m$ for all $m>j$.  The actual
backward sweep in \cite{Stolk} solves however for the correction from
the residual.

To transfer the waves, the author in \cite{Stolk} states: ``{\it We
  have constructed new transmission conditions [\ldots] These are
  designed to ensure that:}
\begin{itemize}
\item[(i)] \it the boundary conditions at the subdomain interfaces are
  non-reflecting;
\item[(ii)] \it if $\Omega_{j-1}$ and $\Omega_j$ are neighboring
  subdomains, then the outgoing wave field from $\Omega_{j-1}$ equals
  the incoming wave field in $\Omega_j$ at the joint boundary and vice
  versa.
\end{itemize}
{\it This is achieved in a simple and accurate way using a PML
  boundary layer added to the subdomains and single layer
  potentials}.''  As noted in \cite{Stolk}, a related approach has
been studied in \cite{Schadle} that is also motivated by matching the
incoming and outgoing waves.  There is, however, a difference in the
concrete forms to achieve the matching; see \Cref{rem:schadle} for a
brief review of the form used in \cite{Schadle}.

The representation of incoming waves is a well-studied topic in computational
electromagnetics; see \cite[pp. 185--220]{THbook}.  The common idea is to
represent the incoming wave $v_{j-1}$ from $\Omega_{j-1}$ to $\Omega_j$ as an
equivalent source term on the surface $\Gamma_{j,j-1}$.  In \cite{Stolk}, the
proposed equivalent source is
$2\delta(s_{j-1})\mathbf{n}_j^T\alpha\nabla v_{j-1}$; here $s_{j-1}$ is a local
coordinate normal to $\Gamma_{j,j-1}$ and $s_{j-1}=0$ corresponds to
$\Gamma_{j,j-1}$, and $\delta(s_{j-1})$ here represents a surface delta
  function.  The author in \cite{Stolk} states\footnote{We replaced the
  original notation in the quote by the notation used in this manuscript.}:
``{\it a short intuitive explanation goes as follows. The term
  $v_{j-1}|_{\Gamma_{j,j-1}}$ exclusively contains forward going waves because
  of the presence of a PML non-reflecting layer immediately to its
  right\footnote{In fact, the PML is placed on the right of $\Gamma_{j-1,j}$
    for $\Omega_{j-1}$.}. The term
  $2\delta(s_{j-1})\mathbf{n}_j^T\alpha\nabla v_{j-1}$ is meant to cause the
  same forward going wave field in the field $v_j$ as in the field
  $v_{j-1}$. [\ldots] the source generates waves propagating both forwardly and
  backwardly in a symmetric fashion.  The factor 2 is introduced so that the
  forward propagating part equals $v_{j-1}$ on $\Gamma_{j,j-1}$. The backward
  propagating part is absorbed in the neighboring PML layer [along
  $\Gamma_{j,j-1}$ for $\Omega_j$].}''

Let us take a closer look at the reasoning above: we want to extend
the waves $v_{j-1}$ from $\Omega_{j-1}$ to $\Omega_j$.  By
\Cref{eq:rep}, we have for $\mathbf{x}\in \Omega_{j}$,
\begin{equation}\label{eq:potentials}
  \medmuskip=-1mu
  \thinmuskip=-1mu
  \thickmuskip=-1mu
  \nulldelimiterspace=-0pt
  \scriptspace=0pt
  v_{j-1}(\mathbf{x}) = \int_{\Gamma_{j,j-1}}(\mathbf{n}_{j}^T\alpha
  \nabla v_{j-1})(\mathbf{y}){G}(\mathbf{x},\mathbf{y})
  -\mathbf{n}_j^T(\mathbf{y})\left(\alpha^T(\mathbf{y})\nabla_{\mathbf{y}}
    G(\mathbf{x},\mathbf{y})\right)v_{j-1}(\mathbf{y})\,
  \mathrm{d}\sigma(\mathbf{y}).
\end{equation}
The Green's function $G(\mathbf{x},\mathbf{y})$ represents the wave
field at $\mathbf{y}\in\Gamma_{j,j-1}$ stimulated by a point source at
$\mathbf{x}\in\Omega_j$.  For the free space problem in a
homogeneous medium, if two point sources at $\mathbf{x}$
and $\mathbf{x}'$ are symmetric w.r.t. $\Gamma_{j,j-1}$, then the
stimulated waves generated by the individual point sources are also
symmetric w.r.t. $\Gamma_{j,j-1}$\footnote{This is like 
  in the method of images for solving PDEs in a half
  space; c.f. \cite{zauderer}.}. 
Hence,
\begin{equation}\label{eq:Gsym}
  \begin{array}{l}
    \forall\mathbf{x}\in\Omega_j,\,\mathbf{x}'\mbox{ and }\mathbf{x}
    \mbox{ symmetric w.r.t. }\Gamma_{j,j-1},\,\forall\mathbf{y}\in\Gamma_{j,j-1}:\\
    G(\mathbf{x},\mathbf{y})=G(\mathbf{x}',\mathbf{y}),\,
    \mathbf{n}_j^T(\mathbf{y})\left(\alpha^T(\mathbf{y})\nabla_{\mathbf{y}}
      (G(\mathbf{x},\mathbf{y})+G(\mathbf{x}',\mathbf{y}))\right)=0.
  \end{array}
\end{equation}
Since both $v_{j-1}$ and $G(\mathbf{x}',\cdot)$ satisfy homogeneous
Helmholtz equations in $\Omega_j$ and represent outgoing waves to the right of
$\Gamma_{j,j+1}$, by using Green's identity we get
\begin{displaymath}
  \medmuskip=-1mu
  \thinmuskip=-1mu
  \thickmuskip=-1mu
  \nulldelimiterspace=-0pt
  \scriptspace=0pt
  \int_{\Gamma_{j,j-1}}\mathbf{n}_j^T(\mathbf{y})
  \left(\alpha^T(\mathbf{y})\nabla_{\mathbf{y}}G(\mathbf{x}',\mathbf{y})\right)
  v_{j-1}(\mathbf{y})\,\mathrm{d}\sigma(\mathbf{y})
  =\int_{\Gamma_{j,j-1}}\mathbf{n}_{j}^T(\mathbf{y})(\alpha \nabla
  v_{j-1}(\mathbf{y})){G}(\mathbf{x}',\mathbf{y})
  \,\mathrm{d}\sigma(\mathbf{y}).
\end{displaymath}
Substituting this and \Cref{eq:Gsym} into \Cref{eq:potentials}, we get the
single layer potential representation
\begin{displaymath}
v_{j-1}(\mathbf{x}) = \int_{\Gamma_{j,j-1}}2\,
\mathbf{n}_{j}^T(\mathbf{y})(\alpha\nabla v_{j-1}(\mathbf{y}))
{G}(\mathbf{x},\mathbf{y})\,\mathrm{d}\sigma(\mathbf{y}),\,\,\mathbf{x}\in\Omega_j,
\end{displaymath}
which is equivalent to being stimulated by the surface source
\begin{displaymath}
v_{j-1}(\mathbf{x}) = \int_{\Omega_j} 2\delta(s_{j-1}(\mathbf{y}))
\mathbf{n}_j^T(\mathbf{y})(\alpha\nabla v_{j-1}(\mathbf{y})) G(\mathbf{x},\mathbf{y})
\,\mathrm{d}\mathbf{y},\,\,\mathbf{x}\in\Omega_j.
\end{displaymath}
This, under the symmetry assumption \Cref{eq:Gsym}, justifies the equivalent
source proposed in \cite{Stolk}.  One can also calculate the representation in
closed form; c.f. \cite{Pike}.

The forward sweep in the method based on single layer potentials is
performed up to the last subdomain $\Omega_J$.  Then, a global
approximation is defined by setting $v:=v_j$ in $\Omega_j$, $j=1,..,J$
(the subdomains are non-overlapping), and a deferred correction
problem will then be solved in the backward sweep. Note that $v$ has
in general jumps across interfaces between subdomains, and the
residual $\tilde{f}:=f-\mathcal{L}u$ involves a very singular
distribution -- the derivative of the surface delta function,
$\delta'(s_{j-1})$.  To avoid the potential obscurity of deciding to
which subdomain such a singular distribution on an {\it interface}
belongs, it is suggested in \cite{Stolk} to use for the backward sweep
another set of non-overlapping subdomains that contains these
distribution residuals in the interior of the subdomains. These
residuals are then taken as new sources which stimulate the correction
wave field.  A procedure similar to the forward sweep but from the
last subdomain to the first one is carried out to find an
approximation of the correction.  Adding the correction to the wave
field previously obtained with the forward sweep gives a global
approximate solution of the original problem, which finalizes the
definition of the preconditioner based on single layer potentials.
Using two sets of subdomains could potentially double the cost of
factorization.  This situation can be avoided by the following trick
from \cite{Stolk}:  for the first set of subdomains the PML along
$\Gamma_{j,j+1}$ begins with using the original operator near
$\Gamma_{j,j+1}$, and only after a certain distance, say one mesh cell,
it changes to the PML modified operator.  The second set of subdomains
are defined by moving $\Gamma_{j,j+1}$ forward by one mesh cell, but
keeping the PML augmented region $\Omega_j\cup\Omega_j^{pml}$ the same
as before (i.e. just moving outward the interfaces that separate
$\Omega_j$ and $\Omega_j^{pml}$ so that $\Omega_j$ gets bigger and
$\Omega_j^{pml}$ gets smaller).  In this way, every subdomain matrix
is unchanged from the forward sweep to the backward sweep and the
same factorization can be used.

We summarize the preconditioner based on single layer potentials in
\Cref{alg:stolkpde} at the PDE level and in \Cref{alg:stolkmat} at the
matrix level. For the matrix version, we give two variants: the first
one, originally presented in \cite{Stolk}, is especially designed for
the 5-point (or 7-point in 3-D) finite difference scheme with $u$
discretized at integer grid points and the continuous interfaces
located at half grid points.  The second form is motivated by finite
element methods with the discrete interfaces superposed on the
continuous interfaces.  While the first matrix form can be explained
as a special discretization of the PDE \Cref{alg:stolkpde}, the second
matrix form is equivalent (under some assumptions) to the first matrix
form, and unlike the PDE form the second matrix form uses the same
partition for the forward and the backward sweeps.  Both matrix forms
do the extension and restriction like ASH does; see \Cref{rem:OSHM}.
\begin{algorithm}
  \caption{{\bf Single layer potential} preconditioner at the {\bf PDE} level}
  \label{alg:stolkpde}
  Input the source terms $f$ and $g$. Suppose the decomposition is
  non-overlapping.\\
  Solve successively for $j=1,..,J$,
  \begin{displaymath}
    \small
  \begin{array}{r@{\hspace{0.2em}}c@{\hspace{0.2em}}ll}
    \mathcal{L}\,v_j&=&f+2\delta(s_{j-1})\mathbf{n}_j^T\alpha\nabla v_{j-1}
    & \mbox{ in }\Omega_j,\\
    \mathcal{B}\,v_j&=&g & \mbox{ on }
    \partial\Omega\cap\partial\Omega_j,\\
    \mathbf{n}_j^T\alpha\nabla v_j+
    \mathrm{DtN}_{j}^{pml}v_j&=&0 & \mbox{ on }\Gamma_{j,j-1},\\
    \mathbf{n}_j^T\alpha\nabla v_j+
    \mathrm{DtN}_{j}^{pml}v_j&=&0 & \mbox{ on }\Gamma_{j,j+1},
  \end{array}
  \end{displaymath}
  where ${\mathrm{DtN}}_{j}^{pml}$ is the PML--DtN operator along
  $\Gamma_{j,j-1}$ and $\Gamma_{j,j+1}$ (see \Cref{rem:impml} for practical
  implementation), $s_{j-1}$ is a local coordinate normal to $\Gamma_{j,j-1}$
  and $s_{j-1}=0$ corresponds to $\Gamma_{j,j-1}$, and $\delta(s_{j-1})$ here
  represents a surface delta function.  The PML for $\Omega_j$ along
  $\Gamma_{j,j+1}$ has a small starting zone between $\Gamma_{j,j+1}$ and
  $\Gamma_{j,j+1}^h$ where the original operator $\mathcal{L}$ is used.

  Let $v\gets v_j$ in $\Omega_j$, $j=1,..,J$ and compute the residual
  $\tilde{f}\gets f-\mathcal{L}v$ in $\Omega$.

  Denote by $\tilde{\Omega}_j$ the resulting subdomain by moving forward the
  boundaries $\Gamma_{j,j+1}$ to $\Gamma_{j,j+1}^h$ and $\Gamma_{j,j-1}$ to
  $\Gamma_{j,j-1}^h$.  In order of $j=J-1,..,1$, solve the problem
  \begin{displaymath}
    \small
  \begin{array}{r@{\hspace{0.2em}}c@{\hspace{0.2em}}ll}
    \mathcal{L}\,w_j&=&\tilde{f}+2\delta(s_{j})\mathbf{n}_j^T\alpha\nabla w_{j+1}
    & \mbox{ in }\tilde{\Omega}_j,\\
    \mathcal{B}\,w_j&=&g & \mbox{ on }
    \partial\Omega\cap\partial\tilde{\Omega}_j,\\
    \mathbf{n}_j^T\alpha\nabla w_j+
    \mathrm{DtN}_{j}^{pml}w_j&=&0 & \mbox{ on }\Gamma_{j,j-1}^h,\\
    \mathbf{n}_j^T\alpha\nabla w_j+
    \mathrm{DtN}_{j}^{pml}w_j&=&0 & \mbox{ on }\Gamma_{j,j+1}^h,
  \end{array}
  \end{displaymath}
  where the $\mathrm{DtN}_j^{pml}$ on $\Gamma_{j,j+1}^h$
  ($\Gamma_{j,j-1}^h$) uses the PML as the subset (superset) of the PML for $v_j$ on
  $\Gamma_{j,j+1}$ ($\Gamma_{j,j-1}$) that starts later (earlier using the
  original operator).

  Output $\tilde{u}\gets v+w_j$ in $\tilde{\Omega}_j$, $j=1,..,J-1$ and
  $\tilde{u}\gets v$ in $\tilde{\Omega}_J$.
\end{algorithm}

\begin{algorithm}
  \caption{{\bf Single layer potential} preconditioner at the {\bf matrix} level}
  \label{alg:stolkmat}
  Input the r.h.s. $\mathbf{f}$.  Suppose the $\Omega_j$'s are non-overlapping.
  Choose one of the following two forms (see \Cref{thm:stolkmat} for their
  equivalence under certian conditions).

  {\bf Form 1.} Extend the non-overlapping subdomains one layer beyond each
  interface.  We indicate the extra d.o.f. beyond the left interfaces of the
  $j$-th subdomain by the subscripts $j\langle\langle$ or $j-1[[$, and similarly
  for $j\rangle\rangle$ or $j+1]]$.

  Solve successively for $j=1,..,J$,
  \begin{displaymath}\small
  \left[
    \begin{array}{ccccc}
      \tilde{S}_{j\langle\langle}^{\langle}
      & A_{j\langle\langle j\langle} & & &\\
      A_{j\langle j\langle\langle} & A_{j\langle}
      & A_{j\langle j\bullet} & &\\
      & A_{j\bullet j\langle} & A_{j\bullet} & A_{j\bullet j\rangle} &\\
      & & A_{j\rangle j\bullet} & A_{j\rangle} & A_{j\rangle j\rangle\rangle}\\
      & & & A_{j\rangle\rangle j\rangle} & \tilde{S}_{j\rangle\rangle}^{\rangle}
    \end{array}
  \right]
  \left[
    \begin{array}{c}
      \mathbf{v}_{j\langle\langle}\vphantom{\tilde{S}_{j\langle}^{\langle}}\\
      \mathbf{v}_{j\langle}\\
      \mathbf{v}_{j\bullet}\\
      \mathbf{v}_{j\rangle}\\
      \mathbf{v}_{j\rangle\rangle}\vphantom{\tilde{S}_{j\langle}^{\langle}}
    \end{array}
  \right] =
  \left[
    \begin{array}{c}
      \mathbf{v}_{j-1[[}
      +A_{j\langle\langle j\langle}\mathbf{v}_{j-1[}\\
      \mathbf{f}_{j\langle}-A_{j\langle j\langle\langle}
      \mathbf{v}_{j-1[[}-\mathbf{v}_{j-1[}\\
      \mathbf{f}_{j\bullet}\vphantom{^\langle_j}\\
      0\vphantom{^\langle_j}\\
      0\vphantom{^\langle_j}
    \end{array}
  \right],
  \end{displaymath}
  where $\tilde{S}_{j\langle\langle}^{\langle}$ and $\tilde{S}_{j\rangle\rangle}^{\rangle}$
  are the Schur complements from the PML (see \Cref{rem:imschur} for practical
  implementation).

  Let $\mathbf{v}\gets\sum_{j=1}^{J}R_j^T(I_j^{j\langle}I^j_{j\langle}+
  I_j^{j\bullet}I^j_{j\bullet})\mathbf{v}_{j}$ and compute
  $\mathbf{\tilde{f}}\gets\mathbf{f}-A\mathbf{v}$.  Let $\mathbf{w}_J\gets0$.

  In order of $j=J-1,..,1$, solve the correction problem
  \begin{displaymath}
    \small
  \left[
    \begin{array}{ccccc}
      \tilde{S}_{j\langle\langle}^{\langle}
      & A_{j\langle\langle j\langle} & & &\\
      A_{j\langle j\langle\langle} & A_{j\langle}
      & A_{j\langle j\bullet} & &\\
      & A_{j\bullet j\langle} & A_{j\bullet} & A_{j\bullet j\rangle} &\\
      & & A_{j\rangle j\bullet} & A_{j\rangle} & A_{j\rangle j\rangle\rangle}\\
      & & & A_{j\rangle\rangle j\rangle} & \tilde{S}_{j\rangle\rangle}^{\rangle}
    \end{array}
  \right]
  \left[
    \begin{array}{c}
      \mathbf{w}_{j\langle\langle}\vphantom{\tilde{S}_{j\langle}^{\langle}}\\
      \mathbf{w}_{j\langle}\\
      \mathbf{w}_{j\bullet}\\
      \mathbf{w}_{j\rangle}\\
      \mathbf{w}_{j\rangle\rangle}\vphantom{\tilde{S}_{j\langle}^{\langle}}
    \end{array}
  \right] =
  \left[
    \begin{array}{c}
      0\vphantom{^\langle_j}\\
      0\vphantom{^\langle_j}\\
      \mathbf{\tilde{f}}_{j\bullet}\vphantom{^\langle_j}\\
      \mathbf{\tilde{f}}_{j\rangle}-A_{j\rangle j\rangle\rangle}
      \mathbf{w}_{j+1]]}-\mathbf{w}_{j+1]}\\
      \mathbf{w}_{j+1]]}
      +A_{j\rangle\rangle j\rangle}\mathbf{w}_{j+1]}
    \end{array}
  \right].
  \end{displaymath}

  Compute the output $\mathbf{\tilde{u}}\gets\mathbf{v}+\sum_{j=1}^{J-1}R_j^T
  (I_j^{j\bullet}I_{j\bullet}^j+I_j^{j\rangle}I^j_{j\rangle})\mathbf{w}_j$.

  %

  {\bf Form 2.}  Solve successively for $j=1,..J$,
  \begin{displaymath}
    \small
  \left[
    \begin{array}{ccc}
      \tilde{S}_{j\langle}^{\langle} & A_{j\langle j\bullet} &\\
      A_{j\bullet j\langle} & A_{j\bullet} & A_{j\bullet j\rangle}\\
      & A_{j\rangle j\bullet} & \tilde{S}_{j\rangle}^{\rangle}
    \end{array}
  \right]
  \left[
    \begin{array}{c}
      \mathbf{v}_{j\langle}\vphantom{\tilde{S}_{j\langle}^{\langle}}\\
      \mathbf{v}_{j\bullet}\\
      \mathbf{v}_{j\rangle}\vphantom{\tilde{S}_{j\langle}^{\langle}}
    \end{array}
  \right] =
    \left[
    \begin{array}{c}
      \mathbf{f}_{j\langle}-2A_{j\langle j-1\bullet}\mathbf{v}_{j-1\bullet}
      -A_{j\langle}\mathbf{v}_{j-1[}\\
      \mathbf{f}_{j\bullet}\vphantom{^j_{j\langle}}\\
      0
    \end{array}
  \right],
  \end{displaymath}
  where $\tilde{S}_{j\langle}^{\langle}$ and $\tilde{S}_{j\rangle}^{\rangle}$
  are the Schur complements from the PML (see \Cref{rem:imschur} for practical
  implementation).

  Let
  $\mathbf{v}\gets\sum_{j=1}^{J}R_j^T(I_j^{j\langle}I^j_{j\langle}+I_j^{j\bullet}I_{j\bullet}^j)
  \mathbf{v}_{j}$ and compute $\mathbf{\tilde{f}}\gets\mathbf{f}-A\mathbf{v}$.  Let
  $\mathbf{w}_J\gets0$.

  In order of $j=J-1,..,1$, solve the correction problem
  \begin{displaymath}
    \small
    \left[
    \begin{array}{ccc}
      \tilde{S}_{j\langle}^{\langle} & A_{\langle j\bullet} &\\
      A_{j\bullet j\langle} & A_{j\bullet} & A_{j\bullet j\rangle}\\
      & A_{j\rangle j\bullet} & \tilde{S}_{j\rangle}^{\rangle}
    \end{array}
  \right]
  \left[
    \begin{array}{c}
      \mathbf{w}_{j\langle}\vphantom{\tilde{S}_{j\langle}^{\langle}}\\
      \mathbf{w}_{j\bullet}\\
      \mathbf{w}_{j\rangle}\vphantom{\tilde{S}_{j\langle}^{\langle}}
    \end{array}
  \right] =
    \left[
    \begin{array}{c}
      0\\
      \mathbf{\tilde{f}}_{j\bullet}\vphantom{^j_{j\langle}}\\
      \mathbf{\tilde{f}}_{j\rangle}-2
        A_{j\rangle j+1\bullet}\mathbf{w}_{j+1\bullet}
        -A_{j\rangle}\mathbf{w}_{j+1]}
    \end{array}
  \right].
  \end{displaymath}

  Compute the output $\mathbf{\tilde{u}}\gets\mathbf{v}+\sum_{j=1}^{J-1}R_j^T(
  I_j^{j\bullet}I_{j\bullet}^j+I_j^{j\rangle}I^j_{j\rangle})\mathbf{w}_j$.
\end{algorithm}

\begin{theorem}
  Suppose the subproblems of \Cref{alg:stolkpde} are well-posed.  If
  the PML--DtN operators on the two sides of each interface are equal,
  i.e. $\mathrm{DtN}_j^{pml}=\mathrm{DtN}_{j+1}^{pml}$ on
  $\Gamma_{j,j+1}$ and on $\Gamma_{j,j+1}^h$, then the single layer
  potential preconditioner as shown in \Cref{alg:stolkpde} is
  equivalent to one iteration of \Cref{alg:DOSMPDE} with zero initial
  guess, $\mathcal{Q}_j^{\langle}=\mathcal{I}$,
  $\mathcal{P}_j^{\langle}=\mathrm{DtN}_j^{pml}|_{\Gamma_{j,j-1}^{h}}$,
  $\mathcal{Q}_j^{\rangle}=\mathcal{I}$,
  $\mathcal{P}_j^{\rangle}=\mathrm{DtN}_j^{pml}|_{\Gamma_{j,j+1}^{h}}$
  and using the two non-overlapping partitions as in
  \Cref{alg:stolkpde}, one partition for the forward and the
  other for the backward sweep.
\end{theorem}

\begin{proof}
  By the zero initial guess and the specific conditions for
  \Cref{alg:DOSMPDE} in the theorem, the $(j-1)$-st subproblem of
  \Cref{alg:DOSMPDE} imposes the following condition on
  $\Gamma_{j-1,j}=\Gamma_{j,j-1}$, since the partition is
  non-overlapping:
  \begin{displaymath}
  \mathbf{n}_{j-1}^T\alpha\nabla v_{j-1}+ \mathrm{DtN}_{j-1}^{pml}v_{j-1}=0.
  \end{displaymath}
  Note that $\mathbf{n}_j=-\mathbf{n}_{j-1}$ on $\Gamma_{j,j-1}$.  Substituting
  these and the assumption $\mathrm{DtN}_j^{pml}=\mathrm{DtN}_{j-1}^{pml}$ into
  the transmission condition on $\Gamma_{j,j-1}$ of the $j$-th subproblem, we
  find
  \begin{displaymath}
  \mathbf{n}_{j}^T\alpha\nabla v_{j}+ \mathrm{DtN}_{j}^{pml}v_{j}=
  2\mathbf{n}_{j}^T\alpha\nabla v_{j-1},
  \end{displaymath}
  which imposes a Neumann jump between $\Omega_j$ and the PML on the
  other side of $\Gamma_{j,j-1}$.  We recover the forward sweep of
  \Cref{alg:stolkpde} by moving the Neumann jump to the r.h.s. of the
  PDE as a surface source. Between the forward and the backward sweep,
  \Cref{alg:stolkpde} takes the residual $\tilde{f}$ and introduces
  overlaps of the old subdomains used in forward sweep and the new
  subdomains used in backward sweep.  This gives the deferred
  correction form as in \Cref{alg:GDCPDE} but mixed with the
  single layer potential on $\Gamma_{j,j+1}^h$ too.  Then, we can
  conclude by adapting the proof of \Cref{thm:GDCPDE} and the same
  arguments as in the forward sweep.
\end{proof}

\begin{theorem}\label{thm:stolkmat}
  Let \Cref{alg:GDCMAT} use
  $Q_{j\langle}^{\langle}=I_{j\langle}$, $Q_{j\rangle}^{\rangle}=I_{j\rangle}$,
  $P_{j\langle}^{\langle}=\tilde{S}_{j\langle}^{\langle}-A_{j\langle}^{\langle}$,
  $P_{j\rangle}^{\rangle}=\tilde{S}_{j\rangle}^{\rangle}-A_{j\rangle}^{\rangle}$,
  and let the initial guess be zero.  The following statements about
  \Cref{alg:stolkmat} then hold:
  \begin{itemize}
  \item[$1^{\circ}$] Suppose the subproblems of the first form of
    \Cref{alg:stolkmat} are well-posed and
    $\tilde{S}_{j\langle\langle}^{\langle}$,
    $\tilde{S}_{j\rangle\rangle}^{\rangle}$ are invertible.  Let
    $\tilde{S}_{j\langle}^{\langle}:=A_{j\langle} -A_{j\langle
      j\langle\langle}\left(\tilde{S}_{j\langle\langle}^{\langle}\right)^{-1}
    A_{j\langle\langle j\langle}$ and
    $\tilde{S}_{j\rangle}^{\rangle}:=A_{j\rangle} -A_{j\rangle
      j\rangle\rangle}\left(\tilde{S}_{j\rangle\rangle}^{\rangle}\right)^{-1}
    A_{j\rangle\rangle j\rangle}$.  If %
    $I_{j\langle\langle}=I_{j\langle}=-A_{j\langle j\langle\langle}$,
    $\tilde{S}_{j\langle\langle}^{\langle}=\tilde{S}_{j-1\rangle}^{\rangle}$ for
    $j=2,..,J$, $I_{j\rangle\rangle}=I_{j\rangle}=-A_{j\rangle
      j\rangle\rangle}$,
    $\tilde{S}_{j\rangle\rangle}^{\rangle}=\tilde{S}_{j+1\langle}^{\langle}$ for
    $j=1,..,J-1$ and $A_{j\langle j-1\bullet}=[A_{j\langle j\langle\langle}, 0]$,
    then the first form of \Cref{alg:stolkmat} is equivalent to the
    harmonic extension variant of \Cref{alg:GDCMAT}, see
    \Cref{rem:OSHM}.
  \item[$2^{\circ}$] Suppose the subproblems of the second form of
    \Cref{alg:stolkmat} are well-posed.  If
    $\tilde{S}_{j\langle}^{\langle}=\tilde{S}_{j-1\rangle}^{\rangle}$ for
    $j=2,..,J$ and
    $\tilde{S}_{j\rangle}^{\rangle}=\tilde{S}_{j+1\langle}^{\langle}$ for
    $j=1,..,J-1$, then the second form of \Cref{alg:stolkmat} is
    equivalent to the harmonic extension variant of \Cref{alg:GDCMAT},
    see \Cref{rem:OSHM}.
  \end{itemize}
\end{theorem}

\begin{proof}
  We first prove claim $1^{\circ}$.  We eliminate the first and the last rows of
  the $j$-th forward subproblem and substitute with the assumptions of
  $1^{\circ}$ to obtain
  \begin{equation}\label{eq:foldlayer}
    \small
    \arraycolsep0.2em
    \medmuskip=-1mu
    \thinmuskip=-1mu
    \thickmuskip=-1mu
    \nulldelimiterspace=-2pt
    \scriptspace=-1pt    
  \left[
    \begin{array}{cccc}
      \tilde{S}_{j\langle}^{\langle} & A_{j\langle j\bullet} & &\\
      A_{j\bullet j\langle} & A_{j\bullet} & A_{j\bullet j\rangle} &\\
      & A_{j\rangle j\bullet} & \tilde{S}_{j\rangle}^{\rangle}
    \end{array}
  \right]\hspace{-0.35em}
  \left[
    \begin{array}{c}
      \mathbf{v}_{j\langle}\vphantom{_{j\langle}^{\langle}}\\
      \mathbf{v}_{j\bullet}\\
      \mathbf{v}_{j\rangle}\vphantom{_{j\langle}^{\langle}}
    \end{array}
  \right]=
  \left[
    \begin{array}{c}
      \mathbf{f}_{j\langle}-A_{j\langle j\langle\langle}\mathbf{v}_{j-1[[}
      +(\tilde{S}_{j\langle}^{\langle}-A_{j\langle})\mathbf{v}_{j-1[}+
      \underline{(\tilde{S}_{j\langle\langle}^{\langle})^{-1}\mathbf{v}_{j-1[[}
      -\mathbf{v}_{j-1[}}\\
      \mathbf{f}_{j\bullet}\vphantom{^\langle_j}\\
      0\vphantom{^\langle_j}
    \end{array}
  \right].
  \end{equation}
  It can be shown that the underlined expression on the right above
  vanishes: in fact, the above subproblem is also used with $j-1$
  replacing $j$. In particular, the last row of the $(j-1)$-st
  subproblem reads
  \begin{equation}\label{eq:prevrow}
  A_{j-1\rangle j-1\bullet}\mathbf{v}_{j-1\bullet}
  +\tilde{S}^{\rangle}_{j-1\rangle}\mathbf{v}_{j-1\rangle}=0.
  \end{equation}
  Recalling that $j-1\rangle$ and $j\langle$ correspond to the same d.o.f.,
  because the decomposition is non-overlapping, and the
  assumptions that $A_{j\langle
    j-1\bullet}=[A_{j\langle j\langle\langle},0]$,
  $A_{j\langle j\langle\langle}=-I_{j\langle\langle}$ and
  $\tilde{S}^{\rangle}_{j-1\rangle}=\tilde{S}^{\langle}_{j\langle\langle}$, we obtain from
  \Cref{eq:prevrow} that
    $\tilde{S}^{\langle}_{j\langle\langle}\mathbf{v}_{j-1[}=\mathbf{v}_{j-1[[}$.
  This shows that the underlined expression in \Cref{eq:foldlayer} is zero.
  Similar to \Cref{eq:matrt2}, \Cref{eq:foldlayer} is almost the same
  subproblem used in \Cref{alg:DOSMPDE} but \Cref{eq:foldlayer} puts
  zero on the last row of the r.h.s. which is a trait of the harmonic extension
  variant.  By reusing some arguments from the proof of 
  \Cref{thm:GDCMAT}, we can show the equivalence to \Cref{alg:GDCMAT} and
  conclude with claim $1^{\circ}$.

  To prove claim $2^{\circ}$, the key is to show that the r.h.s. from
  the second form of \Cref{alg:stolkmat} is the same as the
  r.h.s. from \Cref{alg:DOSMMAT} except that the former puts zeros on
  the right interfaces in the forward sweep, and the left interfaces in the
  backward sweep.  Similar to \Cref{eq:matrt2}, we may write the first
  row of the r.h.s. from the forward sweep of \Cref{alg:DOSMMAT} as
  \begin{displaymath}
  \mathbf{f}_{j\langle}-A_{j\langle j-1\bullet}
  \mathbf{u}_{j-1\bullet}^{(\frac{1}{2})}
  +(\tilde{S}_{j\langle}^{\langle}-A_{j\langle})
  \mathbf{u}_{j-1[}^{(\frac{1}{2})}.
  \end{displaymath}
  Using the assumption that
  $\tilde{S}_{j\langle}^{\langle}=\tilde{S}_{j-1\rangle}^{\rangle}$ and the last
  row (with zeroed r.h.s.) of the ($j-1$)-st subproblem similar to
  \Cref{eq:prevrow}, we see that the above expression is equal to
  \begin{displaymath}
  \mathbf{f}_{j\langle}-2A_{j\langle j-1\bullet}
  \mathbf{u}_{j-1\bullet}^{(\frac{1}{2})}
  -A_{j\langle}\mathbf{u}_{j-1[}^{(\frac{1}{2})}.
  \end{displaymath}
  This is exactly the same as in the forward sweep of
  \Cref{alg:stolkmat}.  The remaining part of the proof of claim
  $2^{\circ}$ can now be done as in the proof of \Cref{thm:GDCMAT}.
\end{proof}

There is a final ingredient used in \cite{Stolk}
based on the idea of right preconditioning.
A preconditioner $M^{-1}$ such as the one defined by \Cref{alg:stolkmat}
can be used either on the
left or the right of the original operator $A$.  For right
preconditioning of \Cref{eq:las}, one first uses an iterative method
like Krylov or Richardson to solve $AM^{-1}\mathbf{r}=f$ for
$\mathbf{r}$, and then obtains the solution $\mathbf{u}$ of
\Cref{eq:las} by computing $\mathbf{u}=M^{-1}\mathbf{r}$.  Let
$\mathbf{r}^{(n)}$ for $n\geq0$ be the iterates for $\mathbf{r}$.
Denote by $\mathbf{v}^{(n+1)}:=M^{-1}\mathbf{r}^{(n)}$.  It can be
shown that if $\mathbf{r}^{(0)}=\mathbf{f}-A\mathbf{u}^{(0)}$ and
$\mathbf{r}^{(n)}$ and $\mathbf{u}^{(n)}$ are generated by the
Richardson iterations
\begin{displaymath}
\mathbf{r}^{(n+1)}=\mathbf{r}^{(n)}+\mathbf{f}-AM^{-1}\mathbf{r}^{(n)},\quad
\mathbf{u}^{(n+1)}=\mathbf{u}^{(n)}+M^{-1}(\mathbf{f}-A\mathbf{u}^{(n)}),
\end{displaymath}
then we have the relation $\mathbf{u}^{(n)}=\mathbf{v}^{(n)}+(I-M^{-1}A)^{n-1}
\mathbf{u}^{(0)}$ for $n\geq1$.  There is also a relation between the GMRES
iterates for the left and the right preconditioned systems; see \cite{SS07}.

If one solves the restricted version of the original problem exactly
in the interior of the subdomains using a direct solver, and then
glues the resulting local approximations into a global approximation,
then the global approximation has mostly a zero residual, except where
the residual is influenced by values of the global approximation from
different subdomains, see e.g. \cite{gander2014new}. This is why in
\Cref{alg:stolkpde}, the intermediate residual $\tilde{f}$ is
concentrated in the neighborhood of the interfaces $\Gamma_{j,j-1}$,
$j\geq2$, and the output $\tilde{u}$ leaves the residual
$f-\mathcal{L}\tilde{u}$ concentrated in the neighborhood of the
shifted interfaces $\Gamma_{j,j-1}^h$, $j\geq2$.  In
\Cref{alg:stolkmat}, a component of $\mathbf{\tilde{f}}$ is non-zero
only if the corresponding row in the matrix $A$ has at least one
non-zero entry belonging to a column associated with an interface
d.o.f., and the residual left by the output $\mathbf{\tilde{u}}$ has
also a similar sparsity.  The sparsity of the residuals can be
leveraged in the right preconditioned system because the essential
unknowns become the non-zero components of the residual.  This was
studied in detail in \cite{Ito06} and was also suggested by Stolk in
\cite{Stolk} for the preconditioner based on single layer potentials.
We summarize the sparse residual algorithm in \Cref{alg:spyres} and
justify it in \Cref{thm:spyres}.  Note that this substructured form
can be adapted to all the preconditioners resulting in sparse residuals,
e.g. \Cref{alg:GDCMAT}.  Compared to \Cref{alg:ISMAT}, the reduced
system in \Cref{alg:spyres} is typically of twice size but free of
applying the PML--DtN operators.

\begin{algorithm}
  \caption{Residual substructuring when most rows of $I-AM^{-1}$ vanish}
  \label{alg:spyres}
  Construct the 0-1 matrix $R_{r}$ such that $(I-R_{r}^TR_{r})(I-AM^{-1})=0$ and
  $R_rR_r^T=I_r$.

  Set an initial guess $\mathbf{u}^{(0)}$ such that
  $(I-R_r^TR_r)(\mathbf{f}-A\mathbf{u}^{(0)})=0$, e.g.
  $\mathbf{u}^{(0)}\gets M^{-1}\mathbf{f}$.

  Let $\mathbf{h}_{r}\gets R_{r}(\mathbf{f}-A\mathbf{u}^{(0)})$.  Solve
  (approximately) the substructured system for $\mathbf{r}_{r}$:
  \begin{equation}\label{eq:resub}
  (R_{r}AM^{-1}R_{r}^T)\mathbf{r}_{r}=\mathbf{h}_{r}.
  \end{equation}

  Let $\mathbf{u}\gets M^{-1}R_r^T\mathbf{r}_r +\mathbf{u}^{(0)}$ which is
  (approximately) the solution of \Cref{eq:las}.
\end{algorithm}

\begin{theorem}\label{thm:spyres}
  If $A$ and $M^{-1}$ are invertible, then the substructured system
  in \Cref{eq:resub} is well-posed, and if in addition
  \Cref{eq:resub} is solved exactly, then the output $\mathbf{u}$
  from \Cref{alg:spyres} is indeed the solution of \Cref{eq:las}.
\end{theorem}
\begin{proof}
  We first assume that \Cref{eq:resub} has at least one solution. Hence,
  \begin{displaymath}
  R_r^TR_{r}AM^{-1}R_{r}^T\mathbf{r}_{r}=R_r^T\mathbf{h}_r=R_r^TR_r
  (\mathbf{f}-A\mathbf{u}^{(0)}) = \mathbf{f}-A\mathbf{u}^{(0)},  
  \end{displaymath}
  where the last equality follows from the assumption on $\mathbf{u}^{(0)}$.  By
  the assumptions on $R_r$, we also have
  \begin{displaymath}
  (I-R_r^TR_r)AM^{-1}R_{r}^T\mathbf{r}_{r}=
  (I-R_r^TR_r)R_{r}^T\mathbf{r}_{r}=0.
  \end{displaymath}
  Summing the two identities above, we obtain
  \begin{displaymath}
  AM^{-1}R_{r}^T\mathbf{r}_{r}=\mathbf{f}-A\mathbf{u}^{(0)},
  \end{displaymath}
  or $A\mathbf{u}=A(M^{-1}R_{r}^T\mathbf{r}_{r}+\mathbf{u}^{(0)})=\mathbf{f}$;
  that is, the output of \Cref{alg:spyres} is indeed the solution of
  \Cref{eq:las}.  Now if $\mathbf{h}_r=0$ then by the assumption on
  $\mathbf{u}^{(0)}$ we have
  $ AM^{-1}R_{r}^T\mathbf{r}_{r}=\mathbf{f}-A\mathbf{u}^{(0)}=0, $ which, since
  $A$ and $M^{-1}$ are invertible, implies $R_r^T\mathbf{r}_r=0$ and further
  $\mathbf{r}_r=0$ using $R_rR_r^T=I_r$.  Hence, \Cref{eq:resub} is well-posed.
\end{proof}

\subsection{Method of polarized traces using single and double layer
  potentials}

We have already seen that \Cref{eq:potentials} can be used to
propagate the wave field in $\Omega_{j-1}$ to $\Omega_j$.  The only
data we take from the previous subdomain are the Neumann and Dirichlet
traces on the interface $\Gamma_{j,j-1}$.  The method of polarized
traces introduced in \cite{Zepeda} iterates the Neumann and
Dirichlet traces from neighboring subdomains in the substeps of the
forward and backward sweeps, and upon completion of a double sweep, a
global approximation is constructed using the representation formula
\Cref{eq:rep} in subdomains where the volume potentials have been
precomputed before the sweeps\footnote{In the full paper \cite{ZD} that followed \cite{Zepeda}, a substructured system for the traces is first solved instead of the original system. For brevity, we will describe only the global form preconditioner, from which the corresponding substructured system is easy to derive.}.  We summarize the polarized traces
preconditioner in \Cref{alg:polarpde} at the PDE level. The relation
to \Cref{alg:DOSMPDE} is shown in \Cref{thm:polarpde}.
\begin{algorithm}
  \caption{{\bf Polarized traces} preconditioner at the PDE level
    (\cite[Algorithm~1]{Zepeda})}
  \label{alg:polarpde}
  Input the source term $f$ and assume for simplicity $g=0$ in \Cref{eq:pde}.

  Suppose $\Omega$ is decomposed into non-overlapping subdomains.
  
  Independently for $j=1,..,J$ solve
  \begin{displaymath}
    \small
  \begin{array}{r@{\hspace{0.2em}}c@{\hspace{0.2em}}ll}
    \mathcal{L}\,{v}_{j}^{0}&=&f_{j} & \mbox{ in }\Omega_j,\\
    \mathcal{B}\,v_{j}^0&=&g & \mbox{ on }
    \partial\Omega\cap\partial\Omega_j,\\
    \mathbf{n}_j^T\alpha\nabla v_{j}^0+
    \mathrm{DtN}_{j}^{pml}v_{j}^0&=&0 & \mbox{ on }\Gamma_{j,j-1},\\
    \mathbf{n}_j^T\alpha\nabla v_{j}^0+
    \mathrm{DtN}_{j}^{pml}v_{j}^0&=&0 & \mbox{ on }\Gamma_{j,j+1},
  \end{array}
  \end{displaymath}
  where $f_j:=f|_{\Omega_j}$, see \Cref{rem:impml} for a practical
  implementation of the PML--DtN operators.  Denote by
  $G_{j}(\mathbf{x},\mathbf{y})$ the Green's function for the
  subproblem above.  We have
  $v_{j}^0(\mathbf{x})=\int_{\Omega_j}G_j(\mathbf{x},\mathbf{y})
  f(\mathbf{y})\,\mathrm{d}\mathbf{y}$ for $\mathbf{x}\in\Omega_j$.

  Let $\lambda_{1\langle}^D\gets0$, $\lambda_{1\langle}^N\gets0$.  Successively
  for $j=2,..,J$ compute for all $\mathbf{x}\in\Gamma_{j,j-1}$,
  \begin{displaymath}
    \small
  \arraycolsep=.2em
  \begin{array}{rcl}
    \lambda^D_{j\langle}(\mathbf{x})&\gets&(\mathcal{S}_{j\langle}\lambda^N_{j-1\langle})
    (\mathbf{x})-(\mathcal{D}_{j\langle}\lambda^D_{j-1\langle})(\mathbf{x})
    + v_{j-1}^0(\mathbf{x}),\\
    \lambda^N_{j\langle}(\mathbf{x})&\gets&
    (\mathcal{D}^*_{j\langle}\lambda^N_{j-1\langle})(\mathbf{x})-
    (\mathcal{N}_{j\langle}\lambda^D_{j-1\langle})(\mathbf{x})+
    (\mathbf{n}_j^T\alpha\nabla v_{j-1}^0)(\mathbf{x}),
  \end{array}
  \end{displaymath}
  where the four surface potentials vanish for $j=2$,  and are given for
  $j\geq 3$ by
  \begin{equation}\label{eq:leftpt}
    \small
    \arraycolsep=.1em
    \begin{array}{rcl}
      (S_{j\langle}w)(\mathbf{x})&:=&\int_{\Gamma_{j-1,j-2}}w(\mathbf{y})
      {G}_{j-1}(\mathbf{x},\mathbf{y})\mathrm{d}\sigma(\mathbf{y}),\\
      (\mathcal{D}_{j\langle}w)(\mathbf{x})&:=&\int_{\Gamma_{j-1,j-2}}w(\mathbf{y})
      (\mathbf{n}_{j-1}^T\alpha^T)(\mathbf{y})\nabla_{\mathbf{y}}
      G_{j-1}(\mathbf{x},\mathbf{y})\,\mathrm{d}\sigma(\mathbf{y}),\\
      (\mathcal{D}^*_{j\langle}w)(\mathbf{x})&:=&\int_{\Gamma_{j-1,j-2}}w(\mathbf{y})
      (\mathbf{n}_{j}^T\alpha)(\mathbf{x})\nabla_{\mathbf{x}}
      G_{j-1}(\mathbf{x},\mathbf{y})\,\mathrm{d}\sigma(\mathbf{y}),\\
      (\mathcal{N}_{j\langle}w)(\mathbf{x})&:=&\int_{\Gamma_{j-1,j-2}}w(\mathbf{y})
      (\mathbf{n}_{j}^T\alpha)(\mathbf{x})\nabla_{\mathbf{x}}
      \left\{(\mathbf{n}_{j-1}^T\alpha^T)(\mathbf{y})\nabla_{\mathbf{y}}
        G_{j-1}(\mathbf{x},\mathbf{y})\right\}\,\mathrm{d}\sigma(\mathbf{y}).
    \end{array}
  \end{equation}

  Let $\lambda^D_{J\rangle}\gets0$, $\lambda^N_{J\rangle}\gets0$.  Successively for
  $j=J-1,..,1$ compute for all $\mathbf{x}\in\Gamma_{j,j+1}$,
  \begin{displaymath}
    \small
  \arraycolsep=.2em
  \begin{array}{rcl}
    \lambda^D_{j\rangle}(\mathbf{x})&\gets&(\mathcal{S}_{j\rangle}\lambda^N_{j+1\rangle})
    (\mathbf{x})-(\mathcal{D}_{j\rangle}\lambda^D_{j+1\rangle})(\mathbf{x})
    + v_{j+1}^0(\mathbf{x}),\\
    \lambda^N_{j\rangle}(\mathbf{x})&\gets&
    (\mathcal{D}^*_{j\rangle}\lambda^N_{j+1\rangle})(\mathbf{x})-
    (\mathcal{N}_{j\rangle}\lambda^D_{j+1\rangle})(\mathbf{x})+
    (\mathbf{n}_j^T\alpha\nabla v_{j+1}^0)(\mathbf{x}),
  \end{array}
  \end{displaymath}
  where the four surface potentials vanish for $j=J-1$, and are given
  for $j\leq J-2$ by
  \begin{equation}\label{eq:rightpt}
    \small
    \arraycolsep=.2em
    \begin{array}{rcl}
      (S_{j\rangle}w)(\mathbf{x})&:=&\int_{\Gamma_{j+1,j+2}}w(\mathbf{y})
      {G}_{j+1}(\mathbf{x},\mathbf{y})\mathrm{d}\sigma(\mathbf{y}),\\
      (\mathcal{D}_{j\rangle}w)(\mathbf{x})&:=&\int_{\Gamma_{j+1,j+2}}w(\mathbf{y})
      (\mathbf{n}_{j+1}^T\alpha^T)(\mathbf{y})\nabla_{\mathbf{y}}
      G_{j+1}(\mathbf{x},\mathbf{y})\,\mathrm{d}\sigma(\mathbf{y}),\\
      (\mathcal{D}^*_{j\rangle}w)(\mathbf{x})&:=&\int_{\Gamma_{j+1,j+2}}w(\mathbf{y})
      (\mathbf{n}_{j}^T\alpha)(\mathbf{x})\nabla_{\mathbf{x}}
      G_{j+1}(\mathbf{x},\mathbf{y})\,\mathrm{d}\sigma(\mathbf{y}),\\
      (\mathcal{N}_{j\rangle}w)(\mathbf{x})&:=&\int_{\Gamma_{j+1,j+2}}w(\mathbf{y})
      (\mathbf{n}_{j}^T\alpha)(\mathbf{x})\nabla_{\mathbf{x}}
      \left\{(\mathbf{n}_{j+1}^T\alpha^T)(\mathbf{y})\nabla_{\mathbf{y}}
        G_{j+1}(\mathbf{x},\mathbf{y})\right\}\,\mathrm{d}\sigma(\mathbf{y}).
    \end{array}
  \end{equation}

  Recover independently the subdomain solutions for $\mathbf{x}\in\Omega_j$
  $j=1,..,J$, by
  \begin{equation}\label{eq:allpt}
    \small
    \arraycolsep=.2em
    \begin{array}{rl}
      v_j(\mathbf{x})\gets
      &v_{j}^0(\mathbf{x})\hspace{-0.2em}+\hspace{-0.2em}
        \int_{\Gamma_{j,j-1}}\hspace{-0.8em}
        \lambda^N_{j\langle}(\mathbf{y}){G}_j(\mathbf{x},\mathbf{y})
        \mathrm{d}\sigma(\mathbf{y})
        \hspace{-0.2em}-\hspace{-0.2em}\int_{\Gamma_{j,j-1}}\hspace{-0.8em}
        \lambda^D_{j\langle}(\mathbf{y})(\mathbf{n}_j^T\alpha^T)
    (\mathbf{y})\nabla_{\mathbf{y}}{G}_j(\mathbf{x},\mathbf{y})
    \mathrm{d}\sigma(\mathbf{y})\\
    +&\int_{\Gamma_{j,j+1}}\lambda^N_{j\rangle}(\mathbf{y})
    {G}_j(\mathbf{x},\mathbf{y})\,\mathrm{d}\sigma(\mathbf{y})
    -\int_{\Gamma_{j,j+1}}\lambda^D_{j\rangle}(\mathbf{y})(\mathbf{n}_j^T\alpha^T)
    (\mathbf{y})\nabla_{\mathbf{y}}{G}_j(\mathbf{x},\mathbf{y})
    \,\mathrm{d}\sigma(\mathbf{y}).
    \end{array}
  \end{equation}

  Output the global approximation $\tilde{u}\gets v_j$ in $\Omega_j$,
  $j=1,..,J$.
\end{algorithm}

\begin{theorem}\label{thm:polarpde}
  Suppose the subproblems used for the $v_{j}^0$ in \Cref{alg:polarpde}
  are well-posed.  Let $\{u_j^{(\frac{1}{2})}\}_{j=1}^{J-1}$ and
  $\{u_j^{(1)}\}_{j=1}^J$ be generated by \Cref{alg:DOSMPDE} with zero
  initial guess, the $\mathcal{Q}$ be the identity and the $\mathcal{P}$ equal
  to the PML--DtN operators.  Let $u_J^{(\frac{1}{2})}:=0$ be defined
  on $\Omega_J$.  We have for \Cref{alg:polarpde}
  $\lambda_{j\langle}^D=u_{j-1}^{(\frac{1}{2})}$, $\lambda_{j\langle}^N=
  \mathbf{n}_j^T\alpha\nabla u_{j-1}^{(\frac{1}{2})}$ on $\Gamma_{j,j-1}$ and
  $\lambda_{j\rangle}^D=u_{j+1}^{(1)}-u_{j+1}^{(\frac{1}{2})}+v_{j+1}^0$,
  $\lambda_{j\rangle}^N= \mathbf{n}_j^T\alpha\nabla
  (u_{j+1}^{(1)}-u_{j+1}^{(\frac{1}{2})}+v_{j+1}^0)$ on $\Gamma_{j,j+1}$.
  Therefore, $v_j=u_j^{(1)}$ in $\Omega_j$.
\end{theorem}

\begin{proof}
  For simplicity, we consider only the case $g=0$ in \Cref{eq:pde}.  According
  to \Cref{alg:DOSMPDE} and \Cref{alg:polarpde},
  $u_1^{(\frac{1}{2})}=v_{1}^0$ in $\Omega_1$ and so
  \begin{equation}\label{eq:DNinit}
    \lambda_{2\langle}^D=u_1^{(\frac{1}{2})},\quad
    \lambda_{2\langle}^N=\mathbf{n}_2^T\alpha\nabla u_1^{(\frac{1}{2})}\quad
    \mbox{on }\Gamma_{2,1}.
  \end{equation}
  From the algorithms, we also have for any fixed
  $\mathbf{x}\in {\Omega}_{j}$, $j=2,..,J$,
  \begin{displaymath}
    \begin{array}{l}
  (\mathbf{n}_j^T\alpha\nabla +
  \mathrm{DtN}^{pml}_j)(u_j^{(\frac{1}{2})}-u_{j-1}^{(\frac{1}{2})})=(\mathbf{n}_j^T\alpha^T
  \nabla+\mathrm{DtN}_{j}^{pml*})G_j(\mathbf{x},\cdot)=0\mbox{ on }\Gamma_{j,j-1},\\
    (\mathbf{n}_j^T\alpha\nabla+\mathrm{DtN}^{pml}_j)u_j^{(\frac{1}{2})}=
      (\mathbf{n}_j^T\alpha^T\nabla+\mathrm{DtN}_{j}^{pml*})G_j(\mathbf{x},\cdot)=0
      \mbox{ on }\Gamma_{j,j+1}
    \end{array}
  \end{displaymath}
  where $\mathrm{DtN}_j^{pml*}$ is similar to $\mathrm{DtN}_j^{pml}$ but using
  $\alpha^T$ instead of $\alpha$.  By the representation formula in
  \Cref{eq:rep} on
  $ \overline{\Omega}_j\cup\overline{\Omega}_{j\rangle}^{pml}$, where
    $\Omega_{j\rangle}^{pml}$ is the PML region along $\Gamma_{j,j+1}$, we have
  \begin{align}
    \small
    \medmuskip=-3mu
    \thinmuskip=-3mu
    \thickmuskip=-3mu
    \nulldelimiterspace=-2pt
    \scriptspace=-2pt
    u_j^{(\frac{1}{2})}(\mathbf{x})\hspace{-0.2em}= 
    &\hspace{-0.5em}\int_{\Omega_j}\hspace{-1em}f(\mathbf{y})
      G_j(\mathbf{x},\mathbf{y})\, \mathrm{d}\mathbf{y}\hspace{-0.2em}+
      \hspace{-0.5em}\int_{\Gamma_{j,j-1}}\hspace{-2em}
      \mathbf{n}_j^T\alpha\nabla u_j^{(\frac{1}{2})}(\mathbf{y})G_j
      (\mathbf{x},\mathbf{y})\hspace{-0.2em}-\hspace{-0.2em}
      \mathbf{n}_j^T\alpha^T\nabla_{\mathbf{y}}
      G_j(\mathbf{x},\mathbf{y}) u_j^{(\frac{1}{2})}(\mathbf{y})\,
      \mathrm{d}\sigma(\mathbf{y})\nonumber\\
    =& v_j^{0}(\mathbf{x})+\int_{\Gamma_{j,j-1}}(\mathbf{n}_j^T\alpha\nabla +
    \mathrm{DtN}_j^{pml})u_j^{(\frac{1}{2})}(\mathbf{y})
    G_j(\mathbf{x},\mathbf{y})\,\mathrm{d}\sigma(\mathbf{y})\label{eq:repuj}\\
    =& v_j^{0}(\mathbf{x})+\int_{\Gamma_{j,j-1}}(\mathbf{n}_j^T\alpha\nabla +
    \mathrm{DtN}_j^{pml})u_{j-1}^{(\frac{1}{2})}(\mathbf{y})
    G_j(\mathbf{x},\mathbf{y})\,\mathrm{d}\sigma(\mathbf{y})\label{eq:repuj3}\\
    =& v_j^{0}(\mathbf{x})+\int_{\Gamma_{j,j-1}}G_j(\mathbf{x},\mathbf{y})
    \mathbf{n}_j^T\alpha\nabla u_{j-1}^{(\frac{1}{2})}(\mathbf{y})-
    u_{j-1}^{(\frac{1}{2})}(\mathbf{y})\mathbf{n}_j^T\alpha^T
    \nabla_{\mathbf{y}}G_j(\mathbf{x},\mathbf{y})\,\mathrm{d}\sigma(\mathbf{y}),
       \label{eq:repujend}
  \end{align}
  where no integrals show up on the other boundaries of
  $\overline{\Omega}_j \cup\overline{\Omega}_{j\rangle}^{pml}$ because
  the boundary conditions there are homogeneous.  \Cref{eq:repuj}
  is obtained by substituting the PML condition for
  $G_j(\mathbf{x},\cdot)$ and using the following identity for any
  fixed $\mathbf{x}\in\Omega_j$ and any trace $v(\mathbf{y})$ (which
  can be proved by the definition of $\mathrm{DtN}_j^{pml}$)
  \begin{equation}\label{eq:Pcom}
  \int_{\Gamma_{j,j-1}}v(\mathbf{y})\mathrm{DtN}_j^{pml*}
  G_j(\mathbf{x},\mathbf{y})\,\mathrm{d}\sigma(\mathbf{y})=
  \int_{\Gamma_{j,j-1}}G_j(\mathbf{x},\mathbf{y})
  \mathrm{DtN}_j^{pml}v(\mathbf{y})\,\mathrm{d}\sigma(\mathbf{y}).
  \end{equation}
  \Cref{eq:repuj3} follows from the transmission
  conditions, and \Cref{eq:repujend} is obtained by applying \Cref{eq:Pcom} and
  substituting the PML condition for $G_j(\mathbf{x},\cdot)$ again.
  Assuming that
  \begin{equation}\label{eq:DNj}
    \lambda_{j\langle}^D=u_{j-1}^{(\frac{1}{2})},\quad
    \lambda_{j\langle}^N=\mathbf{n}_j^T\alpha\nabla u_{j-1}^{(\frac{1}{2})}\quad \mbox{on }
    \Gamma_{j,j-1},
  \end{equation}
  we substitute them into \Cref{eq:repujend} and taking Dirichlet and
  Neumann traces of $u_j^{(\frac{1}{2})}$ on
  $\mathbf{x}\in\Gamma_{j+1,j}$, we find \Cref{eq:DNj} holds for
  $j+1$ replacing $j$.  By induction based on \Cref{eq:DNinit}, we
  conclude that \Cref{eq:DNj} holds for all $j=2,..,J$.  Inserting
  \Cref{eq:DNj} into \Cref{eq:repuj} yields (for $j=J$ change the
  l.h.s. to $u_J^{(1)}$)
  \begin{equation}\label{eq:rephalf}
    u_j^{(\frac{1}{2})}(\mathbf{x})=v_j^{0}(\mathbf{x})+\int_{\Gamma_{j,j-1}}
    G_j(\mathbf{x},\mathbf{y})\lambda_{j\langle}^N(\mathbf{y})-
    \lambda_{j\langle}^D(\mathbf{y})\mathbf{n}_j^T\alpha^T
    \nabla_{\mathbf{y}}G_j(\mathbf{x},\mathbf{y})\,\mathrm{d}\sigma(\mathbf{y}).
  \end{equation}
  In particular, we have $u_J^{(1)}=v_J$ with $v_J$ from
  \Cref{alg:polarpde}.

  In the backward sweep of \Cref{alg:DOSMPDE}, we denote by
  $w_{j}:=u_j^{(1)}-u_j^{(\frac{1}{2})}+v_j^{0}$, $j=J-1,..,1$ and
  $w_J:=u_J^{(1)}$.  We find $w_j$ satisfies the PML conditions,
  homogeneous on $\Gamma_{j,j-1}$ but inhomogeneous on
  $\Gamma_{j,j+1}$.  By arguments similar to the last paragraph,
  we can show for all $j=J-1,..,1$ that
  \begin{displaymath}
    \lambda_{j\rangle}^D=w_{j+1},\quad
    \lambda_{j\rangle}^N=\mathbf{n}_j^T\alpha\nabla w_{j+1}\quad \mbox{on }
    \Gamma_{j,j+1},
  \end{displaymath}
  and further
  \begin{equation}\label{eq:repw}
    w_j(\mathbf{x})=v_j^{0}(\mathbf{x})+\int_{\Gamma_{j,j+1}}G_j(\mathbf{x},\mathbf{y})
    \lambda_{j\rangle}^N(\mathbf{y})-\lambda_{j\rangle}^D(\mathbf{y})\mathbf{n}_j^T\alpha^T
    \nabla_{\mathbf{y}}G_j(\mathbf{x},\mathbf{y})\,\mathrm{d}\sigma(\mathbf{y}).
  \end{equation}
  Combining \Cref{eq:rephalf} and \Cref{eq:repw}, we conclude
  that $u_j^{(1)}=v_j$ with $v_j$ from \Cref{alg:polarpde}.
\end{proof}

\begin{remark}\label{rem:jacobi}
\Cref{alg:polarpde} includes a new technique not present in
\Cref{alg:DOSMPDE}. First, note that the local solutions of the original problem
can be represented as sums of the left going and right going waves.  Furthermore,
the two parts can be simulated independently of each other.
That is, the backward sweep of \Cref{alg:polarpde} can be performed in parallel
to the forward sweep, whereas the backward sweep of \Cref{alg:DOSMPDE} aims to
simulate the total waves and thus needs to wait for the forward sweep to finish.
\Cref{alg:DOSMPDE} can be modified in the same spirit: just use the data
from the original problem for an initial solve on subdomains,
zero the left interface data in the backward sweep,
add the approximations from the forward/backward sweep
and subtract that from the initial solve to get the total waves.
Or for the block 2-by-2 interface system in \Cref{rem:iab} use
block Jacobi instead of Gauss-Seidel.
Similar techniques were proposed in \cite{Poulson,StolkImproved}.  For
\Cref{alg:polarpde}, the waves from the forward/backward sweep have different
{\it polarized} directions and they are propagated through their Dirichlet and
Neumann {\it traces}, which gives the name of the method.
\end{remark}

\begin{remark}
  In all the preceding sections, we did not discuss in detail solvers
  for the subproblems.  Typically, LU factorizations are precomputed
  before the iteration starts, and they are then reused for the
  different r.h.s. in the iterative procedure.  Even for the
  substructured forms, the typical way is {\it not} to precompute the
  interface operators explicitly, see \Cref{rem:iab}, but only to
  implement them as matrix actions through the subdomain LU solves.
  The reason is two-fold: first, to build an interface operator in a
  naive way, we need to solve as many times the subdomain problem as
  the number of d.o.f. on the interface; second, the resulting matrix
  is dense and a naive multiplication with a vector is not cheap.
  However, developments of low rank formats of matrices such as
  $\mathcal{H}$-matrices have greatly improved the situation.  For
  example, for the Laplace equation, the method in \cite{Gillman}
  reduces the building cost to $\mathcal{O}(N^{2-(2/d)})$ and the
  application cost to $\mathcal{O}(N^{1/2})$ in 2-D and
  $\mathcal{O}(N)$ in 3-D, and some tests for the Helmholtz
    equation were also performed. In \cite{Zepeda,ZD}, low rank
  techniques are used for building and applying the surface
  potentials in \Cref{eq:leftpt,eq:rightpt,eq:allpt} for
  \Cref{alg:polarpde}.
\end{remark}

To bring \Cref{alg:polarpde} to the matrix level, we first translate
the representation formula from \Cref{eq:rep} into the matrix language.
Suppose $G$ is the matrix analogue of the Green's function, i.e.
\begin{equation}\label{eq:Gmat}
  \left[
    \begin{array}{ccc}
      G_{e}  & G_{eb} & G_{ei}\\
      G_{be} & G_{b}  & G_{bi}\\
      G_{ie} & G_{ib} & G_{i}
    \end{array}
  \right]
  \left[
    \begin{array}{ccc}
      A_{e} & A_{eb} &\\
      A_{be} & A_{b} & A_{bi}\\
      & A_{ib} & A_{i}
    \end{array}
  \right]=
  \left[
    \begin{array}{ccc}
      I_{e} &  &\\
      & I_{b} & \\
      & & I_{i}
    \end{array}
  \right],
\end{equation}
where the rows and columns with the subscripts containing $e$ may all be empty.
Let $\mathbf{u}$ satisfy
\begin{equation}\label{eq:allu}
  \left[
    \begin{array}{cc}
      \widetilde{A}_{b} & A_{bi}\\
      A_{ib} & A_{i}
    \end{array}
  \right]
  \left[
    \begin{array}{c}
      \mathbf{u}_b\\ \mathbf{u}_i
    \end{array}
  \right]=
  \left[
    \begin{array}{c}
      \mathbf{f}_b+\boldsymbol{\lambda}_b\\ \mathbf{f}_i
    \end{array}
  \right].
\end{equation}
Then, we have the following representation formula for $\mathbf{u}_i$.

\begin{proposition}
  If \Cref{eq:Gmat} holds, then \Cref{eq:allu} implies
  \begin{equation}\label{eq:repmat}
    \mathbf{u}_i=G_i\mathbf{f}_i+G_{ib}\mathbf{f}_b+G_{ib}\boldsymbol{\lambda}_b^N-
    \{G_{ib}A^{(i)}_b+G_iA_{ib}\}\mathbf{u}_b,
  \end{equation}
  where
  $\boldsymbol{\lambda}_b^N:=A_b^{(i)}\mathbf{u}_b+A_{bi}\mathbf{u}_i-\mathbf{f}_b$
  and $A_b^{(i)}$ is an arbitrary matrix of appropriate size.
\end{proposition}

\begin{proof}
  Inserting the definition of $\boldsymbol{\lambda}_b^N$ and the last row of
  \Cref{eq:allu} into the r.h.s. of \Cref{eq:repmat} we obtain
  \begin{displaymath}
  \mbox{r.h.s. of \Cref{eq:repmat}}
  =(G_iA_i+G_{ib}A_{bi})\mathbf{u}_i.
  \end{displaymath}
  On the other hand, from \Cref{eq:Gmat} we have $G_iA_i+G_{ib}A_{bi}=I_i$
  which turns the above equation into \Cref{eq:repmat}.
\end{proof}

From \Cref{eq:repmat}, we can recognize $[G_i,G_{ib}]$, $G_{ib}$ and
$G_{ib}A_b^{(i)}+G_iA_{ib}$ as the volume, single layer and double layer
potentials. We now give the matrix analogue\footnote{In the full paper \cite{ZD}
  that appeared after \cite{Zepeda}, the matrix form of \Cref{alg:polarpde} is
  derived by a first-order finite difference discretization of the Neumann
  derivatives.  Then, $\boldsymbol{\lambda}_{j*}^N$ is not introduced but
  replaced with the d.o.f. immediately next to $\boldsymbol{\lambda}_{j*}^D$ in
  $\Omega_j$.  Since this difference is only minor, we will not study this
  variant further here. A referee pointed out to us that our presentation here
  is more close to \cite[Appendix C]{ZepedaNested}.} of \Cref{alg:polarpde} in
\Cref{alg:polarmat}, and prove its equivalence to \Cref{alg:DOSMMAT} in
\Cref{thm:polarmat}.
\begin{algorithm}
  \caption{{\bf Polarized traces} preconditioner at the matrix level}
  \label{alg:polarmat}
  Input the r.h.s. $\mathbf{f}$. Suppose the decomposition is non-overlapping.

  Independently for $j=1,..,J$ solve
  \begin{displaymath}
    \small
  \left[
    \begin{array}{ccc}
      \tilde{S}_{j\langle}^{\langle} & A_{j\langle j\bullet} & \\
      A_{j\bullet j\langle} & A_{j\bullet} & A_{j\bullet j\rangle}\\
      & A_{j\rangle j\bullet} & \tilde{S}_{j\rangle}^{\rangle}
    \end{array}
  \right]
  \left[
    \begin{array}{c}
      \mathbf{v}_{j\langle}^0\\\mathbf{v}_{j\bullet}^0\\ \mathbf{v}_{j\rangle}^0
    \end{array}
  \right]=
  \left[
    \begin{array}{c}
      \mathbf{f}_{j\langle}\vphantom{^\rangle_\langle}\\
      \mathbf{f}_{j\bullet} \\
      \mathbf{f}_{j\rangle}\vphantom{^\rangle_{\langle}}
    \end{array}
  \right],
  \end{displaymath}
  where $\tilde{S}_{j\langle}^{\langle}$ and $\tilde{S}_{j\rangle}^{\rangle}$
  are defined by the Schur complements of the PML exterior to $\Omega_j$
  (see \Cref{rem:imschur} for practical implementation).
  Denote by $G^{(j)}$ the inverse of the
  above coefficient matrix and let it be partitioned in the same way as the
  above coefficient matrix, e.g., $G_{j\bullet j\langle}^{(j)}:=
  I^j_{j\bullet}G^{(j)}I_j^{j\langle}$.  We can represent
  $\mathbf{v}_{j}^0=G^{(j)}[\mathbf{f}_{j\langle};\mathbf{f}_{j\bullet};
  \mathbf{f}_{j\rangle}]$.

  Let $\boldsymbol{\lambda}_{1\langle}^D\gets0$ and
  $\boldsymbol{\lambda}_{1\langle}^N\gets0$.
  Compute successively for $j=2,..,J$,
  \begin{displaymath}
    \small
  \arraycolsep=.2em
  \begin{array}{rcl}
    \boldsymbol{\lambda}_{j\langle}^D&\gets&S_{j\langle}\boldsymbol{\lambda}_{j-1\langle}^N
    -D_{j\langle}\boldsymbol{\lambda}_{j-1\langle}^D+\mathbf{v}_{j-1[}^0,\\
    \boldsymbol{\lambda}_{j\langle}^N&\gets&D^*_{j\langle}\boldsymbol{\lambda}_{j-1\langle}^N
    -N_{j\langle}\boldsymbol{\lambda}_{j-1\langle}^D-(A_{j\langle}^{\rangle}\mathbf{v}_{j-1[}^0+A_{j\langle j-1\bullet}\mathbf{v}_{j-1\bullet}^0),
  \end{array}
  \end{displaymath}
  where the matrix potentials vanish for $j=2$, and are given for
  $j\geq3$ by
  \begin{displaymath}
    \small
    \medmuskip=-2mu
    \thinmuskip=-2mu
    \thickmuskip=-2mu
    \nulldelimiterspace=0pt
    \scriptspace=-0.5pt
    \arraycolsep=.1em    
    \begin{array}{rcl}
      S_{j\langle}&:=&G^{(j-1)}_{j\langle j-1\langle},\mbox{ }
    D_{j\langle}:=G^{(j-1)}_{j\langle j-1\langle}A_{j-1\langle}^{\langle}+
    G^{(j-1)}_{j\langle j-1\bullet}A_{j-1\bullet j-1\langle},\mbox{ }
    D^*_{j\langle}:=-A_{j\langle}^{\rangle}G^{(j-1)}_{j\langle j-1\langle}-
    A_{j\langle j-1\bullet}G^{(j-1)}_{j-1\bullet j-1\langle},\\
    N_{j\langle}&:=&-A_{j\langle}^\rangle(G^{(j-1)}_{j\langle j-1\langle}
    A_{j-1\langle}^{\langle}+G^{(j-1)}_{j\langle j-1\bullet}
    A_{j-1\bullet j-1\langle})
    -A_{j\langle j-1\bullet}(G^{(j-1)}_{j-1\bullet j-1\langle}
    A_{j-1\langle}^{\langle}+G^{(j-1)}_{j-1\bullet}
    A_{j-1\bullet j-1\langle}).
  \end{array}
  \end{displaymath}

  Let $\boldsymbol{\lambda}_{J\rangle}^D\gets0$ and
  $\boldsymbol{\lambda}_{J\rangle}^N\gets0$.  Compute successively for
  $j=J-1,..,1$,
  \begin{displaymath}
    \small
  \arraycolsep=.2em
  \begin{array}{rcl}
    \boldsymbol{\lambda}_{j\rangle}^D&\gets&S_{j\rangle}\boldsymbol{\lambda}_{j+1\rangle}^N
    -D_{j\rangle}\boldsymbol{\lambda}_{j+1\rangle}^D+\mathbf{v}_{j+1]}^0,\\
    \boldsymbol{\lambda}_{j\rangle}^N&\gets&D^*_{j\rangle}\boldsymbol{\lambda}_{j+1\rangle}^N
    -N_{j\rangle}\boldsymbol{\lambda}_{j+1\rangle}^D-(A_{j\rangle}^{\langle}
   \mathbf{v}_{j+1]}^0+A_{j\rangle j+1\bullet}
    \mathbf{v}_{j+1\bullet}^0),
  \end{array}
  \end{displaymath}
  where the matrix potentials vanish for $j=J-1$, and are given for
  $j\leq J-2$ by
  \begin{displaymath}
    \small
    \medmuskip=-2mu
    \thinmuskip=-2mu
    \thickmuskip=-2mu
    \nulldelimiterspace=0pt
    \scriptspace=-0.5pt
    \arraycolsep=.1em
    \begin{array}{rcl}
      S_{j\rangle}&:=&G^{(j+1)}_{j\rangle j+1\rangle},\mbox{ }
      D_{j\rangle}:=G^{(j+1)}_{j\rangle j+1\rangle}A_{j+1\rangle}^{\rangle}+
    G^{(j+1)}_{j\rangle j+1\bullet}A_{j+1\bullet j+1\rangle},\mbox{ }
    D^*_{j\rangle}:=-A_{j\rangle}^{\langle}G^{(j+1)}_{j\rangle j+1\rangle}-
    A_{j\rangle j+1\bullet}G^{(j+1)}_{j+1\bullet j+1\rangle},\\
    N_{j\rangle}&:=&-A_{j\rangle}^\langle(G^{(j+1)}_{j\rangle j+1\rangle}
    A_{j+1\rangle}^{\rangle}+G^{(j+1)}_{j\rangle j+1\bullet}
    A_{j+1\bullet j+1\rangle})-A_{j\rangle j+1\bullet}(G^{(j+1)}_{j+1\bullet j+1\rangle}
    A_{j+1\rangle}^{\rangle}+G^{(j+1)}_{j+1\bullet}A_{j+1\bullet j+1\rangle}).
  \end{array}
  \end{displaymath}

  Recover the subdomain solutions independently for $j=1,..,J$,
  \begin{displaymath}
    \small
    \mathbf{v}_{j\bullet}\gets G^{(j)}_{j\bullet j\langle}
    \boldsymbol{\lambda}_{j\langle}^N
    -(G^{(j)}_{j\bullet j\langle}A_{j\langle}^{\langle}+ G^{(j)}_{j\bullet}
    A_{j\bullet j\langle})\boldsymbol{\lambda}_{j\langle}^D
     + G^{(j)}_{j\bullet j\rangle}\boldsymbol{\lambda}_{j\rangle}^N
    -(G^{(j)}_{j\bullet j\rangle}A_{j\rangle}^{\rangle}+
    G^{(j)}_{j\bullet}A_{j\bullet j\rangle})\boldsymbol{\lambda}_{j\rangle}^D
    +\mathbf{v}_{j\bullet}^0.
  \end{displaymath}

  Output $\mathbf{\tilde{u}}\gets\sum_{j=1}^{J}R_j^T\left(I_j^{j\langle}\boldsymbol
    {\lambda}_{j\langle}^D+I_j^{j\bullet}\mathbf{v}_{j\bullet}\right)$.
\end{algorithm}

\begin{theorem}\label{thm:polarmat}
  Suppose the subproblems for the $\mathbf{v}_{j}^0$ in
  \Cref{alg:polarmat} are well-posed.  Let
  $\{\mathbf{u}_j^{(\frac{1}{2})}\}_{j=1}^{J-1}$ and
  $\{\mathbf{u}_j^{(1)}\}_{j=1}^J$ be generated by \Cref{alg:DOSMMAT}
  with zero initial guess, the $Q$ be equal to the identity and the
  $P$ be equal to the PML--DtN operators.  Let
  $\mathbf{u}_J^{(\frac{1}{2})}:=0$.  We have for \Cref{alg:polarmat}
  $\boldsymbol{\lambda}_{j\langle}^D=\mathbf{u}_{j-1[}^{(\frac{1}{2})}$,
    $\boldsymbol{\lambda}_{j\langle}^N=-A_{j\langle}^{\rangle}
    \mathbf{u}_{j-1[}^{(\frac{1}{2})}-A_{j\langle j-1\bullet}
      \mathbf{u}_{j-1\bullet}^{(\frac{1}{2})}$ and
      $\boldsymbol{\lambda}_{j\rangle}^D=\mathbf{u}_{j+1]}^{(1)}-
    \mathbf{u}_{j+1]}^{(\frac{1}{2})}+\mathbf{v}_{j+1]}^0$,
  $\boldsymbol{\lambda}_{j\rangle}^N=-A_{j\rangle}^{\langle}
  (\mathbf{u}_{j+1]}^{(1)}-\mathbf{u}_{j+1]}^{(\frac{1}{2})}+\mathbf{v}_{j+1]}^0)
    - A_{j\rangle j+1\bullet}(\mathbf{u}_{j+1\bullet}^{(1)}-
    \mathbf{u}_{j+1\bullet}^{(\frac{1}{2})}+\mathbf{v}_{j+1\bullet}^0)$.
    Therefore, we have
    $\mathbf{v}_{j\bullet}=\mathbf{u}_{j\bullet}^{(1)}$.
\end{theorem}

\begin{proof}
  From the algorithms, we have $\mathbf{u}_1^{(\frac{1}{2})}=\mathbf{v}_{1}^0$
  and so
  \begin{equation}\label{eq:dn0}
    \boldsymbol{\lambda}_{2\langle}^D=\mathbf{u}_{1[}^{(\frac{1}{2})},\quad
    \boldsymbol{\lambda}_{2\langle}^N=-(A_{2\langle}^{\rangle}
    \mathbf{u}_{1[}^{(\frac{1}{2})}+
    A_{2\langle 1\bullet}\mathbf{u}_{1\bullet}^{(\frac{1}{2})}).
  \end{equation}
  Since $G^{(j)}$ is the inverse of the coefficient matrix of the $j$-th
  subproblem of \Cref{alg:DOSMMAT}, the solution can be represented as
  \begin{equation}\label{eq:ujGinv}
    \arraycolsep0.2em
    \left[
      \begin{array}{c}
        \mathbf{u}_{j\bullet}^{(\frac{1}{2})}\\
        \mathbf{u}_{j\rangle}^{(\frac{1}{2})}
      \end{array}
    \right]=
    \left[
      \begin{array}{ccc}
        G^{(j)}_{j\bullet j\langle} & G^{(j)}_{j\bullet} & G^{(j)}_{j\bullet j\rangle}\\
        G^{(j)}_{j\rangle j\langle} & G^{(j)}_{j\rangle j\bullet} & G^{(j)}_{j\rangle}
      \end{array}
    \right]
    \left[
      \begin{array}{c}
        \mathbf{f}_{j\langle}+(\tilde{S}_{j\langle}^{\langle}-A_{j\langle})
        \mathbf{u}_{j-1[}^{(\frac{1}{2})}-A_{j\langle j-1\bullet}
        \mathbf{u}_{j-1\bullet}^{(\frac{1}{2})}
        \\
        \mathbf{f}_{j\bullet}\vphantom{^\rangle_\langle}\\
        \mathbf{f}_{j\rangle}\vphantom{^\rangle_\langle}
      \end{array}
    \right].
  \end{equation}
  We claim that the following identity holds:
  \begin{equation}\label{eq:ujrepform}
    \arraycolsep=.2em
    \begin{array}{rcl}
    \left[
      \begin{array}{c}
        \mathbf{u}_{j\bullet}^{(\frac{1}{2})}\\
        \mathbf{u}_{j\rangle}^{(\frac{1}{2})}
      \end{array}
    \right]&\hspace{-0.5em}=\hspace{-0.5em}&
    \left[
      \begin{array}{ccc}
       G^{(j)}_{j\bullet j\langle}  & G^{(j)}_{j\bullet} & G^{(j)}_{j\bullet j\rangle}\\
       G^{(j)}_{j\rangle j\langle} & G^{(j)}_{j\rangle j\bullet} & G^{(j)}_{j\rangle}
      \end{array}
    \right]\hspace{-0.3em}
    \left[
      \begin{array}{c}
        \mathbf{f}_{j\langle}\\
        \mathbf{f}_{j\bullet}\\
        \mathbf{f}_{j\rangle}
      \end{array}
    \right]\hspace{-0.2em}-\hspace{-0.2em}
    \left[
      \begin{array}{c}
        G^{(j)}_{j\bullet j\langle}\\
        G^{(j)}_{j\rangle j\langle}
      \end{array}
    \right](A_{j\langle}^{\rangle}\mathbf{u}_{j-1[}^{(\frac{1}{2})}
    +A_{j\langle j-1\bullet}\mathbf{u}_{j-1\bullet}^{(\frac{1}{2})})\\
    & & - \left\{
    \left[
      \begin{array}{c}
        G^{(j)}_{j\bullet j\langle}\\
        G^{(j)}_{j\rangle j\langle}
      \end{array}
    \right]A_{j\langle}^{\langle}+
    \left[
      \begin{array}{cc}
        G^{(j)}_{j\bullet} & G^{(j)}_{j\bullet j\rangle}\\
        G^{(j)}_{j\rangle j\bullet} & G^{(j)}_{j\rangle}
      \end{array}
    \right]
    \left[
      \begin{array}{c}
        A_{j\bullet j\langle}\vphantom{^{(j)}_{j\bullet j\rangle}}\\
        0\vphantom{^{(j)}_{j\bullet j\rangle}}
      \end{array}
    \right]\right\}\mathbf{u}_{j-1[}^{(\frac{1}{2})}.
    \end{array}
  \end{equation}
  In fact, the difference of the r.h.s. between \Cref{eq:ujGinv} and
  \Cref{eq:ujrepform} is
  \begin{displaymath}
  \left\{
  \left[
    \begin{array}{cc}
      G^{(j)}_{j\bullet j\langle}   & G^{(j)}_{j\bullet}\\
      G^{(j)}_{j\rangle j\langle} & G^{(j)}_{j\rangle j\bullet}
    \end{array}
  \right]
  \left[
    \begin{array}{c}
      \tilde{S}_{j\langle}^{\langle}\\
      A_{j\bullet j\langle}
    \end{array}
  \right]
  \right\}\mathbf{u}_{j-1[}^{(\frac{1}{2})},
  \end{displaymath}
  and the matrix in the braces vanishes because $G^{(j)}$ is the inverse of the
  coefficient matrix of the subproblem for $\mathbf{v}_{j}^0$ of
  \Cref{alg:polarmat}.  Assuming that
  \begin{equation}\label{eq:dnj}
    \boldsymbol{\lambda}^D_{j\langle}=\mathbf{u}_{j-1[}^{(\frac{1}{2})},\quad
    \boldsymbol{\lambda}^N_{j\langle}=-A_{j\langle}^{\rangle}
    \mathbf{u}_{j-1[}^{(\frac{1}{2})}-A_{j\langle j-1\bullet}
    \mathbf{u}_{j-1\bullet}^{(\frac{1}{2})},
  \end{equation}
  we substitute \Cref{eq:dnj} into \Cref{eq:ujrepform}, take the
  Dirichlet and Neumann traces on $\Gamma_{j+1,j}$, and compare the results with
  the updating rules in \Cref{alg:polarmat} to see that
  \Cref{eq:dnj} also holds for $j+1$ replacing $j$.  By induction based
  on \Cref{eq:dn0}, we conclude that \Cref{eq:dnj} holds for all
  $j=2,..,J$. Substituting \Cref{eq:dnj} and $\mathbf{v}_{j}^0$ into
  \Cref{eq:ujrepform} yields (for $j=J$ change the l.h.s. to
  $\mathbf{u}_J^{(1)}$)
  \begin{equation}\label{eq:ujrephalf}
      \mathbf{u}_{j\bullet}^{(\frac{1}{2})}
      =\mathbf{v}_{j\bullet}^0+
      G^{(j)}_{j\bullet j\langle}\boldsymbol{\lambda}_{j\langle}^N
      - \left(
        G^{(j)}_{j\bullet j\langle}A_{j\langle}^{\langle}+G^{(j)}_{j\bullet}A_{j\bullet j\langle}
      \right)\boldsymbol{\lambda}_{j\langle}^D.
  \end{equation}
  In particular, we have $\mathbf{u}_{J\bullet}^{(1)}=\mathbf{v}_{J\bullet}$ with
  $\mathbf{v}_J$ from \Cref{alg:polarmat}.

  In the backward sweep of \Cref{alg:DOSMMAT}, we denote by
  $\mathbf{w}_{j}:=\mathbf{u}_j^{(1)}-\mathbf{u}_j^{(\frac{1}{2})}+\mathbf{v}_{j}^0$,
  $j=J-1,..,1$ and $\mathbf{w}_J:=\mathbf{u}_J^{(1)}$.  By arguments similar
  to the last paragraph, we can show for all $j=J-1,..,1$ that
  \begin{displaymath}
    \boldsymbol{\lambda}_{j\rangle}^D=\mathbf{w}_{j+1]},\quad
    \boldsymbol{\lambda}_{j\rangle}^N=-A_{j\rangle}^{\langle}\mathbf{w}_{j+1]}
    -A_{j\rangle j+1\bullet}\mathbf{w}_{j+1\bullet},
  \end{displaymath}
  and further 
  \begin{equation}\label{eq:wjrep}
    \mathbf{w}_{j\bullet}=\mathbf{v}_{j\bullet}^0+
    G^{(j)}_{j\bullet j\rangle}\boldsymbol{\lambda}_{j\rangle}^N
    - \left(
      G^{(j)}_{j\bullet j\rangle}A_{j\rangle}^{\rangle}+G^{(j)}_{j\bullet}A_{j\bullet j\rangle}
    \right)\boldsymbol{\lambda}_{j\rangle}^D.
  \end{equation}
  Combining \Cref{eq:ujrephalf} and \Cref{eq:wjrep}, we conclude
  that $\mathbf{u}_{j\bullet}^{(1)}=\mathbf{v}_{j\bullet}$.
\end{proof}

\begin{remark}
  In \Cref{alg:polarmat}, the global approximation $\mathbf{\tilde{u}}$ is
  different from $\mathbf{u}^{(1)}$ of \Cref{alg:GDCMAT}.  In
  \Cref{alg:polarmat},
  $R_{j\rangle}\mathbf{\tilde{u}}=\mathbf{u}_{j\rangle}^{(\frac{1}{2})}$
  with $\mathbf{u}_j^{(\frac{1}{2})}$ from \Cref{alg:DOSMMAT}; while
  \Cref{alg:GDCMAT} takes
  $R_{j\rangle}\mathbf{u}^{(1)}=\mathbf{u}_{j+1]}^{(1)}$ with
  $\mathbf{u}_{j+1}^{(1)}$ from \Cref{alg:DOSMMAT}.
\end{remark}

\section{Optimal parallel Schwarz methods for arbitrary decompositions}

All the methods we discussed so far are only for a domain decomposition into a
sequence of subdomains, and the information is passed gradually from one
subdomain to its neighbor through the linear adjacency of the decomposition. The
methods converge after one double sweep if the appropriate DtN operators are
used in the transmission conditions between the subdomains. If the subdomain
solves are performed in parallel, then the methods converge in a number of
iterations that equals the number of subdomains, as was first pointed out in
\cite{NRS94}, see also \cite{nataf1995factorization}, and this result was
generalized in \cite{nier1998remarques} to domain decompositions whose
connectivity graph has no cycles.  Whether an optimal Schwarz method exists for
an {\it arbitrary} decomposition had been a question until the method was first
created in \cite{GK1}. The method converges in two iterations and thus the
  iteration matrix is nilpotent of degree two; each iteration exposes parallelism {\it
  between} the subdomains in solving the subproblems and after the first
iteration an all-to-all communication is invoked to transmit the interface data
between every pair of subdomains (even if they are not adjacent).  Note that the
communication happens on the whole interfaces of the subdomains, e.g. $\Omega_j$
will map the data on the entire of $\partial\Omega_j\cap\Omega$ to the data on
the entire of $\partial\Omega_l\cap\Omega$ and send them to $\Omega_l$, see
\Cref{fig:gcomm}.
\begin{figure}
  \centering
  \includegraphics[scale=.6]{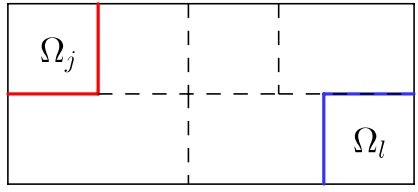}
  \caption{Two subdomains in an arbitrary decomposition.}
  \label{fig:gcomm}
\end{figure}

The optimal algorithm in \cite{GK1} was derived at the discrete level
using linear algebra techniques, and is thus valid for various types
of discretized partial differential equations. We give here an
equivalent formulation at the continuous level, to complete the pair
of discrete and continuous algorithms as we did for all the other
methods in this review. In the optimal algorithm from \cite{GK1},
$\Omega_l$ needs to take into account all the exterior sources as well
as the interior source.  Taking into account the interior source
requires to put a transparent boundary condition on
$\partial\Omega_l\cap\Omega$, while the exterior sources have to be
taken into account with the data $\lambda_{l,}:=\sum_{j\neq
  l}\lambda_{l,j}$ in the transmission condition on
$\partial\Omega_l\cap\Omega$.  Here, $\lambda_{l,j}$ and
$\lambda_{j,j}$ are certain types of traces on
$\partial\Omega_l\cap\Omega$ and $\partial\Omega_j\cap\Omega$ of the
wave field stimulated by the source in $\Omega_j$.  The map that turns
$\lambda_{j,j}$ into $\lambda_{l,j}$ is
\begin{equation}\label{eq:gmap}
  \begin{array}{r@{\hspace{0.2em}}c@{\hspace{0.2em}}ll}
    \mathcal{F}_{l,j}:\;\lambda_{j,j}\rightarrow \lambda_{l,j}=
    (\mathcal{B}_lv)|_{\partial\Omega_l\cap\Omega},
    \mbox{ s.t.}\qquad
    \mathcal{L}\,v&=&0\quad & \mbox{in }\,\Omega-\Omega_j,\\
    \mathcal{B}\,v&=&0\quad &\mbox{on }\,
                              \partial\Omega-\partial\Omega_j,\\
    \mathcal{C}_{j}v&=&\lambda_{j,j}\quad &\mbox{on }\,
                                            \partial\Omega_j\cap\Omega,
  \end{array}
\end{equation}
where $\mathcal{C}_j$ is the trace operator corresponding to
$\lambda_{j,j}$, and $\mathcal{B}_l$ is the trace operator
corresponding to $\lambda_{l,j}$.  For convenience, we can take
$\mathcal{B}_l:=\mathcal{T}_l$, where $\mathcal{T}_l$ is a transparent
boundary operator for the truncation of $\Omega$ to $\Omega_l$, %
which allows us to simulate the waves generated by $\lambda_{l,}$ and the
interior source together by the subproblem in $\Omega_l$.  We may use some
approximation $\widetilde{\mathcal{F}}_{l,j}$ of the operator
in \Cref{eq:gmap} to define a preconditioner, 
which leads to the algorithm given in \Cref{alg:gosm}.%
\smallskip

\begin{algorithm}
  \caption{Optimized Schwarz preconditioner using {\bf global transmission}
    conditions at the {\bf PDE} level}
  \label{alg:gosm}
  Input the source terms $f$ and $g$.

  Suppose the decomposition is arbitrary such that
  $\cup_{j=1}^{J}\overline{\Omega}_j=\overline{\Omega}$.

  Solve the following subproblems independently for $j=1,..,J$,
  \begin{displaymath}
    \small
  \begin{array}{r@{\hspace{0.2em}}c@{\hspace{0.2em}}ll}
    \mathcal{L}\,v_j^{(\frac{1}{2})}&=&f & \mbox{ in }\Omega_j,\\
    \mathcal{B}\,v_j^{(\frac{1}{2})}&=&g & \mbox{ on }\partial\Omega\cap\partial\Omega_j,\\
    \mathcal{B}_jv_j^{(\frac{1}{2})}&=&0 & \mbox{ on }\partial\Omega_j-\partial\Omega,
  \end{array}
  \end{displaymath}
  where $\mathcal{B}_j$ is an approximation of a transparent boundary operator for
  truncation of $\Omega$ to $\Omega_j$.  Take the trace
  $\lambda_{j,j}\gets\mathcal{C}_jv_j^{(\frac{1}{2})}$ on
  $\partial\Omega_j-\partial\Omega$ and map it to
  $\lambda_{l,j}\gets\widetilde{\mathcal{F}}_{l,j}\lambda_{j,j}$ on
  $\partial\Omega_l-\partial\Omega$ for all $l\neq j$.  Here,
  $\widetilde{\mathcal{F}}_{l,j}$ is an approximation of $\mathcal{F}_{l,j}$
  in \Cref{eq:gmap}.

  Solve the following subproblems independently for $j=1,..,J$,
  \begin{displaymath}
    \small
  \begin{array}{r@{\hspace{0.2em}}c@{\hspace{0.2em}}ll}
    \mathcal{L}\,v_j^{(1)}&=&f & \mbox{ in }\Omega_j,\\
    \mathcal{B}\,v_j^{(1)}&=&g & \mbox{ on }\partial\Omega\cap\partial\Omega_j,\\
    \mathcal{B}_jv_j^{(1)}&=&\sum_{l\neq j}\lambda_{j,l}
                               & \mbox{ on }\partial\Omega_j-\partial\Omega.
  \end{array}
  \end{displaymath}

  Output $\tilde{u}\gets\sum_{j=1}^J\mathcal{E}_j(\phi_jv_j^{(1)})$ with
  $\mathcal{E}_j$ the extension by zero to $\Omega$, and
  $\sum_{j=1}^J\mathcal{E}_j\phi_j=1$.
\end{algorithm}

\begin{theorem}\label{thm:gcon}
  If in \Cref{alg:gosm} $\widetilde{\mathcal{F}}_{l,j}=\mathcal{F}_{l,j}$ is
  uniquely defined as in \Cref{eq:gmap} and $\mathcal{B}_j$ is an exact
  transparent boundary operator, then the preconditioner given by
  \Cref{alg:gosm} is exact i.e. the output $\tilde{u}$ is the solution of
  \Cref{eq:pde}. This means the iteration operator is nilpotent of degree
    two.
\end{theorem}

\begin{proof}
  The proof is straightforward by well-posedness and linearity.
\end{proof}

As seen from \Cref{thm:gcon}, approximating \Cref{eq:gmap} is crucial for
\Cref{alg:gosm}.  Essentially, this consists in approximating the
off-diagonal part of the Green's function corresponding to the two interfaces.
But at the time of writing this paper, no effort has been made toward a
practical realization of \Cref{alg:gosm}.

\section{Numerical Experiments}

The main goal of our manuscript is theoretical and formal, namely
  to show that there is a common principle behind the new
  Helmholtz preconditioners based on sequential domain
  decomposition. Numerically, impressive results have been shown for
these algorithms in the literature, see for example
\cite{EY2,Poulson,Chen13b,Stolk,Kimsweep,ZD}.  Nevertheless, it
is interesting and fair for the readers of this review
to also see when these new algorithms get into difficulty. This
motivated us to add this section. A part of the results here have
been submitted to the proceedings of the {\it 24th International
  Conference on Domain Decomposition Methods} held in Svalbard,
Norway.

We consider the Helmholtz equation on the unit square
\begin{equation}\label{HelmholtzVar}
(\Delta+k(x)^2)u=f,\quad \mbox{in $\Omega:=(0,1)^2$},
\end{equation}
with suitable boundary conditions for well-posedness. We discretize
\Cref{HelmholtzVar} by the classical five-point finite difference
method. We split the square sequentially in the $x$ direction
into $p=4,8,16$ equal strips representing the subdomains with
vertical interfaces. Each subdomain has its own constant
wavenumber. For the case of four subdomains, we use the wavenumbers
\begin{equation}\label{LowerWaveNumber}
  k=[20\ 20\ 20\ 20]+\alpha[0\ 20\ 10\ -10],
\end{equation}
where $\alpha$ is a contrast parameter, and for larger $p$ we just
repeat this structure. The mesh resolution we choose guarantees
at least ten points per wavelength for this experiment.  We start with
the case of a wave guide in the $x$ direction, where we use Robin or
PMLs radiation conditions on the left and right, and homogeneous
Dirichlet conditions on top and bottom. We show in \Cref{GuideSolFig}
the real-part of the solution\footnote{The Dirichlet boundary points
  are not plotted, but the PMLs are plotted.} we obtain in the
four-layered medium with $\alpha=1$ stimulated by a point source at
$x=2h$, $y=\frac{1-h}{2}$ in the top row, and below for the point
source at $x=\frac{1}{2}$, $y=\frac{1-h}{2}$.
\begin{figure}
  \centering
  \mbox{\includegraphics[width=0.44\textwidth,clip]{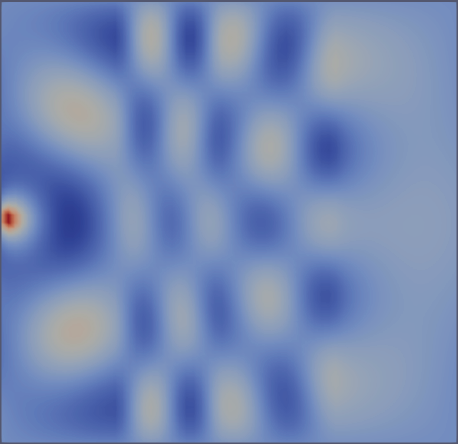}\qquad
  \includegraphics[width=0.505\textwidth,clip]{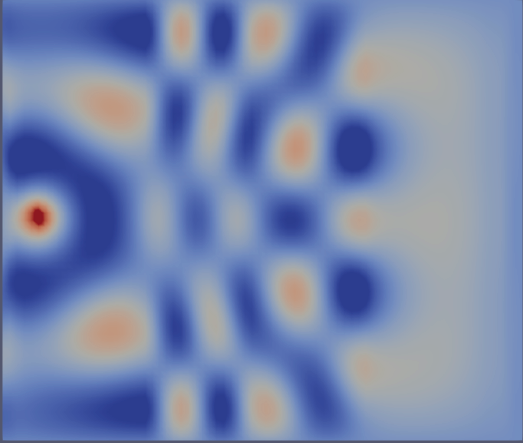}}\\
  \mbox{\includegraphics[width=0.44\textwidth,clip]{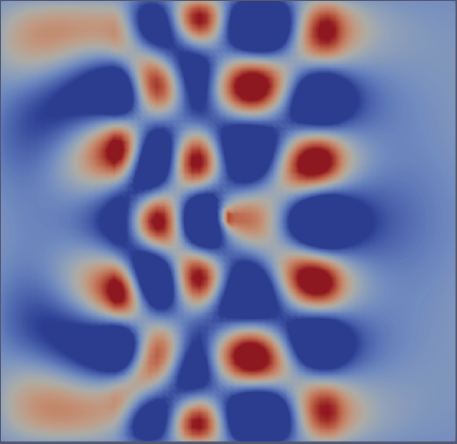}\qquad
  \includegraphics[width=0.505\textwidth,clip]{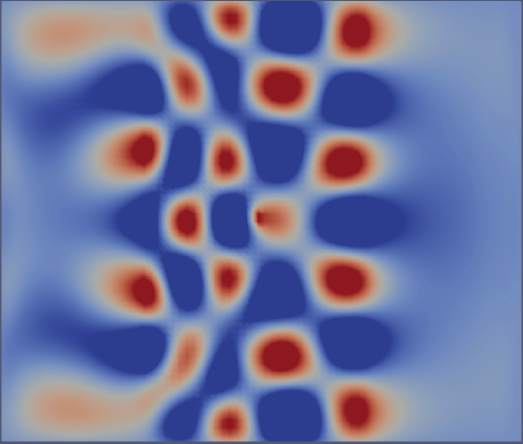}}
\caption{Guided waves from a point source near the left boundary
  (first row) or near the center (second row), with homogeneous
  Dirichlet conditions on the top and bottom and Robin (first column)
  or PMLs (second column) on the left and right boundaries. The media
  have four layers separated by vertical interfaces as in
  \Cref{LowerWaveNumber} with $\alpha=1$.}
  \label{GuideSolFig}
\end{figure}

We are going to test the two fundamental algorithms mentioned already
in \Cref{SecUnderlying} -- one based on the factorization and the other
based on the non-overlapping Schwarz method.  As stated in
\Cref{thm:lu}, the factorization method has (approximate) DtN or Schur
complement derived transmission conditions on the {\it left}
interfaces of subdomains, and Dirichlet conditions on the right.  For
the non-overlapping DOSM (see \Cref{alg:GDCPDE} and
\Cref{alg:GDCMAT}), we use the version with (approximate) DtN
derived transmission conditions on {\it both} interfaces of each
subdomain.  For the case of a constant medium, i.e. $\alpha=0$ in
\Cref{LowerWaveNumber}, we use the exact DtN for \Cref{HelmholtzVar}.
That is, we calculate the exact Schur complement for the discretized
problem.  For $\alpha>0$, we approximate the DtN for
\Cref{HelmholtzVar} by the exact DtN for the Helmholtz equation in
a modified medium: when calculating the exterior DtN on
the {\it left} interface of a subdomain, all the medium to the {\it
  left} of the subdomain is assumed to be the identity extension of
the medium in the {\it left} neighborhood of the
subdomain\footnote{This is a common assumption that most of the
  ABC and
  PML techniques are based on.}.  On the right interface of a subdomain
we do the same. Of course, this approximation has errors
for the heterogeneous media, and we will see that these errors 
have a dramatic impact on the convergence of the algorithms, something which is not yet well documented in the literature on these methods.

We test the two algorithms both as iterative solvers and as
preconditioners for GMRES for varying contrast parameter $\alpha$ and
subdomain numbers.  We do this both for mesh size $h=1/64$ and the
contrast profile in \Cref{LowerWaveNumber}, and on a refined mesh with
$h=1/128$, but also a profile with twice the size for the wavenumber,
i.e.
\begin{equation}\label{HigherWaveNumber}
  k=[40\ 40\ 40\ 40]+\alpha[0\ 40\ 20\ -20],
\end{equation}
so that we still have at least ten points per wavelength
resolution. We show in \Cref{Tab1} and \Cref{Tab2} the number of
iterations the methods took, where we stopped the iterative version of
the algorithms and GMRES when the residual was reduced by $1e-6$, and
we started with a zero initial guess of the solution $u$ for a random
source term $f$ in the physical domain $\Omega$ of
\Cref{HelmholtzVar}. In the PMLs outside $\Omega$, both the initial
guess and the source term are set to zero.
\begin{table}
  \centering
  \caption{LU iteration numbers in the wave guide setting}
  \label{Tab1}
  \setlength{\tabcolsep}{4pt}\vskip-0.5em
  \begin{tabular}{|c|ccc|ccc|ccc|ccc|ccc|ccc|}\hline
          & \multicolumn{6}{c|}{$p=4$} &\multicolumn{6}{c|}{$p=8$} &\multicolumn{6}{c|}{$p=16$} \\\hline
 $\alpha$ & \multicolumn{3}{c|}{Iterative} & \multicolumn{3}{c|}{GMRES}
 & \multicolumn{3}{c|}{Iterative} & \multicolumn{3}{c|}{GMRES}
 & \multicolumn{3}{c|}{Iterative} & \multicolumn{3}{c|}{GMRES}\\\hline
 $0$     & 1& 1& 1& 1& 1& 1& 1& 1& 1& 1& 1& 1& 1& 1& 1& 1& 1& 1\\
 $0.001$ & 4& 3& 3& 3& 3& 3& 5& 3& 3& 4& 3& 3& 6& 3& 3& 4& 3& 3\\
 $0.005$ & 6& 4& 4& 5& 3& 4&12& 5& 5& 7& 4& 4&13& 5& 4& 8& 5& 5\\
 $0.01$  & 8& 5& 4& 5& 4& 4&16& 6& 5& 8& 5& 5&38& 7& 7&11& 6& 6\\
 $0.05$  & -& 8& 6& 8& 6& 5& -&17&12&16& 7& 8& -&12&26&22& 9&10\\
 $0.1$   &32&10&11&10& 7& 6& -& -& -&18&11&11& -& -& -&26&14&15\\
 $1$     & -& -& -&20&19&19& -& -& -&45&38&38& -& -& -&86&63&62\\\hline\hline
 $0$     & 1& 1& 1& 1& 1& 1& 1&1&1& 1& 1& 1& 1& 1& 1& 1& 1& 1\\
 $0.001$ & 4& 2& 3& 3& 3& 3& 5&3&3& 4& 3& 3& 7& 4& 4& 6& 4& 4\\
 $0.005$ & 7& 3& 3& 5& 3& 3& -&5&4& -& 4& 4&30& 6& 6&12& 6& 5\\
 $0.01$  &11& 4& 4& 6& 4& 4& -&6&5&11& 5& 5& -&10&10&19& 7& 6\\
 $0.05$  & -& -& -&13& 7& 6& -&-&-&23&12&11& -& -& -&47&17&16\\
 $0.1$   & -&22&31&14& 9& 9& -&-&-&23&12&11& -& -& -&50&24&23\\
 $1$     & -& -& -&36&21&19& - &-&-&70&64&64& -& -& -& -&95&90\\\hline
\end{tabular}
\end{table}
\begin{table}
  \centering
  \caption{Schwarz iteration numbers in the wave guide setting}
  \label{Tab2}
  \setlength{\tabcolsep}{4pt}\vskip-0.5em
  \begin{tabular}{|c|ccc|ccc|ccc|ccc|ccc|ccc|}\hline
    & \multicolumn{6}{c|}{$p=4$} &\multicolumn{6}{c|}{$p=8$} &\multicolumn{6}{c|}{$p=16$} \\\hline
    $\alpha$ & \multicolumn{3}{c|}{Iterative} & \multicolumn{3}{c|}{GMRES}
 & \multicolumn{3}{c|}{Iterative} & \multicolumn{3}{c|}{GMRES}
 & \multicolumn{3}{c|}{Iterative} & \multicolumn{3}{c|}{GMRES}\\\hline
 $0$     & 1& 1&\,1\,\,& 1& 1& 1& 1& 1& 1& 1& 1& 1&1& 1& 1& 1& 1& 1\\
 $0.001$ & 3& 2&2& 3& 2& 2& 3& 2& 2& 3& 2& 2&3& 2& 2& 3& 2& 2\\
 $0.005$ & 5& 3&3& 4& 3& 3& 5& 3& 3& 4& 3& 3&5& 3& 3& 4& 3& 3\\
 $0.01$  & 7& 4&3& 4& 3& 3& 7& 4& 4& 5& 3& 3&7& 4& 4& 5& 4& 3\\
 $0.05$  &42&12&7& 7& 5& 4& -&16&12& 9& 5& 5&-&21&17&13& 7& 6\\
 $0.1$   & -& -&-& 9& 7& 6& -& -&-&14&12&11&-& -& -&20&17&17\\
 $1$     & -& -&-&26&23&24& -& -&-&48&47&47&-& -& -&59&68&65\\\hline\hline
 $0$     & 1& 1&1& 1& 1& 1& 1&1&1& 1& 1& 1& 1&1&1& 1& 1& 1\\
 $0.001$ & 3& 2&3& 3& 2& 2& 3&2&3& 3& 2& 2& 3&2&3& 3& 2& 2\\
 $0.005$ & 5& 3&4& 4& 3& 3& 6&3&4& 5& 3& 3& 7&3&4& 5& 4& 3\\
 $0.01$  & 8& 4&6& 5& 4& 4&12&4&6& 6& 4& 4&30&5&6& 7& 5& 4\\
 $0.05$  & -& -&-&10& 7& 6& -&-&-&16&11&10& -&-&-&22&16&16\\
 $0.1$   & -& -&-&13&11& 9& -&-&-&19&14&13& -&-&-&32&22&22\\
 $1$     & -& -&-&43&40&38& -&-&-&79&77&77& -&-&-& -& -& -\\\hline
\end{tabular}
\end{table}
The three columns within each 'Iterative' or 'GMRES' column correspond to Robin,
PMLs of thickness five times the mesh size and PMLs of thickness ten times
the mesh size on the left and right of the original domain.  The top
parts are for the smaller wavenumber experiment in \Cref{LowerWaveNumber},
and the bottom parts are for the larger wavenumber experiment in
\Cref{HigherWaveNumber}.  We first see that for $\alpha=0$, i.e. in the constant
wavenumber case, the factorization is exact, both the iterative version and
GMRES converge in one iteration step. As soon as we have however a non-constant
wavenumber, already for $\alpha=0.001$, the factorization is not exact any more.
Nevertheless the algorithms still converge well, up to $\alpha=0.01$ in the
smaller wavenumber case in the top parts of the tables, i.e. a one percent
variation in the wavenumber $k$. For larger contrast, the iterative version of
the algorithms can not be used any more, and GMRES deteriorates now rapidly, for
example if the contrast is at a factor of two, i.e. $\alpha=1$, GMRES iteration
numbers double when $p$ goes from 4 to 8, the two algorithms are not robust any
more. In the higher wavenumber case in the bottom parts of the tables, they
deteriorate even more rapidly for higher contrast. We can also see comparing the
last lines of the top and bottom parts of the tables that doubling the
wavenumber leads to a remarkable growth of the iteration numbers with GMRES as
soon as the contrast is large enough, and GMRES failed to converge in less than
hundred iterations at the bottom right.

We next perform the same set of experiments, but now using Robin or
PML conditions all around the original domain, see \Cref{ScatterSolFig},
\Cref{Tab3} and \Cref{Tab4}.
\begin{figure}
  \centering
  \mbox{\includegraphics[trim=0 -32 0 0,clip,width=0.44\textwidth]{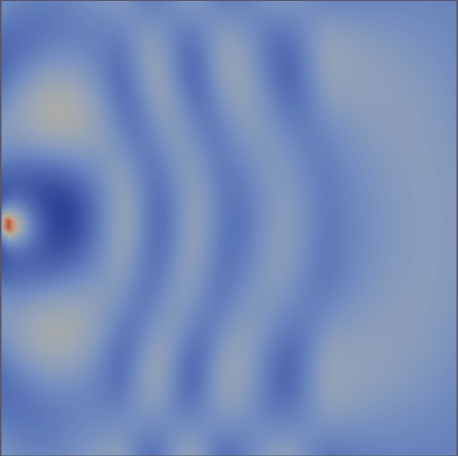}\qquad
  \includegraphics[width=0.505\textwidth,clip]{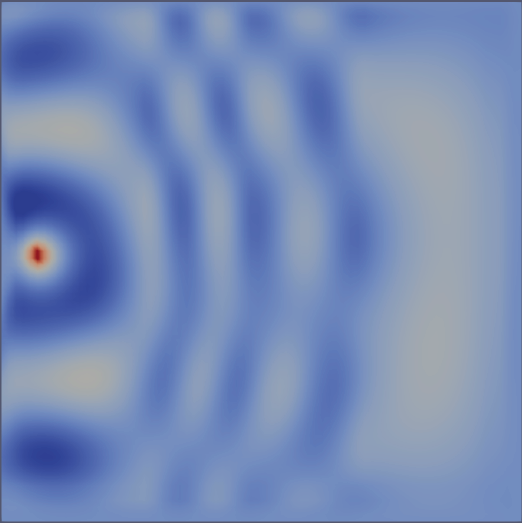}}\\
  \mbox{\includegraphics[trim=0 -32 0 0,clip,width=0.44\textwidth]{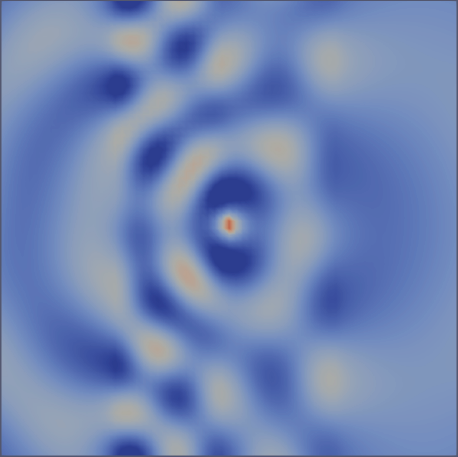}\qquad
  \includegraphics[width=0.505\textwidth,clip]{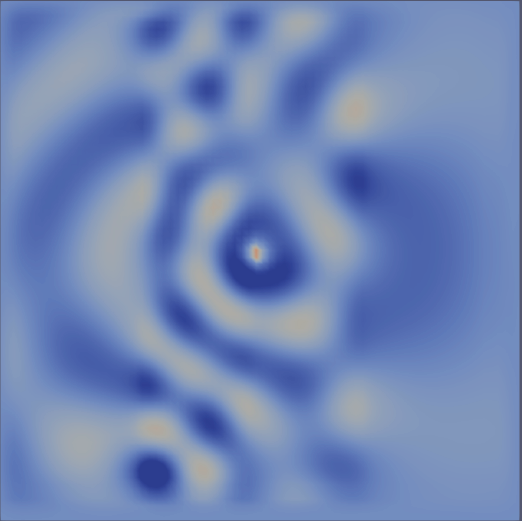}}
\caption{Scattered waves from a point source near the left boundary
  (first row) or near the center (second row), with Robin (first
  column) or PMLs (second column, bigger since the plot includes the PML region) all around the square domain. The
  media have four layers separated by vertical interfaces as in
  \Cref{LowerWaveNumber} with $\alpha=1$.}
  \label{ScatterSolFig}
\end{figure}
\begin{table}
  \centering
  \caption{LU iteration numbers for a domain with Robin or PMLs all around}
  \label{Tab3}
  \setlength{\tabcolsep}{4pt}\vskip-0.5em
    \begin{tabular}{|c|ccc|ccc|ccc|ccc|ccc|ccc|}\hline
    & \multicolumn{6}{c|}{$p=4$} &\multicolumn{6}{c|}{$p=8$} &\multicolumn{6}{c|}{$p=16$} \\\hline
    $\alpha$ & \multicolumn{3}{c|}{Iterative} & \multicolumn{3}{c|}{GMRES}
 & \multicolumn{3}{c|}{Iterative} & \multicolumn{3}{c|}{GMRES}
 & \multicolumn{3}{c|}{Iterative} & \multicolumn{3}{c|}{GMRES}\\\hline
 $0$     & 1& 1& 1& 1& 1& 1& 1&1&1& 1& 1& 1& 1&1&1& 1& 1& 1\\
 $0.001$ & 2& 2& 2& 2& 2& 2& 3&3&2& 3& 2& 2& 3&3&3& 3& 3& 3\\
 $0.005$ & 3& 3& 3& 3& 3& 3& 4&3&3& 4& 3& 3& 4&3&3& 4& 3& 3\\
 $0.01$  & 3& 3& 3& 3& 3& 3& 4&4&3& 4& 3& 3& 5&4&4& 5& 4& 3\\
 $0.05$  & 5& 5& 5& 5& 4& 4& 7&5&5& 6& 5& 5& 8&6&5& 8& 5& 5\\
 $0.1$   & 7& 6& 5& 6& 5& 5& 9&6&6& 8& 6& 6&12&7&7&10& 7& 6\\
 $1$     & -&31&27&15&12&12& -&-&-&32&25&22& -&-&-&58&34&29\\\hline\hline
 $0$     &1& 1& 1& 1& 1& 1& 1&1&1& 1& 1& 1& 1&1&1& 1& 1& 1\\
 $0.001$ &3& 2& 2& 3& 2& 2& 3&2&2& 3& 2& 2& 3&3&3& 3& 3& 3\\
 $0.005$ &3& 3& 3& 3& 3& 3& 4&3&3& 4& 3& 3& 5&3&3& 5& 3& 3\\
 $0.01$  &4& 3& 3& 4& 3& 3& 5&3&3& 5& 3& 3& 6&4&4& 6& 4& 4\\
 $0.05$  &6& 4& 4& 6& 4& 4&10&5&5& 8& 5& 5&14&6&6&11& 6& 6\\
 $0.1$   &8& 5& 5& 7& 5& 5&13&6&6&10& 6& 6&14&8&7&12& 8& 7\\
 $1$     &-& -&52&23&13&12& -&-&-&39&24&22& -&-&-&99&61&48\\\hline
\end{tabular}
\end{table}
\begin{table}  
  \centering
  \caption{Schwarz iteration numbers for a domain with Robin or PMLs all around}
  \label{Tab4}
  \setlength{\tabcolsep}{4pt}\vskip-0.5em
  \begin{tabular}{|c|ccc|ccc|ccc|ccc|ccc|ccc|}\hline
    & \multicolumn{6}{c|}{$p=4$} &\multicolumn{6}{c|}{$p=8$} &\multicolumn{6}{c|}{$p=16$} \\\hline
    $\alpha$ & \multicolumn{3}{c|}{Iterative} & \multicolumn{3}{c|}{GMRES}
 & \multicolumn{3}{c|}{Iterative} & \multicolumn{3}{c|}{GMRES}
 & \multicolumn{3}{c|}{Iterative} & \multicolumn{3}{c|}{GMRES}\\\hline
 $0$     &1&\,1\,\,&\,1\,\,& 1& 1& 1&\,\,1\,\,&1&1& 1& 1& 1&\,1\,\,&1&1& 1& 1& 1\\
 $0.001$ &3&2&2& 2& 2& 2&2&2&2& 2& 2& 2&2&2&2& 2& 2& 2\\
 $0.005$ &3&3&3& 3& 3& 3&3&3&3& 3& 3& 3&3&3&3& 3& 3& 3\\
 $0.01$  &3&3&3& 3& 3& 3&3&3&3& 3& 3& 3&3&3&3& 3& 3& 3\\
 $0.05$  &5&4&4& 4& 4& 4&5&5&4& 4& 4& 4&5&5&4& 4& 4& 4\\
 $0.1$   &6&5&5& 5& 5& 4&6&6&5& 5& 5& 5&7&6&5& 6& 5& 5\\
 $1$     &-&-&-&23&33&37&-&-&-&35&44&44&-&-&-&41&43&48\\\hline\hline
 $0$     &1&1&1& 1& 1& 1&1&1&1& 1& 1& 1&1&1&1& 1& 1& 1\\
 $0.001$ &3&2&2& 2& 2& 2&3&3&2& 2& 2& 2&3&3&2& 2& 2& 2\\
 $0.005$ &3&3&3& 3& 3& 3&3&3&3& 3& 3& 3&3&3&3& 3& 3& 3\\
 $0.01$  &4&3&3& 3& 3& 3&3&3&3& 3& 3& 3&3&3&3& 3& 3& 3\\
 $0.05$  &5&5&5& 5& 4& 4&5&5&5& 5& 4& 4&5&5&5& 5& 5& 4\\
 $0.1$   &7&7&6& 6& 6& 5&9&7&6& 7& 6& 5&9&7&7& 7& 6& 6\\
 $1$     &-&-&-&32&43&45&-&-&-&49&54&61&-&-&-&79&94&75\\\hline
\end{tabular}
\end{table}
We see that the outer radiation conditions are better than the wave
guide setting for the two algorithms, they work now in the iterative
version up to about a 10 percent variation of the wavenumber in this
specific experiment. As soon as however there is a variation as large
as a factor of two, the algorithms are not effective solvers any more,
the iterative versions diverge, and GMRES iteration numbers
deteriorate when the number of subdomains increases, and also when the
wavenumber is doubled. One thus has to be careful when claiming
  optimality of algorithms in this class of Helmholtz
  preconditioners.

\section{Conclusions}

We have seen that for a large class of new Helmholtz solvers the
underlying mathematical technique is the same: the solvers are based
on a nilpotent iteration given by a double sweep in a sequential
domain decomposition that uses the exact Dirichlet-to-Neumann
operators for transmission conditions at the interfaces. At the linear
algebra level, the corresponding algorithm is based on an exact block
LU factorization of the matrix. From domain decomposition, it is known
that when the solves are performed in parallel, instead of in a
sweeping fashion, the method is still nilpotent, but convergence is
then achieved in a number of iterations corresponding to the number of
subdomains \cite{NRS94,nataf1995factorization}. If the domain
decomposition is more general, such that the connectivity graph
includes cycles and thus cross points between subdomains are present,
we have given an algorithm at the continuous level based on the
discrete algorithm in \cite{GK1} that still is nilpotent. This
algorithm requires communication of every subdomain with every other
one, and convergence is achieved in two iterations. While there is
currently no practical realization of this algorithm, the fact that
the algorithm converges in two iterations, independently of the number
of subdomains, suggests that a coarse space component is active in
this optimal algorithm. Coarse spaces leading to nilpotent iterations
have first been described in the lecture notes
\cite{gander2012methode}, and then in
\cite{gander2014new,gander2014discontinuous}, with successful
approximations in \cite{gander:2016:SHEM,gander:2016:AON}.  The
property of domain decomposition methods in general to be nilpotent
has only very recently been investigated in more detail, see
\cite{gander:2016:ONS}.  None of the Helmholtz solvers we
described in this manuscript is using coarse space techniques
at the time of writing, which lets us expect that this area of
research will remain very active over the coming years.

\bibliographystyle{siamplain}
\bibliography{references}
\end{document}